\documentclass[12pt,a4paper,reqno]{amsart}
\usepackage{amsfonts,amsthm,amsmath,amssymb,
	amsbsy,euscript,textcomp,url,stmaryrd} 

\newtheorem{theor}{Theorem}

\theoremstyle{definition}

\newtheorem{proposition}[theor]{Proposition}
\newtheorem{lemma}[theor]{Lemma}
\newtheorem{cor}[theor]{Corollary}
\newtheorem{define}{Definition}

\newtheorem{problem}{Problem}
\newtheorem{open}[problem]{Open problem}
\newtheorem{property}{Property}
\newtheorem{example}{Example}
\theoremstyle{remark}
\newtheorem{rem}{Remark}

\newcommand{\pinner}{\mathbin{\mathchoice
{\hbox{\vrule width0.6em depth0pt height0.4pt
	\vrule width0.4pt depth0pt height0.8ex}}
{\hbox{\vrule width0.6em depth0pt height0.4pt
	\vrule width0.4pt depth0pt height0.8ex}}
{\hbox{\kern0.14em
	\vrule width0.48em depth0pt height0.4pt
	\vrule width0.4pt depth0pt height0.6ex\kern0.14em}}
{\hbox{\kern0.1em
	\vrule width0.39em depth0pt height0.4pt
	\vrule width0.4pt depth0pt height0.5ex\kern0.1em}}}}

\newcommand{\BBR}{\mathbb{R}}

\newcommand{\bcP}{{\boldsymbol{\mathcal{P}}}}
\newcommand{\bcQ}{{\boldsymbol{\mathcal{Q}}}}

\newcommand{\cP}{\mathcal{P}}\newcommand{\cR}{\mathcal{R}}
\newcommand{\cQ}{\mathcal{Q}}

\newcommand{\cX}{{\EuScript X}}    
\newcommand{\cY}{{\EuScript Y}}    
\newcommand{\cZ}{{\EuScript Z}}    

\newcommand{\bx}{{\boldsymbol{x}}}

\newcommand{\veps}{\varepsilon}

\newcommand{\dd}{\partial}
\newcommand{\Id}{{\mathrm d}}

\DeclareMathOperator{\const}{const}

\DeclareMathOperator{\Jac}{Jac}

\newcommand{\lshad}{[\![}
\newcommand{\rshad}{]\!]}

\newcommand{\schouten}[1]{\lshad {#1} \rshad}

\newcommand{\by}[1]{\textit{{#1}}}
\newcommand{\jour}[1]{\textit{{#1}}}
\newcommand{\vol}[1]{\textbf{{#1}}}
\newcommand{\book}[1]{\textrm{{#1}}}

\newcommand{\ground}[1]{\text{\textit{\small #1}}}

\newcommand*{\vcenteredhbox}[1]{\begingroup
\setbox0=\hbox{#1}\parbox{\wd0}{\box0}\endgroup}

\DeclareSymbolFont{extraup}{U}{zavm}{m}{n}
\DeclareMathSymbol{\varheartsuit}{\mathalpha}{extraup}{86}
\DeclareMathSymbol{\vardiamondsuit}{\mathalpha}{extraup}{87}

\title[Universal infinitesimal deformation of Poisson structures]{The Kontsevich tetrahedral flow revisited}
\author[A.~Bouisaghouane]{A.~Bouisaghouane} 
\author[R.~Buring]{R.~Buring}
\author[A.~V.~Kiselev]{A.~
	Kiselev${}^{*,\S}$}
\thanks{\textit{Address}: Johann Ber\-nou\-lli Institute for Mathematics and Computer Science, University of Groningen,
	P.O.~Box 407, 9700~AK Groningen, The Netherlands. 
	\quad${}^{*}$\textit{E-mail}: \texttt{A.V.Kiselev\symbol{"40}rug.nl}
\\
\mbox{ }\quad${}^{*}$ Institut des Hautes $\smash{\text{\'Etudes}}$ Scientifiques, 
35~route de Chart\-res, Bures\/-\/sur\/-\/Yvette, \mbox{F-91440} France.%
\\
\mbox{ }\quad${}^{\S}$ \textit{Present address}: Max Planck Institute for Mathematics,
Vi\-vats\-gas\-se~7, \mbox{D-53111} Bonn, Germany}

\date{24 November 2016, revised 12 March 2017} 



\begin{document}
\begin{abstract}
We prove that the Kontsevich tetrahedral flow $\dot{\mathcal{P}}=\mathcal{Q}_{a:b}(\mathcal{P})$,
the right\/-\/hand side of which is a linear combination of two
differential monomials of degree four in a bi\/-\/vector~$\mathcal{P}$ on an 
affine real Poisson manifold~$N^n$, does infinitesimally preserve the
space of Poisson bi\/-\/vectors on~$N^n$ if and only if the two
monomials in $\mathcal{Q}_{a:b}(\mathcal{P})$ are balanced by the ratio $a:b=1:6$. The proof is explicit; it is written in the language of Kontsevich graphs.
\end{abstract}
\maketitle


\subsection*{Introduction}
\noindent%
The main question which we address in this paper is how Poisson structures can be deformed in such a way that they stay Poisson.
We reveal one such method that works for all Poisson structures on affine real manifolds; the construction 
of that flow on the space of bi\/-\/vectors was proposed in~\cite{Ascona96}: 
the formula is derived from two differently oriented tetrahedral graphs on 
four vertices.
The flow is a linear com\-bi\-na\-ti\-on of two terms, each quartic\/-\/nonlinear in the Poisson structure.
By using several examples of Poisson brackets with high polynomial degree coefficients, 
the first and last authors demonstrated in~\cite{tetra16} that the ratio~$1:6$ is the only possible balance at which the tetrahedral flow can preserve the property of the Cauchy datum to be Poisson.
But does the Kontsevich te\-tra\-he\-d\-ral flow $\dot{\cP}=\cQ_{1:6}(\cP)$ with ratio~$1:6$ actually preserve the space of \emph{all} Poisson bi\/-\/vectors? 

We prove the infinitesimal version of this claim: namely, we show that $\lshad \cP,\cQ_{1:6}(\cP)\rshad=0$ for every bi\/-\/vector~$\cP$ satisfying the master\/-\/equation~$\lshad\cP,\cP\rshad=0$ for Poisson structures.
The proof is graphical: 
to prove that equation \eqref{EqWhichMechanism} holds, we find an operator $\Diamond$, encoded by using the Kontsevich graphs, that solves
%
%
equation~\eqref{EqFactor}. 
%
We also show that there is no universal mechanism (that would involve the language of Kontsevich graphs) for the tetrahedral flow to be trivial in the respective Poisson cohomology.

The text is structured as follows.
In section \ref{SecPrb} we recall 
how oriented graphs can be used to encode differential operators acting on the space 
of multivectors.
In particular, differential polynomials in a given Poisson structure are obtained 
as soon as a copy of that Poisson bi\/-\/vector is placed in every internal vertex of a graph.
Specifically, the right\/-\/hand side $\cQ_{a:b}=a\cdot\Gamma_1+b\cdot\Gamma_2$ 
of the Kontsevich tetrahedral flow $\dot{\cP}=\cQ_{a:b}(\cP)$ on the space of bi\/-\/vectors on an affine Poisson manifold $\bigl(N^{n},\cP\bigr)$ is a linear combination 
of two differential monomials, $\Gamma_1(\cP)$ and~$\Gamma_2(\cP)$,
of degree four in the bi\/-\/vector~$\cP$ that evolves. 

We determine 
at which balance $a:b$ the Kontsevich tetrahedral flow $\dot{\cP}=\cQ_{a:b}(\cP)$ infinitesimally preserves the space of Poisson bi\/-\/vectors, that is, 
the bi\/-\/vector $\cP + \veps\cQ_{a:b}(\cP)+\bar{o}(\veps)$ satisfies the equation 
\begin{equation}\label{Eq1stOrder}
\schouten{\cP + \veps\cQ_{a:b}(\cP)+\bar{o}(\veps),\cP + \veps\cQ_{a:b}(\cP)+\bar{o}(\veps)} \doteq \bar{o}(\veps) \qquad \text{via} \   \schouten{\cP,\cP}=0;
\end{equation}
here we denote by $\schouten{\cdot,\cdot}$ the Schouten bracket 
(see formula~\eqref{EqSchouten} on page~\pageref{EqSchouten}). 
Expanding, we obtain the cocycle condition,
\begin{equation}\label{EqWhichMechanism}
\schouten{\cP, \cQ_{a:b}(\cP)} \doteq 0 \qquad \text{via} \ 
\schouten{\cP,\cP} = 0,
\end{equation}
with respect to the Poisson differential $\boldsymbol{\dd}_{\cP}=\schouten{\cP,\cdot}$. Viewed as an equation with respect to the ratio~$a:b$, condition~\eqref{EqWhichMechanism} is the main object of our study.

Recent counterexamples~\cite{tetra16} 
show that the bi\/-\/vector $\cP+\varepsilon\cQ_{a:b}(\cP)+\bar{o}(\varepsilon)$ can stay 
Poisson \emph{only if} the balance~$a:b$ in~$\cQ_{a:b}$ is equal to~$1:6$. 
We now prove the infinitesimal part of sufficiency:
the deformation $\cP+\varepsilon\cQ_{1:6}(\cP)+\bar{o}(\varepsilon)$ is always infinitesimally Poisson, whence 
the balance $a:b=1:6$ in the Kontsevich tetrahedral flow is universal
for all Poisson bi\/-\/vectors~$\cP$ on all affine manifolds~$N^{n}$.
The proof is explicit: in section~\ref{SecProof} we reveal the mechanism of factorization --\,via the Jacobi identity\,-- in~\eqref{EqWhichMechanism} at $a:b=1:6$. 
On the left\/-\/hand side of 
factorization problem~\eqref{EqWhichMechanism} we expand the 
Poisson differential of the Kontsevich tetrahedral flow at the balance $1 : 6$ into the sum of 39~graphs (see Figure~\ref{FigLHS} on page~\pageref{FigLHS} and Table~\ref{TableLHS} in Appendix~\ref{AppCode}).
On the other side of that factorization, we take the sum that runs 
with undetermined coefficients
over all those fragments of differential consequences of the Jacobi identity $\schouten{\cP,\cP} = 0$ which are known to vanish independently. 
We then find a linear polydifferential operator $\Diamond(\cP,\cdot)$ that acts on the filtered components 
of the Jacobiator $\text{Jac}\,(\cP)\mathrel{{:}{=}} \schouten{\cP,\cP}$ for the bi-vector~$\cP$; the operator~$\Diamond$ provides the factorization $\schouten{\cP,\cQ_{1:6}(\cP)}(f,g,h)=\Diamond\bigl(\cP,\text{Jac}\,(\cP)(\cdot,\cdot,\cdot)\bigr)(f,g,h)$ of the $\boldsymbol{\dd}_{\cP}$-\/cocycle condition, see~\eqref{EqWhichMechanism}, 
through the Jacobi identity $\text{Jac}\,(\cP)=0$. 
To describe the differential operators that produce such 
consequences of the Jacobi identity, we use the pictorial language of graphs: every internal vertex contains a copy of the bi\/-\/vector~$\cP$ and the operators are reduced by using its skew\/-\/sym\-met\-ry. 
There remain $7,025$~graphs, the coefficients of which are linear in the unknowns. We now solve the arising inhomogeneous linear algebraic system. Its solution yields a 
polydifferential operator~$\Diamond$, encoded using Leibniz graphs (see p.\:\pageref{pSolution}), that provides the sought\/-\/for factorization $\schouten{\cP,\cQ_{1:6}} = \Diamond(\cP,\Jac(\cP))$. It is readily seen from formula~\eqref{EqSol} that this 
operator~$\Diamond$ is completely determined by only $8$~nonzero coefficients (out of~$1132$ total).
\footnote{The maximally detailed description of that solution $\Diamond$ is contained in Appendix~\ref{AppCode}.} 
Therefore, although finding an operator~$\Diamond$ was hard, verifying that it does solve the factorization problem has become almost immediate, as we show in the proof of Theorem~\ref{ThMain}. We thereby establish the main result (namely, Corollary~\ref{CorMain} on page~\pageref{CorMain}).
The paper concludes with the 
formulation of an open problem about the integration of tetrahedral flow in~\eqref{Eq1stOrder}
to higher order expansions in~$\veps$, see~\eqref{EqMasterHighOrder} on p.~\pageref{EqMasterHighOrder}.

In Appendix~\ref{AppPerturb} we outline a different method to tackle the factorization problem, namely, by making the Jacobi identity visible in~\eqref{EqWhichMechanism} by perturbing the original structure~$\cP\mapsto \widetilde{\cP}$ 
in such a way that 
$\widetilde{\cP}$ is not Poisson and $\cQ_{1:6}(\widetilde{\cP})\neq 0$. Hence 
$\widetilde{\cP}$ contributes to the right\/-\/hand side of~\eqref{EqWhichMechanism} such that the respectively perturbed bi\/-\/vec\-tor~$\cQ_{1:6}(\widetilde{\cP})$ stops being compatible 
with the perturbed Poisson structure $\widetilde{\cP}$. The first\/-\/order balance of both sides of perturbed equation~\eqref{EqWhichMechanism} then suggests the coefficients of those differential consequences of the Jacobiatior which are actually involved in the factorization mechanism. The coefficients of operators realized by graphs which were found by following this scheme are reproduced in the full run\/-\/through that gave us the solution~$\Diamond$ in section~\ref{SecProof}.

\section{The main problem: From graphs to multivectors}
\label{SecPrb}
	
\subsection{The language of graphs} 
\label{SecGraphOp}
Let us formalise a way to encode polydifferential operators -- in particular multivectors -- using oriented graphs~\cite{MKParisECM,MKZurichICM}.
In an affine real manifold $N^n$ (here $2\leqslant n < \infty$), take a chart $U_\alpha \hookrightarrow \mathbb{R}^n$ and denote the Cartesian coordinates 
by $\bx = (x^1,\ldots,x^n)$.
We now consider only the oriented graphs whose vertices are either tails for an ordered pair of arrows, each decorated with its own index, or sinks (with no issued edges) like the vertices $\ground{1, 2}$ in $(\ground{1}) \xleftarrow{\ i \ } \bullet \xrightarrow{\ j \ } (\ground{2})$. 
The ar\-row\-tail vertices are called \emph{internal}. Every internal vertex $\bullet$ carries a copy of a given Poisson bi-vector $\cP = \cP^{ij}(\bx)\, \partial_i \wedge \partial_j$ with its own pair of indices. For each internal vertex~$\bullet$, the pair of out\/-\/going edges is ordered L~$\prec$~R. The ordering L~$\prec$~R of decorated out\/-\/going edges coincides with the ordering
``first${}\prec{}$se\-c\-ond'' of the indexes in the coefficients of~$\cP$.
Namely, the left edge (L) carries the first index and the other edge (R) carries the second index. By definition, the decorated edge $\bullet \xrightarrow{\ i\ } \bullet$ denotes at once the derivation ${\partial}/{\partial x^{i}}\equiv \partial_{i}$ (that acts on the content of the arrowhead vertex) and the summation $\sum_{i=1}^{n}$ (over the index~$i$ in 
the object which is contained within the ar\-row\-tail vertex).
As it has been explained in~\cite{dq15,KontsevichFormality}, the operator which every graph encodes is equal to the sum (running over all the indexes) of products (running over all the vertices) of those vertices content (differentiated by the in-coming arrows, if any).
Moreover, we let the sinks be ordered (like $\mathit{1}\prec\mathit{2}$ above), so that every such graph defines a polydifferential operator: its arguments are thrown into the respective sinks.

\begin{example}
The \emph{wedge graph} $(\ground{1})
\xleftarrow[L]{\ i\ } \cP^{ij}(\bx) \xrightarrow[R]{\ j\ } (\ground{2})
$ encodes the bi-differential operator $\sum_{i,j=1}^{n}(\ground{1}) \overleftarrow{\partial_{i}} \! \cdot \! \cP^{ij}(\bx) \! \cdot \! \overrightarrow{\partial_{j}}(\ground{2})$.
Such graph
specifies a Poisson bracket (on every chart~$U_\alpha \subseteq N^n$)
if it 
satisfies the Jacobi identity, see~\eqref{EqJacFig} below.
\end{example}

\begin{rem}
In principle, we allow the presence of both the tadpoles and cycles over two vertices (or ``eyes"), see Fig.~\ref{FigTadpoleEye}.
However, in hindsight there will be neither tadpoles nor eyes in the solution which we shall have found in section \ref{SecProof} below.
\begin{figure}[htb]
\unitlength=1mm
\special{em:linewidth 0.4pt}
\linethickness{0.4pt}
\begin{picture}(25.00,15.00)
\put(10.00,10.00){\circle{10.00}}
\put(15.00,10.00){\circle*{1}}
\put(14.85,11.15){\vector(0,-1){1.00}}
\put(15.00,10.00){\vector(1,0){8.00}}
\end{picture}
\quad
\unitlength=1mm
\special{em:linewidth 0.4pt}
\linethickness{0.4pt}
\begin{picture}(40.00,20.00)
\put(10.00,10.00){\circle*{1}}
\put(30.00,10.00){\circle*{1}}
\bezier{112}(10.00,10.00)(20.00,20.00)(30.00,10.00)
\bezier{112}(10.00,10.00)(20.00,0.00)(30.00,10.00)
\bezier{48}(0.00,15.00)(5.00,15.00)(10.00,10.00)
\bezier{48}(30.00,10.00)(35.00,5.00)(40.00,5.00)
\put(28.9,11.1){\vector(1,-1){1.00}}
\put(11.1,8.9){\vector(-1,1){1.00}}
\end{picture}
\caption{A tadpole and an ``eye''.}
\label{FigTadpoleEye}
\end{figure}
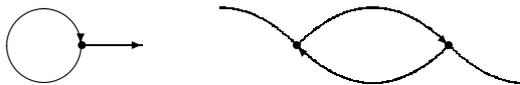
\end{rem}



\begin{rem}\label{RemGamma0}
Under the above assumptions, there exist inhabited graphs that encode zero differential operators.
Namely, consider the graph with a double edge:
\[
\vcenteredhbox{
\unitlength=1mm
\special{em:linewidth 0.4pt}
\linethickness{0.4pt}
\begin{picture}(22.67,20.00)
\put(0.00,10.00){\circle*{1}}
\put(-3.50,9.00){$\boldsymbol{a}$}
\put(20.00,10.00){\circle*{1}}
\put(21.00,9.00){$z$}
\bezier{212}(0.00,10.00)(10.00,20.00)(20.00,10.00)
\put(10.00,16.30){\makebox(0,0)[rb]{\tiny$i$}}
\put(10.80,11.70){\makebox(0,0)[rb]{\tiny$L$}}
\bezier{212}(0.00,10.00)(10.00,0.00)(20.00,10.00)
\put(10.00,1.30){\makebox(0,0)[rb]{\tiny$j$}}
\put(11.00,6.30){\makebox(0,0)[rb]{\tiny$R$}}
\put(18.67,11.33){\vector(1,-1){1}}
\put(18.67,8.67){\vector(1,1){1}}
\end{picture}
}
=
\vcenteredhbox{
\unitlength=1mm
\special{em:linewidth 0.4pt}
\linethickness{0.4pt}
\begin{picture}(22.67,20.00)
\put(0.00,10.00){\circle*{1}}
\put(-3.50,9.00){$\boldsymbol{a}$}
\put(20.00,10.00){\circle*{1}}
\put(21.00,9.00){$z$}
\bezier{212}(0.00,10.00)(10.00,20.00)(20.00,10.00)
\put(10.00,16.00){\makebox(0,0)[rb]{\tiny$j$}}
\put(10.80,11.70){\makebox(0,0)[rb]{\tiny$L$}}
\bezier{212}(0.00,10.00)(10.00,0.00)(20.00,10.00)
\put(10.00,1.60){\makebox(0,0)[rb]{\tiny$i$}}
\put(11.00,6.30){\makebox(0,0)[rb]{\tiny$R$}}
\put(18.67,11.33){\vector(1,-1){1}}
\put(18.67,8.67){\vector(1,1){1}}
\end{picture}
}
=
-\vcenteredhbox{
\unitlength=1mm
\special{em:linewidth 0.4pt}
\linethickness{0.4pt}
\begin{picture}(22.67,20.00)
\put(0.00,10.00){\circle*{1}}
\put(-3.50,9.00){$\boldsymbol{a}$}
\put(20.00,10.00){\circle*{1}}
\put(21.00,9.00){$z$}
\bezier{212}(0.00,10.00)(10.00,20.00)(20.00,10.00)
\put(10.00,16.30){\makebox(0,0)[rb]{\tiny$j$}}
\put(10.80,11.70){\makebox(0,0)[rb]{\tiny$R$}}
\bezier{212}(0.00,10.00)(10.00,0.00)(20.00,10.00)
\put(10.00,1.30){\makebox(0,0)[rb]{\tiny$i$}}
\put(11.00,6.30){\makebox(0,0)[rb]{\tiny$L$}}
\put(18.67,11.33){\vector(1,-1){1}}
\put(18.67,8.67){\vector(1,1){1}}
\end{picture}
}
\stackrel{\text{flip}}{=}
-\vcenteredhbox{
\unitlength=1mm
\special{em:linewidth 0.4pt}
\linethickness{0.4pt}
\begin{picture}(22.67,20.00)
\put(0.00,10.00){\circle*{1}}
\put(-3.50,9.00){$\boldsymbol{a}$}
\put(20.00,10.00){\circle*{1}}
\put(21.00,9.00){$z$.}
\bezier{212}(0.00,10.00)(10.00,20.00)(20.00,10.00)
\put(10.00,16.30){\makebox(0,0)[rb]{\tiny$i$}}
\put(10.80,11.70){\makebox(0,0)[rb]{\tiny$L$}}
\bezier{212}(0.00,10.00)(10.00,0.00)(20.00,10.00)
\put(10.00,1.30){\makebox(0,0)[rb]{\tiny$j$}}
\put(11.00,6.30){\makebox(0,0)[rb]{\tiny$R$}}
\put(18.67,11.33){\vector(1,-1){1}}
\put(18.67,8.67){\vector(1,1){1}}
\end{picture}
}
\]
By first relabelling the summation indices and then swapping L~$\rightleftarrows$~R (and redrawing) we evaluate the operator acting at $z$ to $\sum_{i,j=1}^n a^{ij} \partial_i \partial_j(z)
= -\sum_{i,j=1}^n a^{ij} \partial_i \partial_j (z)$; whence the operator is zero.
In the same way, any graph containing a double edge encodes a zero operator.
Graphs can also encode zero differential operators in a more subtle way.
For example consider the wedge on two wedges:
\begin{equation}\label{EqWedgeOnTwoWedges}
\vcenteredhbox{
\unitlength=1mm
\special{em:linewidth 0.4pt}
\linethickness{0.4pt}
\begin{picture}(31.00,25.67)(0,3)
\put(15.00,25.00){\vector(1,0){10.00}}
\put(25.00,25.00){\vector(-1,-3){5.00}}
\put(25.00,25.00){\vector(1,-3){5.00}}
\put(25.00,15.00){\vector(1,-1){4.67}}
\put(25.00,15.00){\vector(-1,-1){4.67}}
\put(15.00,25.00){\line(1,-1){6.67}}
\put(23,17){\vector(1,-1){2}}
\bezier{16}(23.00,17.00)(24.00,16.00)(25.00,15.00)
\put(15.00,25.00){\circle*{1}}
\put(25.00,25.00){\circle*{1}}
\put(25.00,15.00){\circle*{1}}
\put(30.33,9.67){\circle*{1}}
\put(19.67,9.67){\circle*{1}}
\put(13.33,25.00){\makebox(0,0)[rc]{3}}
\put(26.67,25.00){\makebox(0,0)[lc]{2}}
\put(20.00,25.67){\makebox(0,0)[cb]{\tiny$R$}}
\put(17.67,21.00){\makebox(0,0)[ct]{\tiny$L$}}
\put(25.00,13.33){\makebox(0,0)[ct]{1}}
\put(20.00,8){\makebox(0,0)[ct]{$f$}}
\put(30.00,7.67){\makebox(0,0)[ct]{$g$}}
\end{picture}
\raisebox{14mm}{$=$ \hskip 1mm $0$.}
}
\end{equation}
Swapping the labels $1\rightleftarrows2$ of the lower wedges does not change the operator.
On the other hand, doing the same in a different way, namely, by swapping `left' and `right' in the top wedge introduces a minus sign.
Hence the graph encodes a differential operator equal to minus itself, i.e.~zero.
Proving that a graph which contains the left\/-\/hand side 
of~\eqref{EqWedgeOnTwoWedges} as a subgraph equals zero is an elementary exercise
(cf.\ Example~\ref{Gamma0InSol} on p.~\pageref{Gamma0InSol}).
\end{rem}

Besides the trivial vanishing mechanism in Remark~\ref{RemGamma0}, there is the 
Jacobi identity together with its differential consequences, which will play a key role in what follows.
For any three arguments 
$\mathit{1},\mathit{2},\mathit{3} \in C^\infty(N^n)$, the Jacobi identity $\Jac_{\cP}(\mathit{1},\mathit{2},\mathit{3})=0$ 
is realized\footnote{The notation $\Jac_{\cP}(\ground{1, 2, 3})$ is synonymic to~$\Jac(\cP)(\mathit{1}\otimes\mathit{2}\otimes\mathit{3})$.} by the graph
\begin{equation}\label{EqJacFig}
\vcenteredhbox{
\raisebox{3.3mm}
[6.5mm][3.5mm]{ 
\unitlength=1mm
\special{em:linewidth 0.4pt}
\linethickness{0.4pt}
\begin{picture}(12,15)
\put(0,-10){
\begin{picture}(12.00,15.00)
\put(0.00,10.00){\framebox(12.00,5.00)[cc]{$\bullet\ \bullet$}}
\put(2.00,10.00){\vector(-1,-3){1.33}}
\put(6.00,10.00){\vector(0,-1){4.00}}
\put(10.00,10.00){\vector(1,-3){1.33}}
\put(0.00,4.00){\makebox(0,0)[cb]{\tiny\it1}}
\put(6.00,4.00){\makebox(0,0)[cb]{\tiny\it2}}
\put(11.67,4.00){\makebox(0,0)[cb]{\tiny\it3}}
\end{picture}
}\end{picture}}}\ \ \ 
\mathrel{{:}{=}}
\text{\raisebox{-12pt}[25pt]{
\unitlength=0.70mm
\linethickness{0.4pt}
\begin{picture}(26.00,16.33)
\put(0.00,5.00){\line(1,0){26.00}}
\put(2.00,5.00){\circle*{1.33}}
\put(13.00,5.00){\circle*{1.33}}
\put(24.00,5.00){\circle*{1.33}}
\put(2.00,1.33){\makebox(0,0)[cc]{\tiny\it1}}
\put(13.00,1.33){\makebox(0,0)[cc]{\tiny\it2}}
\put(24.00,1.33){\makebox(0,0)[cc]{\tiny\it3}}
\put(7.33,11.33){\circle*{1.33}}
\put(7.33,11.33){\vector(1,-1){5.5}}
\put(7.33,11.33){\vector(-1,-1){5.5}}
\put(13,17){\circle*{1.33}}
\put(13,17){\vector(1,-1){11.2}}
\put(13,17){\vector(-1,-1){5.1}}
\put(3.00,10.00){\makebox(0,0)[cc]{\tiny$i$}}
\put(12.00,10.00){\makebox(0,0)[cc]{\tiny$j$}}
\put(24.00,10.00){\makebox(0,0)[cc]{\tiny$k$}}
\end{picture}
}}
{-}
\text{\raisebox{-12pt}[25pt]{
\unitlength=0.70mm
\linethickness{0.4pt}
\begin{picture}(26.00,16.33)
\put(0.00,5.00){\line(1,0){26.00}}
\put(2.00,5.00){\circle*{1.33}}
\put(13.00,5.00){\circle*{1.33}}
\put(24.00,5.00){\circle*{1.33}}
\put(2.00,1.33){\makebox(0,0)[cc]{\tiny\it1}}
\put(13.00,1.33){\makebox(0,0)[cc]{\tiny\it2}}
\put(24.00,1.33){\makebox(0,0)[cc]{\tiny\it3}}
\put(13,11.33){\circle*{1.33}}
\put(13,11.33){\vector(2,-1){10.8}}
\put(13,11.33){\vector(-2,-1){10.8}}
\put(18.5,17){\circle*{1.33}}
\put(18.5,17){\vector(-1,-1){5.2}}
\put(18.5,17){\vector(-1,-2){5.6}}
\put(13,15){\tiny $L$}
\put(17,12){\tiny $R$}
\put(4.00,10.00){\makebox(0,0)[cc]{\tiny$i$}}
\put(11.00,8.00){\makebox(0,0)[cc]{\tiny$j$}}
\put(22.00,10.00){\makebox(0,0)[cc]{\tiny$k$}}
\end{picture}
}}
{-}
\text{\raisebox{-12pt}[25pt]{
\unitlength=0.70mm
\linethickness{0.4pt}
\begin{picture}(26.00,16.33)
\put(0.00,5.00){\line(1,0){26.00}}
\put(2.00,5.00){\circle*{1.33}}
\put(13.00,5.00){\circle*{1.33}}
\put(24.00,5.00){\circle*{1.33}}
\put(2.00,1.33){\makebox(0,0)[cc]{\tiny\it1}}
\put(13.00,1.33){\makebox(0,0)[cc]{\tiny\it2}}
\put(24.00,1.33){\makebox(0,0)[cc]{\tiny\it3}}
\put(18.33,11.33){\circle*{1.33}}
\put(18.33,11.33){\vector(1,-1){5.5}}
\put(18.33,11.33){\vector(-1,-1){5.5}}
\put(13,17){\circle*{1.33}}
\put(13,17){\vector(-1,-1){11.2}}
\put(13,17){\vector(1,-1){5.1}}
\put(3.00,10.00){\makebox(0,0)[cc]{\tiny$i$}}
\put(13.00,10.00){\makebox(0,0)[cc]{\tiny$j$}}
\put(24.00,10.00){\makebox(0,0)[cc]{\tiny$k$}}
\end{picture}
}}
= 0. 
\end{equation}
%
%
In our notation 
this identity's left\/-\/hand side encodes a sum over all $(i,j,k)$; 
instead restricting to fixed $(i,j,k)$ corresponds to taking a coefficient of the differential operator (cf.\ Lemma~\ref{Lemma} below), which yields the respective component $\Jac_{\cP}^{ijk}$ of the Jacobiator~$\Jac(\cP)$.
Clearly, the Jacobiator 
is totally skew\/-\/symmetric with respect to its arguments~$\ground{1, 2, 3}$.

In fact, the Jacobiator~$\Jac(\cP)$ is the \emph{Schouten bracket} of a given Poisson bi\/-\/vector~$\cP$ with itself: $\Jac\,({\cP}) 
=\schouten{\cP,\cP} 
$ (depending on conventions, times a constant which is here omitted, cf.~\cite{RingersProtaras}). The 
bracket $\schouten{\cdot,\cdot}$ is a unique extension of the commutator $[\cdot,\cdot]$ on the space of vector fields $\cX^{1}(N^{n})$ to the space $\cX^{*}(N^{n})$ of multivector fields. Let us recall its inductive definition in the finite\/-\/dimensional set\/-\/up.
\begin{define}
The Schouten bracket $\lshad{\cdot},{\cdot}\rshad\colon \cX^{*}(N^{n})\times \cX^{*}(N^{n}) \to \cX^{*}(N^{n})$
coincides with the commutator~$[{\cdot},{\cdot}]$ 
when evaluated on $1$-vectors; when evaluated at a $p$-vector~$\cX$, $q$-vector~$\cY$, and $r$-vector~$\cZ$ for $p,q,r\geqslant 1$, the Schouten bracket 
is shifted\/-\/graded skew\/-\/symmetric,
$\schouten{\cX,\cY}=-(-1)^{(p-1)(q-1)}\schouten{\cY,\cX}$, and it works over each argument via the graded Leibniz rule:
$\schouten{\cX,\cY\wedge\cZ}=\schouten{\cX,\cY}\wedge\cZ+(-1)^{(p-1)q}\cY\wedge\schouten{\cX,\cZ}$. The bracket is then extended by linearity from homogeneous components to the entire space of multivector fields on~$N^n$. 
\end{define}

\begin{rem}
The construction of 
Schouten bracket also reads as follows.
Denote by~$\xi_{i}$ the parity\/-\/odd canonical conjugate of the variable~$x^{i}$ for every $i=1$,\ $\ldots$,\ $n$.
For instance, every bi\/-\/vector is 
realised in terms of local coordinates~$x^{i}$ and~$\xi_{i}$ on~$\Pi T^*N^n$ by using $\cP=\frac{1}{2}\langle \xi_{i}\,\cP^{ij}(\bx)\,\xi_{j}\rangle$.
The Schouten bracket $\schouten{\cdot,\cdot}$ is the parity\/-\/odd Poisson bracket which is locally determined on $\Pi T^{*}N^{n}$ 
by the canonical symplectic structure $\Id\bx\wedge \Id\boldsymbol{\xi}$. 
Our working formula~is\footnote{In the set\/-\/up of infinite jet spaces~$J^\infty(\pi)$ (see~\cite{Olver} and~\cite{gvbv,cycle14,dq15}) the four partial derivatives in formula~\eqref{EqSchouten} for~$\schouten{\cdot,\cdot}$ become the variational derivatives with respect to the same variables, which now pa\-ra\-met\-rise the fibres in the Whitney sum $\pi\mathbin{{\times}_{M^m}}\Pi\widehat{\pi}$ of 
(super-)\/bundles over the $m$-\/dimensional base~$M^m$.}
\begin{equation}\label{EqSchouten}
\schouten{\cP,\cQ} = (\cP){\frac{\overleftarrow{\partial}}{\partial x^{i}}}\cdot{\frac{\overrightarrow{\partial}}{\partial \xi_{i}}}(\cQ)
-(\cP){\frac{\overleftarrow{\partial}}{\partial \xi_{i}}}
\cdot{\frac{\overrightarrow{\partial}}{\partial x^{i}}}(\cQ).
\end{equation}
It is now readily seen that the Schouten bracket of homogeneous arguments
satisfies its own, shifted\/-\/graded Jacobi identity,
\[
\lshad \cX,\lshad \cY,\,\cdot\,\rshad\rshad(\cZ)-(-)^{(|\cX|-1)\cdot(|\cY|-1)}\lshad \cY,\lshad \cX,\,\cdot\,\rshad\rshad(\cZ)=
\lshad\lshad \cX,\cY\rshad,\,\cdot\,\rshad(\cZ).
\]
Hence for a 
bi\/-\/vector~$\cP$ such that $\schouten{\cP,\cP}=0$, 
the map $\boldsymbol{\dd}_{\cP}=\schouten{\cP,\cdot}\colon
\cX^\ell(N^n)\to\cX^{\ell+1}(N^n)$ is a differential.%
\end{rem}

\begin{rem}\label{RemSchoutenGraph}
The graphical calculation of the Schouten bracket $\schouten{\cdot ,\cdot }$ of two arguments amounts to the action --\,via the Leibniz rule\,-- of every out\/-\/going edge in an argument on all the internal vertices in the other argument.
For the Schouten bracket of a $k$-\/vector with an $\ell$-\/vector, the rule of signs is this.
For the sake of definition, enumerate the sinks in the first and second arguments by using $0$,\ $\ldots$,\ $k-1$ and $0$,\ $\ldots$, $\ell-1$, respectively.
Then the arrow into the $j${th} sink in the second argument acts on the internal vertices of the first argument, acquiring the sign factor $(-)^j$; here $0 \leqslant j < \ell$.
On the other hand, the arrow to the $i${th} sink in the first argument acts on the second argument's internal vertices with the sign factor $-(-)^{(k-1)-i}$ for $0 \leqslant i \leqslant k-1$.

The rule of signs, as it has been phrased above, is valid~--- provided that, for a $k$-\/vector~$\cX$ and $\ell$-\/vector~$\cY$, the numbers $0$,\ $\ldots$ of the~$k$ (or~$k-1$) sinks originating in the $(k+\ell-1)$-\/vector $\lshad\cX,\cY\rshad$ 
from the first argument~$\cX$ \emph{precede} the numbers of~$\ell-1$ (resp.,~$\ell$) sinks originating from~$\cY$ in the overall enumeration of those $k+\ell-1$~sinks.\footnote{Such is the default convention which formula~\eqref{EqSchouten} suggests for the product of 
parity\/-\/odd variables~$\xi_{i_\alpha}$, where~$1\leqslant\alpha\leqslant k+\ell-1$.}
For example, it is this ordering of sinks using $\mathit{1}\prec\mathit{2}$ which is shown in~\eqref{Fig21New},
\begin{align}\label{Fig21New}
\lshad
\unitlength=.8mm
\special{em:linewidth 0.4pt}
\linethickness{0.4pt}
\begin{picture}(10.67,7.67)(0,2)
\put(7.00,7.00){\circle*{1}}
\put(7.00,7.00){\vector(-1,-2){3.67}}
\put(7.00,7.00){\vector(1,-2){3.67}}
\put(3.00,-2.00){\makebox(0,0)[cc]{\tiny\it1}}
\put(10.00,-2.00){\makebox(0,0)[cc]{\tiny\it2}}
\end{picture}
\ ,\ %
\hspace{-5mm}
\begin{picture}(10.67,7.67)(0,2)
\put(7.00,7.00){\circle*{1}}
\put(7.00,7.00){\vector(0,-1){7.00}}
\put(7.00,-2.00){\makebox(0,0)[cc]{\tiny\it1}}
\end{picture}
\rshad=
{+} 
\begin{picture}(10.67,7.67)(0,0)
\put(7.00,0.00){\circle*{1}}
\put(7.00,0.00){\vector(-1,-2){3.67}}
\put(7.00,0.00){\vector(1,-2){3.67}}
\put(7.00,7.50){\circle*{1}}
\put(7.00,7.50){\vector(0,-1){7.00}}
\put(3.00,-9.00){\makebox(0,0)[cc]{\tiny\it1}}
\put(10.00,-9.00){\makebox(0,0)[cc]{\tiny\it2}}
\end{picture}
{-} \left(
\begin{picture}(10.67,7.67)(0,0)
\put(7.00,7.50){\circle*{1}}
\put(7.00,7.50){\vector(-1,-4){3.67}}
\put(7.00,7.50){\vector(1,-2){3.67}}
\put(10.67,0.00){\circle*{1}}
\put(10.67,0.00){\vector(0,-1){7.00}}
\put(3.00,-9.00){\makebox(0,0)[cc]{\tiny\it1}}
\put(10.00,-9.00){\makebox(0,0)[cc]{\tiny\it2}}
\end{picture}
\ \ 
{-}
\begin{picture}(10.67,7.67)(0,0)
\put(7.00,7.50){\circle*{1}}
\put(7.00,7.50){\vector(-1,-2){3.67}}
\put(7.00,7.50){\vector(1,-4){3.67}}
\put(3.33,0.00){\circle*{1}}
\put(3.33,0.00){\vector(0,-1){7.00}}
\put(3.00,-9.00){\makebox(0,0)[cc]{\tiny\it2}}
\put(10.00,-9.00){\makebox(0,0)[cc]{\tiny\it1}}
\end{picture}
\ \ 
\right);
\end{align}\\
here $k=2$,\ $\ell=1$ and the enumeration of arguments begins at $\it1$.

Still let us note that in its realization via Kontsevich graphs, the calculation of the Schouten bracket~$\lshad{\cdot},{\cdot}\rshad$ effectively amounts to a consecutive plugging of one of its arguments into each of the other argument's sinks (see~\eqref{Fig21New} again). Therefore, it would be more natural to start enumerating the sinks of the graph that acts (on the new content in one of its sinks, possibly the first), but when that new argument is reached, to interrupt and now enumerate the argument's own sinks, then continuing the enumeration of sinks (if there still remain any to be counted) in the graph that acts. This change of enumeration strategy comes at a price of having extra sign factors in front of the graphs. Namely, the arrow into the $j$th~sink in the second argument acquires the extra sign factor~$(-)^{j\cdot k}$. Similarly, the arrow to the $i$th~sink in the first argument of~$\lshad{\cdot},{\cdot}\rshad$ must now be multiplied by~$(-)^{\ell\cdot(k-1-i)}$; we recall that $0\leqslant i <k$ and~$0\leqslant j <\ell$. 
We note that for~$k$ and $\ell$~even (e.g., $k=2$ and~$\ell=2$ in formula~\eqref{EqLHSGraphs}) 
no extra sign factors appear at all from the re\/-\/ordering at a price of~$(-)^{j\cdot k}$ and~$(-)^{\ell\cdot(k-1-i)}$. For example, such is the final ordering of the $3=2+1=1+2$ sinks which is shown in Fig.~\ref{FigLHS} on p.~\pageref{FigLHS}.
\end{rem}

Summarizing, to be Poisson a bi\/-\/vector~$\cP$ must satisfy the \emph{master\/-\/equation},
\begin{equation}\label{EqMasterEq}
\schouten{\cP,\cP}=0,
\end{equation}
of which formula~\eqref{EqJacFig} is the component 
expansion with respect to the indices 
$(i,j,k)$ in the tri\/-\/vector~$\schouten{\cP,\cP}(\bx,\boldsymbol{\xi})$.

\begin{define}\label{DefInfPoisson}
Let~$\cP$ be a Poisson bi\/-\/vector on the manifold~$N^n$ at hand and consider its deformation $\cP\mapsto\cP+\veps \cQ(\cP)+\bar{o}(\veps)$. We say that after such deformation the bi\/-\/vector stays \emph{infinitesimally Poisson} if
\begin{equation}\tag{\ref{Eq1stOrder}${}'$}
\schouten{\cP+\varepsilon \cQ(\cP)+\bar{o}(\varepsilon),\cP+\varepsilon \cQ(\cP)+\bar{o}(\varepsilon)}=\bar{o}(\varepsilon),
\end{equation}
that is, the master\/-\/equation is still satisfied up to~$\bar{o}(\veps)$ for a given solution~$\cP$ of~\eqref{EqMasterEq}.
\end{define}

\begin{rem}
Nowhere above should one expect that the leading deformation term~$\cQ$ in~
$\cP+\varepsilon\cQ+\bar{o}(\varepsilon)$ itself would be a Poisson bi\/-\/vector. This may happen for~$\cQ$ only incidentally.
\end{rem}

Expanding the left\/-\/hand side 
of equation~\eqref{Eq1stOrder} and using the shifted\/-\/graded skew\/-\/symmetry of the Schouten bracket $\schouten{\cdot,\cdot}$, 
we extract the deformation equation 
\begin{equation}\tag{\ref{EqWhichMechanism}}
\schouten{\cP, \cQ} \doteq 0 \quad \quad \text{via} \  
\schouten{\cP,\cP} = 0. 
\end{equation}
Let us consider a class of its solutions~$\cQ=\cQ(\cP)$ which are universal with respect to all finite\/-\/dimensional affine Poisson manifolds~$(N^n,\cP)$.

\subsection{The Kontsevich tetrahedral flow}
In the paper \cite{Ascona96}, Kontsevich proposed a particular construction of 
infinitesimal deformations $\cP\mapsto\cP+\veps \cQ(\cP)+\bar{o}(\veps)$
for Poisson structures on affine real manifolds. 
One such flow~$\dot{\cP}=\cQ(\cP)$ on the space of Poisson bi-vectors~$\cP$ 
is associated with the complete graph on four vertices, that is, the tetrahedron. Up to symmetry, there are two essentially different ways, resulting in $\Gamma_{1}$ and $\Gamma_{2}'$, to orient its edges, provided that every vertex is a source for two arrows and, as an elementary count suggests,
there are two arrows leaving the tetrahedron and acting on the arguments of the bi-differential operator which the tetrahedral graph encodes. The two oriented tetrahedral graphs are shown in Fig.~\ref{FigTetra}.
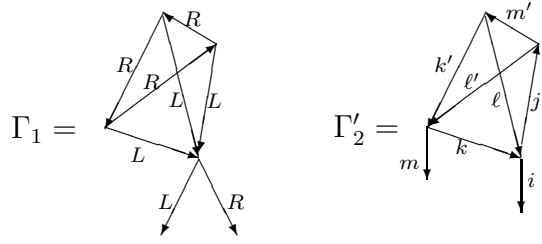
\begin{figure}[htb]
\unitlength=1mm
\special{em:linewidth 0.4pt}
\linethickness{0.4pt}
\begin{picture}(30.00,30.00)
\put(20.00,30.00){\vector(-1,-2){7.67}}
\put(12.33,15.00){\vector(4,3){14.67}}
\put(20.00,30.1){\vector(-4,3){0}}
\put(27.00,26){\line(-5,3){7.00}}
\put(20.00,29.67){\vector(1,-4){4.67}}
\put(24.7,10.65){\vector(1,-2){5.00}}
\put(24.7,10.65){\vector(-1,-2){5.00}}
\put(12.33,15.00){\vector(3,-1){12.33}}
\put(27.00,26.00){\line(-1,-6){2.2}}
\put(24.85,12.5){\vector(-1,-4){0.2}}
\put(0.00,14.67){\makebox(0,0)[lc]{$\Gamma_1={}$}}
\put(16,22.33){\makebox(0,0)[rb]{\tiny$R$}}
\put(22.7,17.67){\makebox(0,0)[rb]{\tiny$L$}}
\put(22.67,28.33){\makebox(0,0)[lb]{\tiny$R$}}
\put(25.67,18.5){\makebox(0,0)[lc]{\tiny$L$}}
\put(19.5,20.00){\makebox(0,0)[rb]{\tiny$R$}}
\put(17.67,12.33){\makebox(0,0)[rt]{\tiny$L$}}
\put(21.33,5.00){\makebox(0,0)[rc]{\tiny$L$}}
\put(28.33,5.00){\makebox(0,0)[lc]{\tiny$R$}}
\end{picture}
\qquad
\unitlength=1mm
\special{em:linewidth 0.4pt}
\linethickness{0.4pt}
\begin{picture}(30.00,30.00)
\put(20.00,30.00){\vector(-1,-2){7.67}}
\put(27,26){\vector(-4,-3){14.5}}
\put(20.00,30.1){\vector(-4,3){0}}
\put(27.00,26){\line(-5,3){7.00}}
\put(20.00,29.67){\vector(1,-4){4.67}}
\put(24.7,10.65){\vector(0,-1){7.00}}
\put(12.33,15){\vector(0,-1){7.00}}
\put(12.33,15.00){\vector(3,-1){12.33}}
\put(27.00,26.00){\line(-1,-6){2.22}}
\put(27.00,26){\vector(1,4){0}}
\put(0.00,14.67){\makebox(0,0)[lc]{$\Gamma_2'={}$}}
\put(15.83,22.33){\makebox(0,0)[rb]{\tiny$k'$}}
\put(22.00,17.67){\makebox(0,0)[rb]{\tiny$\ell$}}
\put(22.67,29.33){\makebox(0,0)[lb]{\tiny$m'$}}
\put(26,18.33){\makebox(0,0)[lc]{\tiny$j$}}
\put(19.33,20.00){\makebox(0,0)[rb]{\tiny$\ell'$}}
\put(17.67,13.33){\makebox(0,0)[rt]{\tiny$k$}}
\put(11.33,10.00){\makebox(0,0)[rc]{\tiny$m$}}
\put(25.5,8.00){\makebox(0,0)[lc]{\tiny$i$}}
\end{picture}
\caption{The Kontsevich tetrahedral graphs encode two bi\/-\/linear bi\/-\/differential operators on the product $C^\infty(N^n)\times C^\infty(N^n)$.}\label{FigTetra}
\end{figure}
Unlike the operator encoded by $\Gamma_{1}$, that of $\Gamma_{2}'$ is generally speaking not skew\/-\/symmetric with respect to its arguments.
By definition, put $\Gamma_2 \mathrel{{:}{=}} \frac{1}{2}(\Gamma_{2}'(\ground{1},\ground{2})-\Gamma_{2}'(\ground{2},\ground{1}))$ to extract the antisymmetric part, that is, the bi\/-\/vector encoded by~$\Gamma_{2}'$. Explicitly, the quartic-nonlinear differential polynomials $\Gamma_{1}(\cP)$ and $\Gamma_{2}(\cP)$, depending on a Poisson bi-vector $\cP$, are given by the formulae
\begin{subequations}\label{EqFlows}
\begin{align}\label{EqFlow1}
\Gamma_{1}(\cP)&=\sum_{i,j=1}^{n} 
\biggl(\sum\limits_{k,\ell,m,k',\ell',m'=1}^{n} \frac{\partial^{3}\cP^{ij}}{\partial x^{k} \partial x^{\ell}  \partial x^{m}} \frac{\partial \cP^{kk'} }{\partial x^{\ell'}} \frac{\partial \cP^{\ell\ell'}}{\partial x^{m'}} \frac{\partial \cP^{mm'}}{\partial x^{k'}}
\biggr)\frac{\partial}{\partial x^{i}} \wedge \frac{\partial}{\partial x^{j}}\\
\intertext{and}
\label{EqFlow2}
\Gamma_{2}(\cP)&=\sum_{i,m=1}^{n} 
\biggl( \sum\limits_{j,k,\ell,k',\ell',m'=1}^{n} \frac{\partial^{2}\cP^{ij}}{\partial x^{k} \partial x^{\ell}} \frac{\partial^{2}\cP^{km}}{\partial x^{k'} \partial x^{\ell'} }  \frac{\partial\cP^{k'\ell}}{\partial x^{m'}}   \frac{\partial\cP^{m'\ell'}}{\partial x^{j}}   
\biggr)\frac{\partial}{\partial x^{i}} \wedge \frac{\partial}{\partial x^{m}},
\end{align}
\end{subequations}
respectively.
To construct a class of flows on the space of bi-vectors, 
Kontsevich suggested to consider linear combinations, balanced by using the ratio $a:b$, of the 
bi\/-\/vectors~$\Gamma_{1}$ and~$\Gamma_{2}$.
We recall from section~\ref{SecGraphOp} that every internal vertex of each graph is inhabited by a copy of a given Poisson bi\/-\/vector~$\cP$, so that the linear combination of two graphs encodes the bi\/-\/vector $\cQ_{a:b}(\cP) = a\cdot\Gamma_1(\cP) + b\cdot\Gamma_2(\cP)$, quartic in $\cP$ and balanced using~$a:b$.
We now inspect at which ratio $a:b$ the bi\/-\/vector $\cP+\varepsilon \cQ_{a:b}(\cP)+\bar{o}(\varepsilon)$ stays infinitesimally Poisson
, that is, 
\begin{equation}\tag{\ref{Eq1stOrder}}
\schouten{\cP+\varepsilon \cQ_{a:b}(\cP)+\bar{o}(\varepsilon),\cP+\varepsilon \cQ_{a:b}(\cP)+\bar{o}(\varepsilon)}=\bar{o}(\varepsilon).
\end{equation}
The left\/-\/hand side of the deformation equation, 
\begin{equation}\tag{\ref{EqWhichMechanism}}
\schouten{\cP, \cQ_{a:b}(\cP)} \doteq 0 \quad \quad \text{via} \  
\schouten{\cP,\cP} = 0,
\end{equation}
can be seen in terms of graphs:
%
\begin{equation}
\schouten{\cP,a \cdot \Gamma_{1}+b \cdot \Gamma_{2}}= 
\Biggl\llbracket
\unitlength=1mm
\special{em:linewidth 0.4pt}
\linethickness{0.4pt}
\hspace{-20mm}
\raisebox{-20pt}[15pt][30pt]{
	\begin{picture}(30.00,0.00)
	\put(24.7,10.65){\vector(1,-2){3.60}}
	\put(24.7,10.65){\vector(-1,-2){3.60}}
	\put(28.00,2.00){\makebox(0,0)[cc]{\tiny\it2}}
	\put(20.70,2.00){\makebox(0,0)[cc]{\tiny\it1}}
	\end{picture}}
,
a \cdot\hspace{-8mm}
\unitlength=.6mm
\special{em:linewidth 0.4pt}
\linethickness{0.4pt}
\raisebox{-30pt}[15pt][25pt]{
	\begin{picture}(30.00,30.00)
	\put(20.00,30.00){\vector(-1,-2){7.67}}
	\put(12.33,15.00){\vector(4,3){14.67}}
	\put(20.00,30.1){\vector(-4,3){0}}
	\put(27.00,26){\line(-5,3){7.00}}
	\put(20.00,29.67){\vector(1,-4){4.67}}
	\put(24.7,10.65){\vector(1,-2){3.60}}
	\put(24.7,10.65){\vector(-1,-2){3.60}}
	\put(12.33,15.00){\vector(3,-1){12.33}}
	\put(27.00,26.00){\line(-1,-6){2.2}}
	\put(24.85,12.5){\vector(-1,-4){0.2}}
	\put(28.00,1.00){\makebox(0,0)[cc]{\tiny\it2}}
	\put(20.70,1.00){\makebox(0,0)[cc]{\tiny\it1}}
	\end{picture}}
+ \frac{b}{2} \cdot
\Biggl(
\hspace{-7mm}
\raisebox{-30pt}[15pt][15pt]{
	\begin{picture}(30.00,30.00)
	\put(20.00,30.00){\vector(-1,-2){7.67}}
	\put(27,26){\vector(-4,-3){14.5}}
	\put(20.00,30.1){\vector(-4,3){0}}
	\put(27.00,26){\line(-5,3){7.00}}
	\put(20.00,29.67){\vector(1,-4){4.67}}
	\put(24.7,10.65){\vector(0,-1){7.00}}
	\put(12.33,15){\vector(0,-1){11.30}}
	\put(12.33,15.00){\vector(3,-1){12.33}}
	\put(27.00,26.00){\line(-1,-6){2.22}}
	\put(27.00,26){\vector(1,4){0}}
	\put(24.00,1.00){\makebox(0,0)[cc]{\tiny\it2}}
	\put(12.00,1.00){\makebox(0,0)[cc]{\tiny\it1}}
	\end{picture}}
-
\hspace{-8mm}
\raisebox{-30pt}[15pt][6.8mm]{
	\begin{picture}(30.00,30.00)
	\put(20.00,30.00){\vector(-1,-2){7.67}}
	\put(27,26){\vector(-4,-3){14.5}}
	\put(20.00,30.1){\vector(-4,3){0}}
	\put(27.00,26){\line(-5,3){7.00}}
	\put(20.00,29.67){\vector(1,-4){4.67}}
	\qbezier(24.7,10.65)(23.2,9.65)(21.7,8.65)
	\qbezier(17.7,6.65)(15,5.15)(12.33,3.65)
	\put(12.33,3.65){\vector(-4,-3){0.00}}
	\qbezier(12.33,15)(18.515,9.325)(24.7,3.65)
	\put(24.7,3.65){\vector(4,-3){0.00}}
	\put(12.33,15.00){\vector(3,-1){12.33}}
	\put(27.00,26.00){\line(-1,-6){2.22}}
	\put(27.00,26){\vector(1,4){0}}
	\put(24.00,1.00){\makebox(0,0)[cc]{\tiny\it2}}
	\put(12.00,1.00){\makebox(0,0)[cc]{\tiny\it1}}
	\end{picture}}
\Biggr)
\Biggr\rrbracket.
\label{EqLHSGraphs}
\end{equation}
Let $a:b=1:6$ (specifically, $a=\tfrac{1}{4}$ and $b=\tfrac{3}{2}$).
Then the left\/-\/hand side of~\eqref{EqWhichMechanism} takes the shape depicted in Fig.~\ref{FigLHS}. 
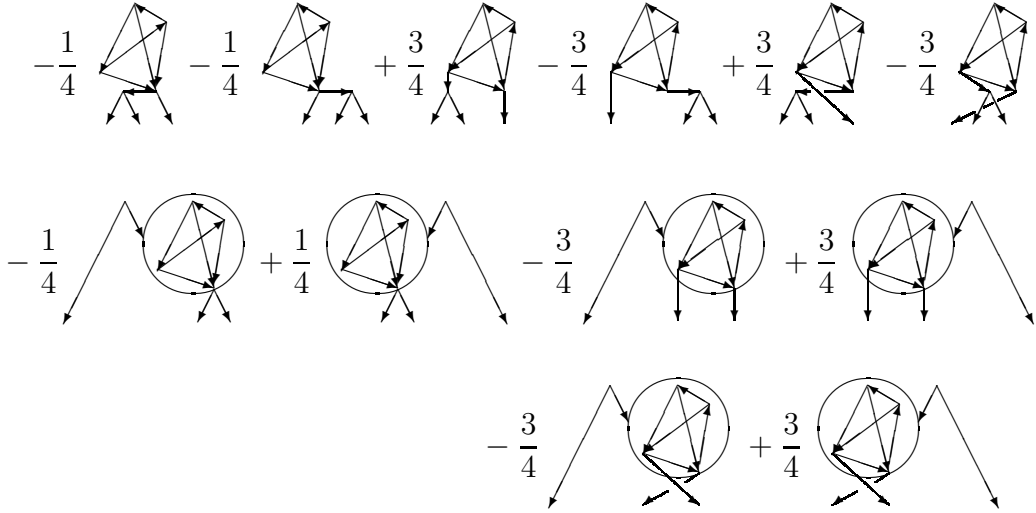
\begin{figure}[htb]
\begin{multline*}
\qquad\ \ {-}\frac{1}{4}	\unitlength=.6mm
	\special{em:linewidth 0.4pt}
	\linethickness{0.4pt}
	\raisebox{-25pt}[40pt][30pt]{
		\begin{picture}(20.00,30.00)(10,0)
		\put(20.00,30.00){\vector(-1,-2){7.67}}
		\put(12.33,15.00){\vector(4,3){14.67}}
		\put(20.00,30.1){\vector(-4,3){0}}
		\put(27.00,26){\line(-5,3){7.00}}
		\put(20.00,29.67){\vector(1,-4){4.67}}
		\put(24.7,10.65){\vector(1,-2){3.60}}
		\put(24.7,10.65){\vector(-1,0){7.20}}
		\put(12.33,15.00){\vector(3,-1){12.33}}
		\put(27.00,26.00){\line(-1,-6){2.2}}
		\put(24.85,12.5){\vector(-1,-4){0.2}}
		\put(17.5,10.65){\vector(1,-2){3.60}}
		\put(17.5,10.65){\vector(-1,-2){3.60}}
		\end{picture}}
{}-\frac{1}{4}
	\unitlength=.6mm
	\special{em:linewidth 0.4pt}
	\linethickness{0.4pt}
	\raisebox{-25pt}[40pt][30pt]{
		\begin{picture}(25.00,30.00)(10,0)
		\put(20.00,30.00){\vector(-1,-2){7.67}}
		\put(12.33,15.00){\vector(4,3){14.67}}
		\put(20.00,30.1){\vector(-4,3){0}}
		\put(27.00,26){\line(-5,3){7.00}}
		\put(20.00,29.67){\vector(1,-4){4.67}}
		\put(24.7,10.65){\vector(1,0){7.20}}
		\put(24.7,10.65){\vector(-1,-2){3.60}}
		\put(12.33,15.00){\vector(3,-1){12.33}}
		\put(27.00,26.00){\line(-1,-6){2.2}}
		\put(24.85,12.5){\vector(-1,-4){0.2}}
		\put(31.9,10.65){\vector(1,-2){3.60}}
		\put(31.9,10.65){\vector(-1,-2){3.60}}
		\end{picture}}
{}+\frac{3}{4}
	\unitlength=.6mm
	\special{em:linewidth 0.4pt}
	\linethickness{0.4pt}
	\raisebox{-25pt}[40pt][30pt]{
		\begin{picture}(20.00,30.00)(10,0)
		\put(20.00,30.00){\vector(-1,-2){7.67}}
		\put(27,26){\vector(-4,-3){14.5}}
		\put(20.00,30.1){\vector(-4,3){0}}
		\put(27.00,26){\line(-5,3){7.00}}
		\put(20.00,29.67){\vector(1,-4){4.67}}
		\put(24.7,10.65){\vector(0,-1){7.00}}
		\put(12.33,15){\vector(0,-1){4.5}}
		\put(12.33,15.00){\vector(3,-1){12.33}}
		\put(27.00,26.00){\line(-1,-6){2.22}}
		\put(27.00,26){\vector(1,4){0}}
		\put(12.33,10.65){\vector(1,-2){3.60}}
		\put(12.33,10.65){\vector(-1,-2){3.60}}
		\end{picture}}
{}-\frac{3}{4}
	\unitlength=.6mm
	\special{em:linewidth 0.4pt}
	\linethickness{0.4pt}
	\raisebox{-25pt}[40pt][30pt]{
		\begin{picture}(25.00,30.00)(10,0)
		\put(20.00,30.00){\vector(-1,-2){7.67}}
		\put(27,26){\vector(-4,-3){14.5}}
		\put(20.00,30.1){\vector(-4,3){0}}
		\put(27.00,26){\line(-5,3){7.00}}
		\put(20.00,29.67){\vector(1,-4){4.67}}
		\put(24.7,10.65){\vector(1,0){7.00}}
		\put(12.33,15){\vector(0,-1){11.30}}
		\put(12.33,15.00){\vector(3,-1){12.33}}
		\put(27.00,26.00){\line(-1,-6){2.22}}
		\put(27.00,26){\vector(1,4){0}}
		\put(31.9,10.65){\vector(1,-2){3.60}}
		\put(31.9,10.65){\vector(-1,-2){3.60}}
		\end{picture}}
{}+\frac{3}{4}
	\unitlength=.6mm
	\special{em:linewidth 0.4pt}
	\linethickness{0.4pt}
	\raisebox{-25pt}[40pt][30pt]{
		\begin{picture}(20.00,30.00)(10,0)
		\put(20.00,30.00){\vector(-1,-2){7.67}}
		\put(27,26){\vector(-4,-3){14.5}}
		\put(20.00,30.1){\vector(-4,3){0}}
		\put(27.00,26){\line(-5,3){7.00}}
		\put(20.00,29.67){\vector(1,-4){4.67}}
		\qbezier(24.7,10.65)(21.7,10.65)(18.7,10.65)
		\qbezier(15.33,10.65)(13.33,10.65)(12.33,10.65)
		\put(12.33,10.65){\vector(-1,-0){0.00}}
		\qbezier(12.33,15)(18.515,9.325)(24.7,3.65)
		\put(24.7,3.65){\vector(4,-3){0.00}}
		\put(12.33,15.00){\vector(3,-1){12.33}}
		\put(27.00,26.00){\line(-1,-6){2.22}}
		\put(27.00,26){\vector(1,4){0}}
		\put(12.33,10.65){\vector(1,-2){3.60}}
		\put(12.33,10.65){\vector(-1,-2){3.60}}
		\end{picture}}
{}-\frac{3}{4}
	\unitlength=.6mm
	\special{em:linewidth 0.4pt}
	\linethickness{0.4pt}
	\raisebox{-25pt}[40pt][30pt]{
		\begin{picture}(20.00,30.00)(10,0)
		\put(20.00,30.00){\vector(-1,-2){7.67}}
		\put(27,26){\vector(-4,-3){14.5}}
		\put(20.00,30.1){\vector(-4,3){0}}
		\put(27.00,26){\line(-5,3){7.00}}
		\put(20.00,29.67){\vector(1,-4){4.67}}
		\qbezier(12.33,15)(12.33,15)(19,10.5)
		\put(19,10.5){\vector(2,-1){0.00}}
		\qbezier(24.7,10.65)(24.7,10.65)(20.7,8.65)
		\qbezier(17.65,7.15)(17.65,7.15)(19.65,8.15)
		\qbezier(15.65,6.15)(15.65,6.15)(10.65,3.65)
		\put(10.65,3.65){\vector(-2,-1){0}}
		\put(12.33,15.00){\vector(3,-1){12.33}}
		\put(27.00,26.00){\line(-1,-6){2.22}}
		\put(27.00,26){\vector(1,4){0}}
		\put(19,10.65){\vector(1,-2){3.60}}
		\put(19,10.65){\vector(-1,-2){3.60}}
		\end{picture}}
	\\
{}-{\frac{1}{4}} \quad
	\unitlength=.6mm
	\special{em:linewidth 0.4pt}
	\linethickness{0.4pt}
	\raisebox{-25pt}[40pt][25pt]{
		\begin{picture}(33.00,30.00)
		\put(20.00,30.00){\vector(-1,-2){7.67}}
		\put(12.33,15.00){\vector(4,3){14.67}}
		\put(20.00,30.1){\vector(-4,3){0}}
		\put(27.00,26){\line(-5,3){7.00}}
		\put(20.00,29.67){\vector(1,-4){4.67}}
		\put(24.7,10.65){\vector(1,-2){3.60}}
		\put(24.7,10.65){\vector(-1,-2){3.60}}
		\put(12.33,15.00){\vector(3,-1){12.33}}
		\put(27.00,26.00){\line(-1,-6){2.2}}
		\put(24.85,12.5){\vector(-1,-4){0.2}}
		\put(20.2,20.65){\oval(22,22)}
		\put(5.20,30.00){\vector(1,-2){4}}
		\put(5.20,30.00){\vector(-1,-2){13.5}}
		\end{picture}}
{}+{\frac{1}{4}} \hspace{1mm}
	\unitlength=.6mm
	\special{em:linewidth 0.4pt}
	\linethickness{0.4pt}
	\raisebox{-25pt}[40pt][25pt]{
		\begin{picture}(40.00,30.00)(10,0)
		\put(20.00,30.00){\vector(-1,-2){7.67}}
		\put(12.33,15.00){\vector(4,3){14.67}}
		\put(20.00,30.1){\vector(-4,3){0}}
		\put(27.00,26){\line(-5,3){7.00}}
		\put(20.00,29.67){\vector(1,-4){4.67}}
		\put(24.7,10.65){\vector(1,-2){3.60}}
		\put(24.7,10.65){\vector(-1,-2){3.60}}
		\put(12.33,15.00){\vector(3,-1){12.33}}
		\put(27.00,26.00){\line(-1,-6){2.2}}
		\put(24.85,12.5){\vector(-1,-4){0.2}}
		\put(20.2,20.65){\oval(22,22)}
		\put(35.20,30.00){\vector(1,-2){13.5}}
		\put(35.20,30.00){\vector(-1,-2){4}}
		\end{picture}}
{}- \frac{3}{4} 
	\unitlength=.6mm
	\special{em:linewidth 0.4pt}
	\linethickness{0.4pt}
	\raisebox{-25pt}[40pt][25pt]{
		\begin{picture}(42.00,30.00)(-8,0)
		\put(20.00,30.00){\vector(-1,-2){7.67}}
		\put(27,26){\vector(-4,-3){14.5}}
		\put(20.00,30.1){\vector(-4,3){0}}
		\put(27.00,26){\line(-5,3){7.00}}
		\put(20.00,29.67){\vector(1,-4){4.67}}
		\put(24.7,10.65){\vector(0,-1){7.00}}
		\put(12.33,15){\vector(0,-1){11.30}}
		\put(12.33,15.00){\vector(3,-1){12.33}}
		\put(27.00,26.00){\line(-1,-6){2.22}}
		\put(27.00,26){\vector(1,4){0}}
		\put(20.2,20.65){\oval(22,22)}
		\put(5.20,30.00){\vector(1,-2){4}}
		\put(5.20,30.00){\vector(-1,-2){13.5}}
		\end{picture}}
	\unitlength=.6mm
	\special{em:linewidth 0.4pt}
	\linethickness{0.4pt}
{}+\frac{3}{4} 
	\unitlength=.6mm
	\special{em:linewidth 0.4pt}
	\linethickness{0.4pt}
	\raisebox{-25pt}[40pt][25pt]{
		\begin{picture}(40.00,30.00)(8,0)
		\put(20.00,30.00){\vector(-1,-2){7.67}}
		\put(27,26){\vector(-4,-3){14.5}}
		\put(20.00,30.1){\vector(-4,3){0}}
		\put(27.00,26){\line(-5,3){7.00}}
		\put(20.00,29.67){\vector(1,-4){4.67}}
		\put(24.7,10.65){\vector(0,-1){7.00}}
		\put(12.33,15){\vector(0,-1){11.30}}
		\put(12.33,15.00){\vector(3,-1){12.33}}
		\put(27.00,26.00){\line(-1,-6){2.22}}
		\put(27.00,26){\vector(1,4){0}}
		\put(20.2,20.65){\oval(22,22)}
		\put(35.20,30.00){\vector(1,-2){13.5}}
		\put(35.20,30.00){\vector(-1,-2){4}}
		\end{picture}}
\\ 
{}- \frac{3}{4}  
	\unitlength=.6mm
	\special{em:linewidth 0.4pt}
	\linethickness{0.4pt}
	\raisebox{-25pt}[40pt][25pt]{
		\begin{picture}(42.00,30.00)(-8,0)
		\put(20.00,30.00){\vector(-1,-2){7.67}}
		\put(27,26){\vector(-4,-3){14.5}}
		\put(20.00,30.1){\vector(-4,3){0}}
		\put(27.00,26){\line(-5,3){7.00}}
		\put(20.00,29.67){\vector(1,-4){4.67}}
		\qbezier(24.7,10.65)(23.2,9.65)(21.7,8.65)
		\qbezier(17.7,6.65)(15,5.15)(12.33,3.65)
		\put(12.33,3.65){\vector(-4,-3){0.00}}
		\qbezier(12.33,15)(18.515,9.325)(24.7,3.65)
		\put(24.7,3.65){\vector(4,-3){0.00}}
		\put(12.33,15.00){\vector(3,-1){12.33}}
		\put(27.00,26.00){\line(-1,-6){2.22}}
		\put(27.00,26){\vector(1,4){0}}
		\put(20.2,20.65){\oval(22,22)}
		\put(5.20,30.00){\vector(1,-2){4}}
		\put(5.20,30.00){\vector(-1,-2){13.5}}
		\end{picture}}
{}+ \frac{3}{4} 
	\unitlength=.6mm
	\special{em:linewidth 0.4pt}
	\linethickness{0.4pt}
	\raisebox{-25pt}[40pt][25pt]{
		\begin{picture}(45.00,30.00)(8,0)
		\put(20.00,30.00){\vector(-1,-2){7.67}}
		\put(27,26){\vector(-4,-3){14.5}}
		\put(20.00,30.1){\vector(-4,3){0}}
		\put(27.00,26){\line(-5,3){7.00}}
		\put(20.00,29.67){\vector(1,-4){4.67}}
		\qbezier(24.7,10.65)(23.2,9.65)(21.7,8.65)
		\qbezier(17.7,6.65)(15,5.15)(12.33,3.65)
		\put(12.33,3.65){\vector(-4,-3){0.00}}
		\qbezier(12.33,15)(18.515,9.325)(24.7,3.65)
		\put(24.7,3.65){\vector(4,-3){0.00}}
		\put(12.33,15.00){\vector(3,-1){12.33}}
		\put(27.00,26.00){\line(-1,-6){2.22}}
		\put(27.00,26){\vector(1,4){0}}
		\put(20.2,20.65){\oval(22,22)}
		\put(35.20,30.00){\vector(1,-2){13.5}}
		\put(35.20,30.00){\vector(-1,-2){4}}
		\end{picture}}\qquad
\end{multline*}
\caption{Incoming arrows act on the content 
of boxes via the Leibniz rule; to obtain the tri\/-\/vector, the entire picture must be skew\/-\/symmetrized over the content of three sinks using $\sum_{\sigma \in S_{3}}(-)^{\sigma}$. Expanding and skew-symmetrizing, one obtains $39$ graphs in the left-hand side of~\eqref{EqWhichMechanism}.}\label{FigLHS}
\end{figure}
After the expansion of Leibniz rules and skew\/-\/symmetrization, 
the sum in Fig.~\ref{FigLHS} simplifies to 39~graphs; they are listed in Table~\ref{TableLHS} on p.~\pageref{TableLHS} below. 
Collecting, we conclude that  the left-hand side of \eqref{EqWhichMechanism} is the sum of 9~manifestly skew\/-\/symmetric expressions, see Fig.~\ref{Fig9Skew} (and Table~\ref{Tab9Skew} in Appendix~\ref{AppCode}).
\begin{figure}[htb]
\begin{align*}
\unitlength=.6mm
\special{em:linewidth 0.4pt}
\linethickness{0.4pt}
\raisebox{0mm}{${}  \sum\limits_{\sigma\in S_3}(-)^\sigma\Biggl\{$}{-\frac{1}{2}}\ 
\raisebox{-25pt}[40pt][30pt]{
	\begin{picture}(20.00,30.00)(10,0)
	\put(20.00,30.00){\vector(-1,-2){7.67}}
	\put(12.33,15.00){\vector(4,3){14.67}}
	\put(20.00,30.1){\vector(-4,3){0}}
	\put(27.00,26){\line(-5,3){7.00}}
	\put(20.00,29.67){\vector(1,-4){4.67}}
	\put(24.7,10.65){\vector(1,-2){3.60}}
	\put(24.7,10.65){\vector(-1,0){7.20}}
	\put(12.33,15.00){\vector(3,-1){12.33}}
	\put(27.00,26.00){\line(-1,-6){2.2}}
	\put(24.85,12.5){\vector(-1,-4){0.2}}
	\put(17.5,10.65){\vector(1,-2){3.60}}
	\put(17.5,10.65){\vector(-1,-2){3.60}}
	\put(13.50,0.50){\makebox(0,0)[cc]{\tiny$0$}}
	\put(21.00,0.50){\makebox(0,0)[cc]{\tiny$1$}}
	\put(29.00,0.50){\makebox(0,0)[cc]{\tiny$2$}}
	\put(15.00,11.00){\makebox(0,0)[cc]{\tiny$3$}}
	\put(27.50,11.00){\makebox(0,0)[cc]{\tiny$4$}}
	\put(10.00,14.50){\makebox(0,0)[cc]{\tiny$5$}}
	\put(29.50,27.50){\makebox(0,0)[cc]{\tiny$6$}}
	\put(20.00,33.00){\makebox(0,0)[cc]{\tiny$7$}}
	\end{picture}}
{}-{\frac{1}{2}}\ 
\raisebox{-25pt}[40pt][25pt]{
		\begin{picture}(40.00,30.00)(10,0)
		\put(20.00,30.00){\vector(-1,-2){7.67}}
		\put(12.33,15.00){\vector(4,3){14.67}}
		\put(20.00,30.1){\vector(-4,3){0}}
		\put(27.00,26){\line(-5,3){7.00}}
		\put(20.00,29.67){\vector(1,-4){4.67}}
		\put(24.7,10.65){\vector(1,-2){3.60}}
		\put(24.7,10.65){\vector(-1,-2){3.60}}
		\put(12.33,15.00){\vector(3,-1){12.33}}
		\put(27.00,26.00){\line(-1,-6){2.2}}
		\put(24.85,12.5){\vector(-1,-4){0.2}}
		\put(28.7,19.00){\vector(1,-2){7.8}}
		\put(28.7,19){\vector(-1,-2){4}}
		\put(21.00,0.50){\makebox(0,0)[cc]{\tiny$0$}}
		\put(28.50,0.50){\makebox(0,0)[cc]{\tiny$1$}}
		\put(36.50,0.50){\makebox(0,0)[cc]{\tiny$2$}}
		\put(30.50,21.50){\makebox(0,0)[cc]{\tiny$3$}}
		\put(27.50,11.00){\makebox(0,0)[cc]{\tiny$4$}}
		\put(10.00,14.50){\makebox(0,0)[cc]{\tiny$5$}}
		\put(29.50,27.50){\makebox(0,0)[cc]{\tiny$6$}}
		\put(20.00,33.00){\makebox(0,0)[cc]{\tiny$7$}}
		\end{picture}}\hspace{-7mm}
{}+{\frac{3}{2}}
\raisebox{-25pt}[40pt][25pt]{
		\begin{picture}(33.00,30.00)
		\put(20.00,30.00){\vector(-1,-2){7.67}}
		\put(12.33,15.00){\vector(4,3){14.67}}
		\put(20.00,30.1){\vector(-4,3){0}}
		\put(27.00,26){\line(-5,3){7.00}}
		\put(20.00,29.67){\vector(1,-4){4.67}}
		\put(24.7,10.65){\vector(1,-2){3.60}}
		\put(24.7,10.65){\vector(-1,-2){3.60}}
		\put(12.33,15.00){\vector(3,-1){12.33}}
		\put(27.00,26.00){\line(-1,-6){2.2}}
		\put(24.85,12.5){\vector(-1,-4){0.2}}
		\put(8.33,23.00){\vector(1,-2){4}}
		\put(8.33,23.00){\vector(-1,-2){9.66}}
		\end{picture}}
\hspace{-2mm}
{}+{\frac{3}{2}}
\raisebox{-25pt}[40pt][30pt]{
		\begin{picture}(25.00,30.00)(10,0)
		\put(20.00,30.00){\vector(-1,-2){7.67}}
		\put(27,26){\vector(-4,-3){14.5}}
		\put(20.00,30.1){\vector(-4,3){0}}
		\put(27.00,26){\line(-5,3){7.00}}
		\put(20.00,29.67){\vector(1,-4){4.67}}
		\put(24.7,10.65){\vector(1,0){7.00}}
		\put(12.33,15){\vector(0,-1){11.30}}
		\put(12.33,15.00){\vector(3,-1){12.33}}
		\put(27.00,26.00){\line(-1,-6){2.22}}
		\put(27.00,26){\vector(1,4){0}}
		\put(31.9,10.65){\vector(1,-2){3.60}}
		\put(31.9,10.65){\vector(-1,-2){3.60}}
		\end{picture}}
{}+{\frac{3}{2}}
\raisebox{-25pt}[40pt][30pt]{
		\begin{picture}(25.00,30.00)(10,0)
		\put(20.00,30.00){\vector(-1,-2){7.67}}
		\put(27,26){\vector(-4,-3){14.5}}
		\put(20.00,30.1){\vector(-4,3){0}}
		\put(27.00,26){\line(-5,3){7.00}}
		\put(20.00,29.67){\vector(1,-4){4.67}}
		\put(24.7,10.65){\vector(0,-1){7.00}}
		\put(13.33,17){\vector(-1,-2){3.33}}
		\put(12.33,15.00){\vector(3,-1){12.33}}
		\put(27.00,26.00){\line(-1,-6){2.22}}
		\put(27.00,26){\vector(1,4){0}}
		\put(10,10.65){\vector(1,-2){3.60}}
		\put(10,10.65){\vector(-1,-2){3.60}}
		\end{picture}}
\\
{}-{3}\hspace{-5mm}
\unitlength=.6mm
\special{em:linewidth 0.4pt}
\linethickness{0.4pt}
\raisebox{-25pt}[40pt][25pt]{
		\begin{picture}(42.00,30.00)(-8,0)
		\put(20.00,30.00){\vector(-1,-2){7.67}}
		\put(27,26){\vector(-4,-3){14.5}}
		\put(20.00,30.1){\vector(-4,3){0}}
		\put(27.00,26){\line(-5,3){7.00}}
		\put(20.00,29.67){\vector(1,-4){4.67}}
		\put(24.7,10.65){\vector(0,-1){7.00}}
		\put(12.33,15){\vector(0,-1){11.30}}
		\put(12.33,15.00){\vector(3,-1){12.33}}
		\put(27.00,26.00){\line(-1,-6){2.22}}
		\put(27.00,26){\vector(1,4){0}}
		\put(8.33,23.00){\vector(1,-2){4}}
		\put(8.33,23.00){\vector(-1,-2){9.66}}
		\end{picture}}
{}+{3}\!
	\raisebox{-25pt}[40pt][25pt]{
		\begin{picture}(40.00,30.00)(8,0)
		\put(20.00,30.00){\vector(-1,-2){7.67}}
		\put(27,26){\vector(-4,-3){14.5}}
		\put(20.00,30.1){\vector(-4,3){0}}
		\put(27.00,26){\line(-5,3){7.00}}
		\put(20.00,29.67){\vector(1,-4){4.67}}
		\put(24.7,10.65){\vector(0,-1){7.00}}
		\put(12.33,15){\vector(0,-1){11.30}}
		\put(12.33,15.00){\vector(3,-1){12.33}}
		\put(27.00,26.00){\line(-1,-6){2.22}}
		\put(27.00,26){\vector(1,4){0}}
		\put(28.7,18.65){\vector(1,-2){7.5}}
		\put(28.7,18.65){\vector(-1,-2){4}}
		\end{picture}}\hspace{-5mm}
{}+{3}\hspace{-5mm}
\raisebox{-25pt}[40pt][25pt]{
		\begin{picture}(42.00,30.00)(-8,0)
		\put(20.00,30.00){\vector(-1,-2){7.67}}
		\put(27,26){\vector(-4,-3){14.5}}
		\put(20.00,30.1){\vector(-4,3){0}}
		\put(27.00,26){\line(-5,3){7.00}}
		\put(20.00,29.67){\vector(1,-4){4.67}}
		\qbezier(24.7,10.65)(23.2,9.65)(21.7,8.65)
		\qbezier(17.7,6.65)(15,5.15)(12.33,3.65)
		\put(12.33,3.65){\vector(-4,-3){0.00}}
		\qbezier(12.33,15)(18.515,9.325)(24.7,3.65)
		\put(24.7,3.65){\vector(4,-3){0.00}}
		\put(12.33,15.00){\vector(3,-1){12.33}}
		\put(27.00,26.00){\line(-1,-6){2.22}}
		\put(27.00,26){\vector(1,4){0}}
		\put(16.00,38.00){\vector(1,-2){4}}
		\put(16.00,38.00){\vector(-1,-2){17.3}}
		\end{picture}}
{}-{3}\!\!
	\raisebox{-25pt}[40pt][25pt]{
		\begin{picture}(45.00,30.00)(8,0)
		\put(20.00,30.00){\vector(-1,-2){7.67}}
		\put(27,26){\vector(-4,-3){14.5}}
		\put(20.00,30.1){\vector(-4,3){0}}
		\put(27.00,26){\line(-5,3){7.00}}
		\put(20.00,29.67){\vector(1,-4){4.67}}
		\qbezier(24.7,10.65)(23.2,9.65)(21.7,8.65)
		\qbezier(17.7,6.65)(15,5.15)(12.33,3.65)
		\put(12.33,3.65){\vector(-4,-3){0.00}}
		\qbezier(12.33,15)(18.515,9.325)(24.7,3.65)
		\put(24.7,3.65){\vector(4,-3){0.00}}
		\put(12.33,15.00){\vector(3,-1){12.33}}
		\put(27.00,26.00){\line(-1,-6){2.22}}
		\put(27.00,26){\vector(1,4){0}}
		\put(31.00,34.00){\vector(1,-2){15.3}}
		\put(31.00,34.00){\vector(-1,-2){4}}
		\end{picture}}
\hspace{-10mm}
\raisebox{0mm}{\qquad$\Biggr\}.$}
\end{align*}
\caption{This sum of graphs is the skew-symmetrized content of Fig.~\ref{FigLHS}. In what follows, we realize these 9~terms in the left-hand side of~\eqref{EqWhichMechanism} by using an operator~$\Diamond$ acting, in the right-hand side of~\eqref{EqFactor} below, on the Ja\-co\-bi\-a\-tor~\eqref{EqJacFig}.}\label{Fig9Skew}
\end{figure}
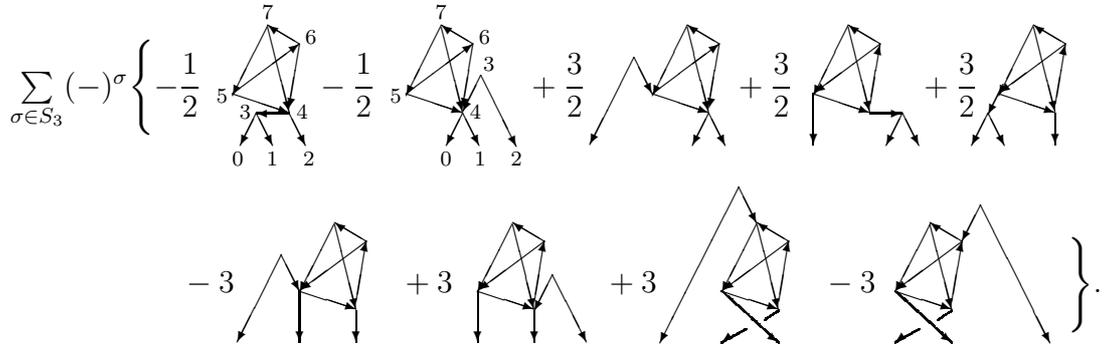
For example, 
when outlining a proof of our main theorem (see p.~\pageref{EqSol}), we shall explain how the coefficient~$-\frac{1}{2}$ of the first and second graphs in Fig.~\ref{Fig9Skew} is accumulated from the terms in the right-hand side of~\eqref{EqFactor}. Simultaneously, we shall track how the coefficients cancel out for the two other graphs which are produced by expanding the same Leibniz rules (that gave the above two graphs).

\subsection{Main result}
The reason why we are particularly concerned with the ratio $a:b = 1:6$ is that this condition is \emph{necessary} for equation \eqref{EqWhichMechanism} to hold.
This has been proved in~\cite{tetra16} by producing examples of Poisson bi\/-\/vector~$\cP$ such that $\schouten{\cP,\cQ_{a:b}(\cP)}=0$ only when $a:b=1:6$.
Let us now inspect whether this condition is also sufficient. 
The task is to factorize the content of Fig.~\ref{Fig9Skew} through the Jacobi identity in~\eqref{EqJacFig}.

We first examine the mechanism for the tri\/-\/vector $\schouten{\cP,\cQ_{1:6}(\cP)}$ in~\eqref{EqWhichMechanism} to vanish by virtue of the Jacobi identity $\Jac(\cP) 
=0$ for a given Poisson bi\/-\/vector~$\cP$ on an affine manifold~$N^n$ of any dimension. 
We claim that the Jacobiator $\Jac_\cP{}(\cdot,\cdot,\cdot)$ is not necessarily (indeed, far not always!\,) evaluated at the three arguments~$f,g,h$ of the tri\/-\/vector $\schouten{\cP,\cQ_{1:6}(\cP)}$.
A sample graph that can actually appear in such factorizing operators~$\Diamond$ is drawn in Fig.~\ref{FigSample} below.

\begin{lemma}[\cite{sqs15}]\label{Lemma}
A tri\/-\/differential operator $C = \sum_{|I|,|J|,|K|\geqslant 0} c^{IJK}\ \partial_I \otimes \partial_J \otimes \partial_K$ 
with coefficients $c^{IJK} \in C^\infty(N^n)$
vanishes identically iff all its homogeneous components 
$C_{ijk} = \sum_{|I|=i,|J|=j,|K|=k} c^{IJK}\ \partial_I \otimes \partial_J \otimes \partial_K$ 
vanish for all differential orders $(i,j,k)$ of the respective 
multi\/-\/indices $(I,J,K)$; here $\partial_L = \partial_1^{\alpha_1} \circ \cdots \circ \partial_n^{\alpha_n}$ for a multi\/-\/index $L = (\alpha_1, \ldots, \alpha_n)$. 
\end{lemma}

In practice, Lemma~\ref{Lemma} states that for every arrow falling on the Jacobiator $\Jac_{\cP}(\it1,\it2,\it3)$ --\,for which, in turn, a triple of arguments $\it1,\it2,\it3$
is specified\,-- the expansion of the 
Leibniz rule yields four fragments which vanish separately: e.g., we have that
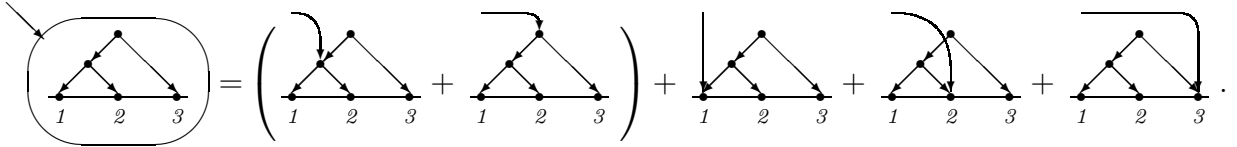
\begin{align*}
\text{\raisebox{-12pt}[30pt]{
	\unitlength=0.7mm
	\linethickness{0.4pt}
	\begin{picture}(26.00,16.33)
	\put(0.00,5.00){\line(1,0){26.00}}
	\put(2.00,5.00){\circle*{1.33}}
	\put(13.00,5.00){\circle*{1.33}}
	\put(24.00,5.00){\circle*{1.33}}
	\put(2.00,1.33){\makebox(0,0)[cc]{\tiny\it1}}
	\put(13.00,1.33){\makebox(0,0)[cc]{\tiny\it2}}
	\put(24.00,1.33){\makebox(0,0)[cc]{\tiny\it3}}
	\put(7.33,11.33){\circle*{1.33}}
	\put(7.33,11.33){\vector(1,-1){5.5}}
	\put(7.33,11.33){\vector(-1,-1){5.5}}
	\put(13,17){\circle*{1.33}}
	\put(13,17){\vector(1,-1){11.2}}
	\put(13,17){\vector(-1,-1){5.1}}
	\put(13,8){\oval(34,24)}
	\put(-8.00,23.00){\vector(1,-1){7.00}}
	\end{picture}
}}\ \ 
{=}
\left(\!\!\!
\text{\raisebox{-12pt}[25pt]{
	\unitlength=0.70mm
	\linethickness{0.4pt}
	\begin{picture}(26.00,16.33)
	\put(0.00,5.00){\line(1,0){26.00}}
	\put(2.00,5.00){\circle*{1.33}}
	\put(13.00,5.00){\circle*{1.33}}
	\put(24.00,5.00){\circle*{1.33}}
	\put(2.00,1.33){\makebox(0,0)[cc]{\tiny\it1}}
	\put(13.00,1.33){\makebox(0,0)[cc]{\tiny\it2}}
	\put(24.00,1.33){\makebox(0,0)[cc]{\tiny\it3}}
	\put(7.33,11.33){\circle*{1.33}}
	\put(7.33,11.33){\vector(1,-1){5.5}}
	\put(7.33,11.33){\vector(-1,-1){5.5}}
	\put(13,17){\circle*{1.33}}
	\put(13,17){\vector(1,-1){11.2}}
	\put(13,17){\vector(-1,-1){5.1}}
	\bezier{80}(2.00,21.00)(7.33,21.00)(7.33,15.33)
	\put(7.33,15.33){\line(0,-1){2.00}}
	\put(7.33,12.00){\vector(0,-1){0.00}}
	\end{picture}
}}
{+}\!
\text{\raisebox{-12pt}[25pt]{
	\unitlength=0.70mm
	\linethickness{0.4pt}
	\begin{picture}(26.00,16.33)
	\put(0.00,5.00){\line(1,0){26.00}}
	\put(2.00,5.00){\circle*{1.33}}
	\put(13.00,5.00){\circle*{1.33}}
	\put(24.00,5.00){\circle*{1.33}}
	\put(2.00,1.33){\makebox(0,0)[cc]{\tiny\it1}}
	\put(13.00,1.33){\makebox(0,0)[cc]{\tiny\it2}}
	\put(24.00,1.33){\makebox(0,0)[cc]{\tiny\it3}}
	\put(7.33,11.33){\circle*{1.33}}
	\put(7.33,11.33){\vector(1,-1){5.5}}
	\put(7.33,11.33){\vector(-1,-1){5.5}}
	\put(13,17){\circle*{1.33}}
	\put(13,17){\vector(1,-1){11.2}}
	\put(13,17){\vector(-1,-1){5.1}}
	\put(2.00,21.00){\line(1,0){7.00}}
	\bezier{80}(9.00,21.00)(13.00,21.00)(13.00,19.00)
	\put(13.00,17.50){\vector(0,-1){0.00}}
	\end{picture}
}}
\right)
{+}\!
\text{\raisebox{-12pt}[25pt]{
	\unitlength=0.70mm
	\linethickness{0.4pt}
	\begin{picture}(26.00,16.33)
	\put(0.00,5.00){\line(1,0){26.00}}
	\put(2.00,5.00){\circle*{1.33}}
	\put(13.00,5.00){\circle*{1.33}}
	\put(24.00,5.00){\circle*{1.33}}
	\put(2.00,1.33){\makebox(0,0)[cc]{\tiny\it1}}
	\put(13.00,1.33){\makebox(0,0)[cc]{\tiny\it2}}
	\put(24.00,1.33){\makebox(0,0)[cc]{\tiny\it3}}
	\put(7.33,11.33){\circle*{1.33}}
	\put(7.33,11.33){\vector(1,-1){5.5}}
	\put(7.33,11.33){\vector(-1,-1){5.5}}
	\put(13,17){\circle*{1.33}}
	\put(13,17){\vector(1,-1){11.2}}
	\put(13,17){\vector(-1,-1){5.1}}
	\put(2.00,21.00){\vector(0,-1){15}}
	\end{picture}
}}
{+}\!
\text{\raisebox{-12pt}[25pt]{
	\unitlength=0.70mm
	\linethickness{0.4pt}
	\begin{picture}(26.00,16.33)
	\put(0.00,5.00){\line(1,0){26.00}}
	\put(2.00,5.00){\circle*{1.33}}
	\put(13.00,5.00){\circle*{1.33}}
	\put(24.00,5.00){\circle*{1.33}}
	\put(2.00,1.33){\makebox(0,0)[cc]{\tiny\it1}}
	\put(13.00,1.33){\makebox(0,0)[cc]{\tiny\it2}}
	\put(24.00,1.33){\makebox(0,0)[cc]{\tiny\it3}}
	\put(7.33,11.33){\circle*{1.33}}
	\put(7.33,11.33){\vector(1,-1){5.5}}
	\put(7.33,11.33){\vector(-1,-1){5.5}}
	\put(13,17){\circle*{1.33}}
	\put(13,17){\vector(1,-1){11.2}}
	\put(13,17){\vector(-1,-1){5.1}}
	\bezier{160}(2.00,21.00)(13.00,21.00)(13.00,10.00)
	\put(13.00,10.00){\vector(0,-1){4.00}}
	\end{picture}
}}
{+}\!
\text{\raisebox{-12pt}[25pt]{
	\unitlength=0.70mm
	\linethickness{0.4pt}
	\begin{picture}(26.00,16.33)
	\put(0.00,5.00){\line(1,0){26.00}}
	\put(2.00,5.00){\circle*{1.33}}
	\put(13.00,5.00){\circle*{1.33}}
	\put(24.00,5.00){\circle*{1.33}}
	\put(2.00,1.33){\makebox(0,0)[cc]{\tiny\it1}}
	\put(13.00,1.33){\makebox(0,0)[cc]{\tiny\it2}}
	\put(24.00,1.33){\makebox(0,0)[cc]{\tiny\it3}}
	\put(7.33,11.33){\circle*{1.33}}
	\put(7.33,11.33){\vector(1,-1){5.5}}
	\put(7.33,11.33){\vector(-1,-1){5.5}}
	\put(13,17){\circle*{1.33}}
	\put(13,17){\vector(1,-1){11.2}}
	\put(13,17){\vector(-1,-1){5.1}}
	\put(2.00,21.00){\line(1,0){18.00}}
	\bezier{80}(20.00,21.00)(24.00,21.00)(24.00,17.00)
	\put(24.00,17.00){\vector(0,-1){11}}
	\end{picture}
}}.
\end{align*}
Namely, there is the fragment such that the derivation acts on the content $\cP$ of the Jacobiator's two internal vertices, and there are three fragments such that the arrow falls on the first, second, or third argument of the Jacobiator.
Now it is readily seen that the action of a derivative~$\dd_i$ on an argument of the Jacobiator 
amounts to an appropriate redefinition of that 
argument: $\dd_i\bigl(\Jac_\cP(\it1,\it2,\it3)\bigr)=$
\[
\underbrace{\bigl(\dd_i\Jac_\cP\bigr)(\mathit{1},\mathit{2},\mathit{3})}_{=0}+
\underbrace{\Jac_\cP\bigl(\dd_i(\mathit{1}),\mathit{2},\mathit{3}\bigr)}_{=0}+
\underbrace{\Jac_\cP\bigl(\mathit{1},\dd_i(\mathit{2}),\mathit{3}\bigr)}_{=0}+
\underbrace{\Jac_\cP\bigl(\mathit{1},\mathit{2},\dd_i(\mathit{3})\bigr)}_{=0}=0.
\]
Let us introduce a name for the (class of) graphs which make the first term --\,out of four\,-- in the expansion of Leibniz rule in the above formula.

\begin{define}
A \emph{Leibniz graph} is a graph whose vertices are either sinks, or the sources for two arrows, 
or the Jacobiator (which is a source for three arrows). 
There must be at least one Jacobiator vertex.
The three arrows originating from a Jacobiator vertex must land on three distinct vertices.
Each edge falling on a Jacobiator works by the Leibniz rule on the two internal vertices in~it.
\end{define}

An example of a Leibniz graph is given in Fig.~\ref{FigSample}.
Every Leibniz graph can be expanded to a sum of Kontsevich graphs, by expanding both the Leibniz rule(s) and all copies of the Jacobiator; e.g. see \eqref{EqStayVanish}.
In this way Leibniz graphs also encode (poly)differential operators, depending on the bi\/-\/vector $\cP$ and the tri\/-\/vector $\Jac(\cP)$.
\begin{figure}[htb]
\begin{minipage}{0.3\textwidth}
\begin{align*}
\unitlength=1mm
\special{em:linewidth 0.4pt}
\linethickness{0.4pt}
\begin{picture}(40.67,35.00)
\put(15.00,20.00){\framebox(20.00,10.00)[cc]{$\bullet\quad\bullet$}}
\put(25.00,20.00){\vector(0,-1){15.00}}
\put(18.00,20.00){\vector(-1,-2){5.00}}
\put(32.00,20.00){\vector(1,-3){5.00}}
\put(13.00,10.00){\vector(0,-1){5.00}}
\put(13.00,10.00){\vector(-1,0){8.00}}
\put(13.00,0.00){\makebox(0,0)[cb]{\tiny(\ )}}
\put(25.00,0.00){\makebox(0,0)[cb]{\tiny(\ )}}
\put(5.00,10.00){\vector(0,1){8.00}}
\put(5.00,18.00){\vector(1,-1){8.00}}
\put(37.00,0.00){\makebox(0,0)[cb]{\tiny(\ )}}
\bezier{100}(5.00,10.00)(9.00,5.00)(13.00,10.00)
\put(13.00,10.00){\vector(1,1){0.00}}
\put(5.00,18.00){\line(0,1){12}}
\bezier{80}(5.00,30.00)(5.00,35.00)(10.00,35.00)
\put(10.00,35.00){\line(1,0){5.00}}
\bezier{80}(15.00,35.00)(20.00,35.00)(20.00,30.00)
\put(20.00,30.00){\vector(0,-1){0.00}}
\put(13.00,10.00){\circle*{1}}
\put(5.00,10.00){\circle*{1}}
\put(5.00,18.00){\circle*{1}}
\end{picture}
\end{align*}
\end{minipage}
\small
\begin{minipage}{0.4\textwidth}
\begin{itemize}
\item There is a cycle, 
\item there is a loop, 
\item there are no tadpoles in this graph,
\item an arrow falls back on $\Jac(\cP)$, 
\item and $\Jac(\cP)$ does not stand on all of the three sinks.
\end{itemize}
\end{minipage}
\normalsize
\caption{This is an example of Leibniz graph of which the factorizing operators
can consist.}\label{FigSample}
\end{figure}

\begin{proposition}\label{PropLeibnizGraphZero}
For every Poisson bi\/-\/vector~$\cP$ the value --\,at the Jacobiator $\Jac(\cP)$\,-- of every (poly)\/dif\-fe\-ren\-ti\-al operator encoded by the Leibniz graph(s) is zero.
\end{proposition}

\begin{theor}\label{ThMain}
There exists a polydifferential operator 
\[
\Diamond \in \text{\textup{PolyDiff}}\,\bigl(\Gamma(\bigwedge\nolimits^{2}TN^{n})\times \Gamma(\bigwedge\nolimits^{3}TN^{n}) \rightarrow \Gamma(\bigwedge\nolimits^{3}TN^{n}) \bigr)
\]
which solves the factorization problem
\begin{equation}\label{EqFactor}
\schouten{\cP,\cQ_{1:6}(\cP)}(f,g,h)=\Diamond\,\bigl(\cP,\Jac_\cP{}(\cdot,\cdot,\cdot)\bigr)(f,g,h).
\end{equation}
The 
polydifferential operator~$\Diamond$ is realised using Leibniz graphs in
formula~\eqref{EqSol}, see p.~\textup{\pageref{pSolution}} below.
\end{theor}

\begin{cor}[Main result]\label{CorMain}
Whenever a bi-vector $\cP$ on an affine real manifold $N^{n}$ is Poisson, the deformation $\cP + \varepsilon \cQ_{1:6}(\cP)+ \bar{o}(\varepsilon
)$ using the Kontsevich tetrahedral flow is infinitesimally Poisson.
\end{cor}


\begin{rem}
It is readily seen that the Kontsevich tetrahedral flow $\dot{\cP}=\cQ_{1:6}(\cP)$ is well defined on the space of Poisson bi\/-\/vectors on a given affine manifold~$N^n$. Indeed, it does not depend on a choice of coordinates up to their arbitrary affine reparametrisations. In other words, the velocity $\dot{\cP}{\bigr|}_{\boldsymbol{u}\in N^n}$ does not depend on the choice of a chart $\mathcal{U}\ni\boldsymbol{u}$ from an atlas in which only \emph{affine} changes of variables are allowed. (Let us remember that affine manifolds can of course be topologically nontrivial.) 

Suppose however that a given affine structure on the manifold~$N^n$ is extended to a larger atlas on it; for the sake of definition let that atlas be a smooth one. Assume that the smooth structure is now reduced --\,by discarding a number of charts\,-- to another affine structure on the same manifold. The tetrahedral flow $\dot{\cP}=\cQ_{1:6}(\cP)$ which one initially had can be contrasted with the tetrahedral flow $\dot{\tilde{\cP}}=\cQ_{1:6}(\tilde{\cP})$ which one finally obtains for the Poisson bi\/-\/vector~$\smash{\tilde{\cP}{\bigr|}_{\tilde{\boldsymbol{u}}(\boldsymbol{u})} = \cP{\bigr|}_{\boldsymbol{u}}}$ in the course of a nonlinear change of coordinates on~$N^n$. Indeed, the respective velocities~$\dot{\cP}$ and~$\dot{\tilde{\cP}}$ can be different whenever they are 
expressed by using 
essentially different parametrisations of a neighbourhood of a point~$\boldsymbol{u}$ in~$N^n$. For example, the tetrahedral flow vanishes identically when expressed in the Darboux canonical variables on a chart in a symplectic manifold. But after a nonlinear 
transformation, the right\/-\/hand side $\cQ_{1:6}(\tilde{\cP})$ can become nonzero at the same points of that Darboux chart.

This shows that an affine structure on the manifold~$N^n$
is a necessary part of the input data for construction of the Kontsevich tetrahedral flows~$\dot{\cP}=\cQ_{1:6}(\cP)$.
\end{rem}

\section{Solution of the factorization problem}\label{SecProof}
\noindent%
Expanding the Leibniz rules in $\schouten{\cP,\cQ_{1:6}(\cP)}$, we obtain the sum of $39$~graphs with $5$~internal vertices and $3$~sinks (so that from Figure~\ref{FigLHS} we produce Table~\ref{TableLHS}, see page~\pageref{TableLHS} below). By construction, the Schouten bracket $\schouten{\cP,\cQ_{1:6}(\cP)} \in \Gamma(\bigwedge^3 TN^n)$ is a tri\/-\/vector on the underlying manifold $N^n$, that is, it is a totally antisymmetric tri\/-\/linear polyderivation $C^\infty(N^n) \times C^\infty(N^n) \times C^\infty(N^n) \to C^\infty(N^n)$. 
At the same time, we seek to recognize the tri\/-\/vector $\schouten{\cP,\cQ_{1:6}(\cP)}$ as the result of application of a (poly)\/differential operator~$\Diamond$ (see~\eqref{EqFactor} in Theorem~\ref{ThMain}) to the Jacobiator~$\Jac(\cP)$ (see~\eqref{EqJacFig} on p.~\pageref{EqJacFig}). 
%

We now explain how the operator~$\Diamond$ is found.\footnote{Another method for solving the factorization problem is outlined in Appendix~\ref{AppPerturb}.}
The ansatz for $\Diamond$ is the sum -- with undetermined coefficients -- of all (separately vanishing) Leibniz graphs containing one Jacobiator and three wedges, and having differential order $(1,1,1)$ with respect to the sinks (see Fig.~\ref{FigSomeLeibniz}). 
We thus have $28,202$ unknowns introduced (counted with possible repetitions of graphs which they refer to).  
Expanding all the Leibniz rules and Jacobiators, we obtain a sum of Kontsevich graphs with $5$ internal vertices on $3$ sinks.
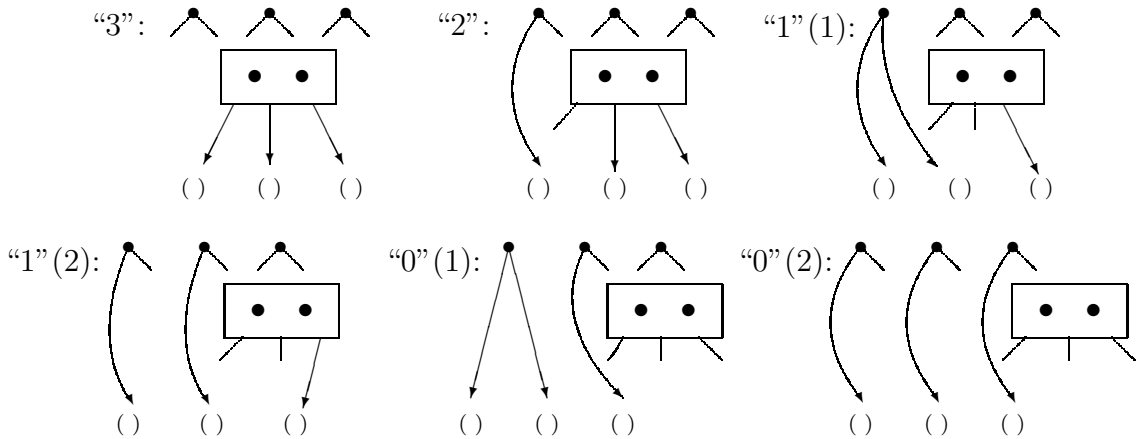
\begin{figure}[htb]
\begin{center}
\unitlength=1mm
\special{em:linewidth 0.4pt}
\linethickness{0.4pt}
\begin{picture}(33.00,30.67)
\put(0.00,30.00){\makebox(0,0)[ct]{``3":}}
\put(10.00,30.00){\circle*{1.33}}
\put(20.00,30.00){\circle*{1.33}}
\put(30.00,30.00){\circle*{1.33}}
\bezier{16}(10.00,30.00)(8.00,28.00)(7.00,27.00)
\bezier{16}(10.00,30.00)(12.00,28.00)(13.00,27.00)
\bezier{16}(20.00,30.00)(18.00,28.00)(17.00,27.00)
\bezier{16}(20.00,30.00)(22.00,28.00)(23.00,27.00)
\bezier{16}(30.00,30.00)(28.00,28.00)(27.00,27.00)
\bezier{16}(30.00,30.00)(32.00,28.00)(33.00,27.00)
\put(13.67,18.00){\framebox(15,7.00)[cc]{$\bullet\quad\bullet$}}
\put(15.33,18.00){\vector(-1,-2){4.00}}
\put(20.00,18.00){\vector(0,-1){8.00}}
\put(25.67,18.00){\vector(1,-2){4.00}}
\put(10.00,5.33){\makebox(0,0)[cb]{\tiny(\ )}}
\put(20.00,5.33){\makebox(0,0)[cb]{\tiny(\ )}}
\put(30.67,5.33){\makebox(0,0)[cb]{\tiny(\ )}}
\end{picture}
\qquad
\unitlength=1.00mm
\special{em:linewidth 0.4pt}
\linethickness{0.4pt}
\begin{picture}(33.00,30.66)
\put(0.00,30.00){\makebox(0,0)[ct]{``2":}}
\put(10.00,30.00){\circle*{1.33}}
\put(20.00,30.00){\circle*{1.33}}
\put(30.00,30.00){\circle*{1.33}}
\bezier{16}(10.00,30.00)(12.00,28.00)(13.00,27.00)
\bezier{16}(20.00,30.00)(18.00,28.00)(17.00,27.00)
\bezier{16}(20.00,30.00)(22.00,28.00)(23.00,27.00)
\bezier{16}(30.00,30.00)(28.00,28.00)(27.00,27.00)
\bezier{16}(30.00,30.00)(32.00,28.00)(33.00,27.00)
\put(14.33,18.00){\framebox(15,7.00)[cc]{$\bullet\quad\bullet$}}
\put(20.00,18.00){\vector(0,-1){9.00}}
\put(25.67,18.00){\vector(1,-2){4.00}}
\put(10.00,5.33){\makebox(0,0)[cb]{\tiny(\ )}}
\put(20.00,5.33){\makebox(0,0)[cb]{\tiny(\ )}}
\put(30.67,5.33){\makebox(0,0)[cb]{\tiny(\ )}}
\bezier{96}(10.00,30.00)(3.00,20.00)(10.00,10.00)
\put(9.67,10.67){\vector(2,-3){0.67}}
\bezier{16}(15.00,18.00)(13.00,15.67)(12.00,14.67)
\end{picture}
\qquad
\unitlength=1.00mm
\special{em:linewidth 0.4pt}
\linethickness{0.4pt}
\begin{picture}(33.00,30.66)
\put(0.00,30.00){\makebox(0,0)[ct]{``1"(1):}}
\put(10.00,30.00){\circle*{1.33}}
\put(20.00,30.00){\circle*{1.33}}
\put(30.00,30.00){\circle*{1.33}}
\bezier{16}(20.00,30.00)(18.00,28.00)(17.00,27.00)
\bezier{16}(20.00,30.00)(22.00,28.00)(23.00,27.00)
\bezier{16}(30.00,30.00)(28.00,28.00)(27.00,27.00)
\bezier{16}(30.00,30.00)(32.00,28.00)(33.00,27.00)
\put(16.00,18.00){\framebox(15,7.00)[cc]{$\bullet\quad\bullet$}}
\put(25.67,18.00){\vector(1,-2){4.67}}
\put(10.00,5.33){\makebox(0,0)[cb]{\tiny(\ )}}
\put(20.00,5.33){\makebox(0,0)[cb]{\tiny(\ )}}
\put(30.67,5.33){\makebox(0,0)[cb]{\tiny(\ )}}
\bezier{96}(10.00,30.00)(3.00,20.00)(10.00,10.00)
\put(9.67,10.67){\vector(2,-3){0.67}}
\bezier{16}(19.00,18.00)(17.00,15.67)(16.00,14.67)
\bezier{13}(22.00,18.00)(22.00,15.67)(22.00,14.67)
\bezier{100}(10.00,30.00)(8.67,20.00)(17.00,9.67)
\put(16.33,10.33){\vector(1,-1){0.67}}
\end{picture}
\unitlength=1mm
\special{em:linewidth 0.4pt}
\linethickness{0.4pt}
\begin{picture}(37.67,30.67)
\put(0.00,30.00){\makebox(0,0)[ct]{``1"(2):}}
\put(10.00,30.00){\circle*{1.33}}
\put(20.00,30.00){\circle*{1.33}}
\put(30.00,30.00){\circle*{1.33}}
\put(22.67,18.00){\framebox(15.00,7.00)[cc]{$\bullet\quad\bullet$}}
\put(10.00,5.00){\makebox(0,0)[cb]{\tiny(\ )}}
\put(21.00,5.00){\makebox(0,0)[cb]{\tiny(\ )}}
\put(32.00,5.00){\makebox(0,0)[cb]{\tiny(\ )}}
\bezier{88}(10.00,30.00)(5.00,18.00)(10.00,10.00)
\bezier{16}(10.00,30.00)(12.00,28.00)(13.00,27.00)
\bezier{88}(20.00,30.00)(15.00,18.00)(20.00,10.00)
\bezier{16}(20.00,30.00)(22.00,28.00)(23.00,27.00)
\bezier{16}(30.00,30.00)(28.00,28.00)(27.00,27.00)
\bezier{16}(30.00,30.00)(32.00,28.00)(33.00,27.00)
\bezier{16}(25.00,18.00)(23.00,16.00)(22.00,15.00)
\bezier{12}(30.00,18.00)(30.00,16.00)(30.00,15.00)
\put(35.33,18.00){\vector(-1,-4){2.33}}
\put(10.33,9.67){\vector(1,-2){0.33}}
\put(20.33,9.67){\vector(1,-2){0.33}}
\end{picture}
\qquad
\unitlength=1mm
\special{em:linewidth 0.4pt}
\linethickness{0.4pt}
\begin{picture}(38.00,30.67)
\put(0.00,30.00){\makebox(0,0)[ct]{``0"(1):}}
\put(10.00,30.00){\circle*{1.33}}
\put(20.00,30.00){\circle*{1.33}}
\put(30.00,30.00){\circle*{1.33}}
\put(5.00,5.00){\makebox(0,0)[cb]{\tiny(\ )}}
\put(15.00,5.00){\makebox(0,0)[cb]{\tiny(\ )}}
\put(25.00,5.00){\makebox(0,0)[cb]{\tiny(\ )}}
\put(10.00,30.00){\vector(-1,-4){5.00}}
\put(10.00,30.00){\vector(1,-4){5.00}}
\bezier{100}(20.00,30.00)(15.00,20.00)(25.00,10.00)
\put(24.33,10.67){\vector(1,-1){0.67}}
\put(23.00,18.00){\framebox(15,7.00)[cc]{$\bullet\quad\bullet$}}
\bezier{16}(20.00,30.00)(22.00,28.00)(23.00,27.00)
\bezier{16}(30.00,30.00)(28.00,28.00)(27.00,27.00)
\bezier{16}(30.00,30.00)(32.00,28.00)(33.00,27.00)
\bezier{16}(25.00,18.00)(24.00,16.00)(23.00,15.00)
\bezier{12}(30.00,18.00)(30.00,16.00)(30.00,15.00)
\bezier{16}(35.00,18.00)(36.67,16.33)(38.00,15.00)
\end{picture}
\quad
\unitlength=1mm
\special{em:linewidth 0.4pt}
\linethickness{0.4pt}
\begin{picture}(46.00,30.67)
\put(0.00,30.00){\makebox(0,0)[ct]{``0"(2):}}
\put(10.00,30.00){\circle*{1.33}}
\put(20.00,30.00){\circle*{1.33}}
\put(30.00,30.00){\circle*{1.33}}
\put(10.00,5.00){\makebox(0,0)[cb]{\tiny(\ )}}
\put(20.00,5.00){\makebox(0,0)[cb]{\tiny(\ )}}
\put(30.00,5.00){\makebox(0,0)[cb]{\tiny(\ )}}
\bezier{104}(10.00,30.00)(2.00,20.00)(10.00,10.00)
\bezier{104}(20.00,30.00)(12.00,20.00)(20.00,10.00)
\bezier{104}(30.00,30.00)(22.00,20.00)(30.00,10.00)
\put(10.00,10.00){\vector(1,-2){0.33}}
\put(20.00,10.00){\vector(1,-2){0.33}}
\put(30.00,10.00){\vector(1,-2){0.33}}
\bezier{16}(10.00,30.00)(12.00,28.00)(13.00,27.00)
\bezier{16}(20.00,30.00)(22.00,28.00)(23.00,27.00)
\bezier{16}(30.00,30.00)(32.00,28.00)(33.00,27.00)
\put(30.00,18.00){\framebox(15.00,7.00)[cc]{$\bullet\quad\bullet$}}
\bezier{12}(37.33,18.00)(37.33,16.00)(37.33,15.00)
\bezier{16}(43.00,18.00)(45.00,16.00)(46.00,15.00)
\bezier{16}(31.67,18.00)(29.67,16.00)(28.67,15.00)
\end{picture}
\end{center}
\caption{This is the list of all different types of Leibniz graphs 
which are linear in the Jacobiator and which 
have differential order $(1,1,1)$
with respect to the sinks. The list is ordered by the number of ground vertices on which the Jacobiator stands.
%
}\label{FigSomeLeibniz}
\end{figure}

As soon as we take into account the order $L \prec R$ and the antisymmetry of graphs under the reversion of that ordering at an internal vertex, the graphs that encode zero differential operators 
are eliminated.\footnote{The relevant algebra of sums of graphs modulo skew\/-\/symmetry and the Jacobi identity has been realized in software by the second author. 
An 
implementation 
of those tools in the problem of high\/-\/order expansion of the Kontsevich 
$\star$-\/product is explained in a separate paper, see~\cite{cpcOrder4}.}
There remain $7,025$ admissible graphs with $5$ internal vertices on $3$ sinks; the coefficient of every such graph is a linear combination of the undetermined coefficients of the Leibniz graphs.
In conclusion, we view \eqref{EqFactor} as the system of $7,025$ linear inhomogeneous equations for the coefficients of Leibniz graphs in the operator $\Diamond$.
Solving this linear system is a way towards a proof of our main result (which is expressed in Corollary~\ref{CorMain}).
The process of finding a solution~$\Diamond$ itself does not constitute that proof.
Therefore, the justification of the claim in Theorem~\ref{ThMain} will be performed separately.
In the meantime, using software tools, we solve the 
linear algebraic system at hand. The duplications of graph labellings are conveniently eliminated by our request for the program to find a solution with a minimal number of nonzero components.
Totally antisymmetric in tri-vector's arguments, the solution consists of $27$ Leibniz graphs, which are assimilated 
into the sum of $8$ manifestly skew\/-\/symmetric terms as follows:\label{pSolution}
\begin{multline}
\raisebox{17mm}{$\Diamond\quad=\quad$}
\unitlength=1mm
\special{em:linewidth 0.4pt}
\linethickness{0.4pt}
\begin{picture}(20.33,35.67)
\put(0.00,10.00){\framebox(20.00,10.00)[cc]{$\bullet\quad\bullet$}}
\put(3.00,10.00){\vector(-2,-3){3.33}}
\put(10.00,10.00){\vector(0,-1){5.00}}
\put(17.00,10.00){\vector(2,-3){3.33}}
\put(0.00,0.00){\makebox(0,0)[cb]{\tiny(\ )}}
\put(10.00,0.00){\makebox(0,0)[cb]{\tiny(\ )}}
\put(20.00,0.00){\makebox(0,0)[cb]{\tiny(\ )}}
\put(3.00,35.00){\vector(0,-1){15.00}}
\put(3.00,35.00){\vector(1,-1){7.00}}
\put(10.00,28.00){\vector(1,1){7.00}}
\put(17.00,35.00){\vector(-1,0){14.00}}
\put(10.00,28.00){\vector(0,-1){8.00}}
\put(17.00,35.00){\vector(0,-1){15.00}}
\put(17.00,35.00){\circle*{1}}
\put(3.00,35.00){\circle*{1}}
\put(10.00,28.00){\circle*{1}}
\end{picture}
\quad
\raisebox{17mm}{${}+ 3 \sum\limits_{\tau\in S_2}(-)^\tau$}
\quad
%
\unitlength=1mm
\special{em:linewidth 0.4pt}
\linethickness{0.4pt}
\begin{picture}(23.33,35.00)
\put(0.00,10.00){\vector(0,-1){5.00}}
\put(20.00,10.00){\vector(-1,0){20.00}}
\put(10.00,11.00){\line(0,1){9}}
\put(10.00,9.00){\vector(0,-1){4.00}}
\put(20.00,10.00){\vector(0,-1){5.00}}
\put(0.00,10.00){\vector(1,1){10.00}}
\put(10.00,20.00){\vector(1,-1){9.20}}
\put(15.00,25.00){\vector(1,-3){5.00}}
\put(10.00,25.00){\vector(0,-1){5.00}}
\put(5.00,25.00){\vector(-1,-3){5.00}}
\put(0.00,25.00){\framebox(20.00,10.00)[cc]{$\bullet\quad\bullet$}}
\put(-1.67,0.33){\llap{[}}
\put(0.00,0.00){\makebox(0,0)[cb]{\tiny(\ )}}
\put(10.00,0.00){\makebox(0,0)[cb]{\tiny(\ )}}
\put(11.67,0.33){]}
\put(20,0.00){\makebox(0,0)[cb]{\tiny(\ )}}
\put(20.00,10.00){\circle*{1}}
\put(0.00,10.00){\circle*{1}}
\put(10.00,20.00){\circle*{1}}
\end{picture}
\quad
\raisebox{17mm}{${}+ 3 \sum\limits_{\circlearrowright}$}
\quad
%
\unitlength=1mm
\special{em:linewidth 0.4pt}
\linethickness{0.4pt}
\begin{picture}(23.33,35.00)
\put(0.00,10.00){\vector(0,-1){5.00}}
\put(20.00,10.00){\vector(-1,0){20.00}}
\put(0.00,10.00){\vector(2,-1){10.00}}
\put(20.00,10.00){\vector(0,-1){5.00}}
\put(10.00,20.00){\vector(-1,-1){9.2}}
\put(10.00,20.00){\vector(1,-1){9.2}}
\put(15.00,25.00){\vector(1,-3){5.00}}
\put(10.00,25.00){\vector(0,-1){5.00}}
\put(5.00,25.00){\vector(-1,-3){5.00}}
\put(0.00,25.00){\framebox(20.00,10.00)[cc]{$\bullet\quad\bullet$}}
\put(0.00,0.00){\makebox(0,0)[cb]{\tiny(\ )}}
\put(10.00,0.00){\makebox(0,0)[cb]{\tiny(\ )}}
\put(20,0.00){\makebox(0,0)[cb]{\tiny(\ )}}
\put(20.00,10.00){\circle*{1}}
\put(0.00,10.00){\circle*{1}}
\put(10.00,20.00){\circle*{1}}
\end{picture}
%
\\
\raisebox{17mm}{${}+ 3\sum\limits_{\circlearrowright}\Biggl\{$}
\hspace{-10mm}
\unitlength=1mm
\special{em:linewidth 0.4pt}
\linethickness{0.4pt}
\begin{picture}(40.67,35.00)
\put(30.00,10.00){\vector(1,-1){5.00}}
\put(40.00,10.00){\vector(-1,0){10.00}}
\put(35.00,15.00){\vector(-1,-1){5.00}}
\put(35.00,15.00){\vector(1,-1){5.00}}
\put(30.00,10.00){\line(0,1){10.00}}
\put(15.00,20.00){\framebox(20.00,10.00)[cc]{$\bullet\quad\bullet$}}
\put(32.00,20.00){\vector(1,-2){2.33}}
\put(25.00,20.00){\vector(0,-1){15.00}}
\put(18.00,20.00){\vector(-1,-3){5.00}}
\put(13.00,0.00){\makebox(0,0)[cb]{\tiny(\ )}}
\put(25.00,0.00){\makebox(0,0)[cb]{\tiny(\ )}}
\put(36.00,0.00){\makebox(0,0)[cb]{\tiny(\ )}}
\put(40.00,10.00){\line(0,1){20.00}}
\bezier{15}(30.00,20.00)(30.00,25)(30.00,30.00)
\bezier{40}(40.00,30.00)(40.00,35.00)(37,35.00)
\bezier{40}(37,35.00)(33.00,35.00)(33.00,30.00)
\put(33.00,31.00){\vector(0,-1){1.00}}
\bezier{40}(30.00,30.00)(30.00,35.00)(25.00,35.00)
\bezier{40}(25.00,35.00)(20.00,35.00)(20.00,30.00)
\put(20.00,31.00){\vector(0,-1){1.00}}
\put(35.00,15.00){\circle*{1}}
\put(30.00,10.00){\circle*{1}}
\put(40.00,10.00){\circle*{1}}
\end{picture}
\quad
\raisebox{17mm}{$+$}
\quad
%
\unitlength=1mm
\special{em:linewidth 0.4pt}
\linethickness{0.4pt}
\begin{picture}(30.00,32.00)
\put(0.00,15.00){\framebox(20.00,10.00)[cc]{$\bullet\quad\bullet$}}
\put(15.00,10.00){\vector(1,-1){5.00}}
\put(25.00,5.00){\vector(-2,1){10.00}}
\put(15.00,10.00){\vector(2,1){10.00}}
\put(10.00,15.00){\vector(0,-1){10.00}}
\put(3.00,15.00){\vector(-1,-3){3.33}}
\put(15.00,15.00){\vector(0,-1){4.67}}
\put(25.00,15.00){\line(0,1){10.00}}
\put(25.00,15.00){\vector(0,-1){10.00}}
\bezier{40}(25.00,25.00)(25.00,30.00)(20.00,30.00)
\bezier{40}(20.00,30.00)(15.00,30.00)(15.00,25.00)
\put(15.00,26.00){\vector(0,-1){1.00}}
\bezier{40}(25.00,5.00)(30.00,5.00)(30.00,10.00)
\put(30.00,10.00){\line(0,1){17.00}}
\bezier{40}(30.00,27.00)(30.00,32.00)(25.00,32.00)
\put(25.00,32.00){\line(-1,0){15.00}}
\bezier{48}(10.00,32.00)(5.00,32.00)(5.00,25.00)
\put(5.00,26.00){\vector(0,-1){1.00}}
\put(15.00,10.00){\circle*{1}}
\put(25.00,15.00){\circle*{1}}
\put(25.00,5.00){\circle*{1}}
\put(0.00,0.00){\makebox(0,0)[cb]{\tiny(\ )}}
\put(10.00,0.00){\makebox(0,0)[cb]{\tiny(\ )}}
\put(20.00,0.00){\makebox(0,0)[cb]{\tiny(\ )}}
\end{picture}
\quad
\raisebox{17mm}{$+$}
\unitlength=1mm
\special{em:linewidth 0.4pt}
\linethickness{0.4pt}
\begin{picture}(30.00,35.00)
\put(15.00,10.00){\vector(0,-1){5.5}}
\put(15.00,10.00){\vector(-4,-3){7.00}}
\put(25.00,5.00){\vector(-2,1){10.00}}
\put(25.00,15.00){\vector(-2,-1){10.00}}
\put(25.00,15.00){\vector(0,-1){10.00}}
\put(15.00,20.00){\vector(0,-1){10.00}}
\put(20.00,20.00){\vector(1,-1){5.00}}
\put(8.00,20.00){\vector(-1,-2){7.67}}
\put(0.00,0.00){\makebox(0,0)[cb]{\tiny(\ )}}
\put(7.5,0.00){\makebox(0,0)[cb]{\tiny(\ )}}
\put(15.00,0.00){\makebox(0,0)[cb]{\tiny(\ )}}
\put(5.00,20.00){\framebox(20.00,10.00)[cc]{$\bullet\quad\bullet$}}
\bezier{40}(25.00,5.00)(30.00,5.00)(30.00,10.00)
\put(30.00,10.00){\line(0,1){20.00}}
\bezier{40}(30.00,30.00)(30.00,35.00)(25.00,35.00)
\put(25.00,35.00){\line(-1,0){5.00}}
\bezier{40}(20.00,35.00)(15.00,35.00)(15.00,30.00)
\put(15.00,31.00){\vector(0,-1){1.00}}
\put(15.00,10.00){\circle*{1}}
\put(25.00,15.00){\circle*{1}}
\put(25.00,5.00){\circle*{1}}
\end{picture}
\raisebox{17mm}{\quad$\Biggr\}$}
\\
\raisebox{17mm}{${}+ 3 \sum\limits_{\sigma\in S_3}(-)^\sigma\Biggl\{$}
%
\unitlength=1mm
\special{em:linewidth 0.4pt}
\linethickness{0.4pt}
\begin{picture}(30.00,35.00)
\put(15.00,10.00){\vector(-1,-4){1.75}}
\put(15.00,10.00){\vector(2,-1){10.00}}
\put(25.00,5.00){\vector(0,-1){3.00}}
\put(25.00,15.00){\vector(-2,-1){10.00}}
\put(25.00,15.00){\vector(0,-1){10.00}}
\put(15.00,20.00){\vector(0,-1){10.00}}
\put(20.00,20.00){\vector(1,-1){5.00}}
\put(8.00,20.00){\vector(-1,-2){8.5}}
\put(-0.10,-1){\makebox(0,0)[cb]{\tiny(\ )}}
\put(25.10,-1){\makebox(0,0)[cb]{\tiny(\ )}}
\put(13.00,-1){\makebox(0,0)[cb]{\tiny(\ )}}
\put(5.00,20.00){\framebox(20.00,10.00)[cc]{$\bullet\quad\bullet$}}
\bezier{40}(25.00,5.00)(30.00,5.00)(30.00,10.00)
\put(30.00,10.00){\line(0,1){20.00}}
\bezier{40}(30.00,30.00)(30.00,35.00)(25.00,35.00)
\put(25.00,35.00){\line(-1,0){5.00}}
\bezier{40}(20.00,35.00)(15.00,35.00)(15.00,30.00)
\put(15.00,31.00){\vector(0,-1){1.00}}
\put(15.00,10.00){\circle*{1}}
\put(25.00,15.00){\circle*{1}}
\put(25.00,5.00){\circle*{1}}
\end{picture}
\quad
\raisebox{17mm}{$+$}
%
\unitlength=1mm
\special{em:linewidth 0.4pt}
\linethickness{0.4pt}
\begin{picture}(30.00,35.00)
\put(15.00,10.00){\vector(1,-4){1.5}}
\put(25.00,5.00){\vector(-2,1){10.00}}
\put(15.00,10.00){\vector(2,1){10.00}}
\put(25.00,15.00){\vector(0,-1){10.00}}
\put(15.00,20.00){\vector(0,-1){10.00}}
\put(20.00,20.00){\vector(1,-1){5.00}}
\put(8.00,20.00){\vector(-1,-2){8}}
\put(0.00,0.00){\makebox(0,0)[cb]{\tiny(\ )}}
\put(33.00,0){\makebox(0,0)[cb]{\tiny(\ )}}
\put(16.50,0.00){\makebox(0,0)[cb]{\tiny(\ )}}
\put(5.00,20.00){\framebox(20.00,10.00)[cc]{$\bullet\quad\bullet$}}
\bezier{40}(25.00,5.00)(30.00,5.00)(30.00,10.00)
\put(30.00,10.00){\line(0,1){20.00}}
\bezier{40}(30.00,30.00)(30.00,35.00)(25.00,35.00)
\put(25.00,35.00){\line(-1,0){5.00}}
\bezier{40}(20.00,35.00)(15.00,35.00)(15.00,30.00)
\put(15.00,31.00){\vector(0,-1){1.00}}
\put(25.00,15.00){\line(1,0){4.00}}
\bezier{40}(31.00,15.00)(32.9,15.00)(32.9,11.00)
\put(33.00,11.00){\line(0,-1){5.00}}
\put(33.00,6.00){\vector(0,-1){2.00}}
\put(15.00,10.00){\circle*{1}}
\put(25.00,15.00){\circle*{1}}
\put(25.00,5.00){\circle*{1}}
\end{picture}
\raisebox{17mm}{\qquad$\Biggr\}.$}
\label{EqSol} 
\end{multline}
To display the $L \prec R$ ordering at every internal vertex and to make possible the arithmetic and algebra of graphs, we use the 
notation which is explained in Appendix~\ref{AppCode}. 


\begin{rem}
\label{RemTwoSets}
We remember that the set $\{\ground{1},\ground{2},\ground{3}\}$ of
three arguments of the Jacobiator need not coincide with the set $\{f,g,h\}$ of the arguments of the tri\/-\/vector $\Diamond(\cP, \Jac(\cP))$.
Of course, the two sets can intersect; this provides 
a natural filtration for the components of solution~\eqref{EqSol}.
Namely, the number of elements in the intersection runs from three for the first term to zero in the second or third graph.
\end{rem}


In fact, 
Remark~ \ref{RemTwoSets} reveals a highly nontrivial role of the operator~$\Diamond$ 
in~\eqref{EqFactor}. 
Some of the three internal vertices of its graphs can be 
arguments of~$\Jac(\cP)$ whereas some of the other such vertices (if any) can be tails for the arrows falling on~$\Jac(\cP)$.
In retrospect, the two subsets of such vertices of~$\Diamond$ do not intersect; every vertex in the intersection, if it were nonempty, would produce a two-cycle, but there are no ``eyes'' in~\eqref{EqSol}.

\begin{proof}[Proof of Theorem~\textup{\ref{ThMain}}]
So far, we have constructed operator~\eqref{EqSol}; 
it involves a reasonably small number of Leibniz graphs
so that the factorization in~\eqref{EqFactor} can be verified 
by a straightforward calculation.
The sums in~\eqref{EqSol} contain 
$27$~Leibniz graphs.
Now expand all the Leibniz rules; this yields the sum of $201$~Kontsevich graphs with $3$~sinks and $5$~internal vertices: together with their coefficients, they are listed in Table~\ref{Table201} in Appendix~\ref{AppCode}, see page~\pageref{AppCode}. 
We claim that by collecting similar terms, one obtains the $39$~graphs from the left\/-\/hand side of~\eqref{EqFactor}, see Fig.~\ref{Fig9Skew} 
and the encoding of those graphs in Table~\ref{Tab9Skew} on page~\pageref{Tab9Skew}.
Because we are free to enumerate the five internal vertices in every graph in a way we like, and because the ordering of every pair of outgoing edges is also under our control, at once do we bring all the graphs to their normal form.\footnote{The normal form of a graph is obtained by running over the group $S_{5} \times (\mathbb{Z}_{2})^{5}$ of all the relabellings of internal vertices and swaps $L\rightleftarrows R$ of orderings at each vertex. (We recall that every swap negates the coefficient of a graph; the permutations from $S_{5}$ are responsible for encoding a given topological profile in seemingly ``different" ways.) By definition, the \emph{normal form} of a graph is the minimal sequence of five ordered pairs of target vertices viewed as $10$-\/digit base-$(3+5)$ numbers. (By convention, the three ordered sinks are enumerated $0,1,2$ and the internal vertices are the octonary digits $3,\ldots,7$.)}

It is readily seen that there are many repetitions in Table~\ref{Table201}. 
We must inspect what vanishes and what stays. Let us do a sample reasoning first. 
Namely, let us inspect 
the contribution to the left-hand side of \eqref{EqFactor} from the first term of~\eqref{EqSol}. We have that
\begin{equation}\label{EqStayVanish}
\unitlength=.71mm
\special{em:linewidth 0.4pt}
\linethickness{0.4pt}
\raisebox{-29pt}[40pt][10pt]{
\begin{picture}(20.33,35.67)
\put(0.00,10.00){\framebox(20.00,10.00)[cc]{$\bullet\quad\bullet$}}
\put(3.00,10.00){\vector(-2,-3){3.33}}
\put(10.00,10.00){\vector(0,-1){5.00}}
\put(17.00,10.00){\vector(2,-3){3.33}}
\put(3.00,35.00){\vector(0,-1){15.00}}
\put(3.00,35.00){\vector(1,-1){7.00}}
\put(10.00,28.00){\vector(1,1){7.00}}
\put(17.00,35.00){\vector(-1,0){14.00}}
\put(10.00,28.00){\vector(0,-1){8.00}}
\put(17.00,35.00){\vector(0,-1){15.00}}
\put(17.00,35.00){\circle*{1}}
\put(3.00,35.00){\circle*{1}}
\put(10.00,28.00){\circle*{1}}
\end{picture}}
\ \ {=}
\raisebox{0mm}{$\ \  \sum\limits_{\circlearrowright}\Biggl\{$}
\unitlength=.6mm
\special{em:linewidth 0.4pt}
\linethickness{0.4pt}
\!\!
\raisebox{-25pt}[10pt][10pt]{
		\begin{picture}(40.00,30.00)(10,0)
		\put(20.00,30.00){\vector(-1,-2){7.67}}
		\put(12.33,15.00){\vector(4,3){14.67}}
		\put(20.00,30.1){\vector(-4,3){0}}
		\put(27.00,26){\line(-5,3){7.00}}
		\put(20.00,29.67){\vector(1,-4){4.67}}
		\put(24.7,10.65){\vector(1,-2){3.60}}
		\put(24.7,10.65){\vector(-1,-2){3.60}}
		\put(12.33,15.00){\vector(3,-1){12.33}}
		\put(27.00,26.00){\line(-1,-6){2.2}}
		\put(24.85,12.5){\vector(-1,-4){0.2}}
		\put(28.7,19.00){\vector(1,-2){7.8}}
		\put(28.7,19){\vector(-1,-2){4}}
		\put(21.50,0.50){\makebox(0,0)[cc]{\tiny$0$}}
		\put(29.00,0.50){\makebox(0,0)[cc]{\tiny$1$}}
		\put(37.00,0.50){\makebox(0,0)[cc]{\tiny$2$}}
		\put(27.50,11.00){\makebox(0,0)[cc]{\tiny$3$}}
		\put(30.00,21.00){\makebox(0,0)[cc]{\tiny$4$}}
		\put(10.00,14.50){\makebox(0,0)[cc]{\tiny$5$}}
		\put(29.50,28.50){\makebox(0,0)[cc]{\tiny$6$}}
		\put(20.00,33.00){\makebox(0,0)[cc]{\tiny$7$}}
		\end{picture}}\hspace{-9mm}
{+}
\raisebox{-25pt}[10pt][20pt]{
	\begin{picture}(20.00,30.00)(10,0)
	\put(20.00,30.00){\vector(-1,-2){7.67}}
	\put(12.33,15.00){\vector(4,3){14.67}}
	\put(20.00,30.1){\vector(-4,3){0}}
	\put(27.00,26){\line(-5,3){7.00}}
	\put(20.00,29.67){\vector(1,-4){4.67}}
	\put(24.7,10.65){\vector(1,-2){3.60}}
	\put(24.7,10.65){\vector(-1,0){7.20}}
	\put(12.33,15.00){\vector(3,-1){12.33}}
	\put(27.00,26.00){\line(-1,-6){2.2}}
	\put(24.85,12.5){\vector(-1,-4){0.2}}
	\put(17.5,10.65){\vector(1,-2){3.60}}
	\put(17.5,10.65){\vector(-1,-2){3.60}}
	\put(13.00,0.50){\makebox(0,0)[cc]{\tiny$0$}}
	\put(20.50,0.50){\makebox(0,0)[cc]{\tiny$1$}}
	\put(28.50,0.50){\makebox(0,0)[cc]{\tiny$2$}}
	\put(14.50,11.00){\makebox(0,0)[cc]{\tiny$3$}}
	\put(27.50,11.00){\makebox(0,0)[cc]{\tiny$4$}}
	\put(10.00,14.50){\makebox(0,0)[cc]{\tiny$5$}}
	\put(29.50,27.50){\makebox(0,0)[cc]{\tiny$6$}}
	\put(20.00,33.00){\makebox(0,0)[cc]{\tiny$7$}}
	\end{picture}}
\hspace{-10mm}
\qquad+{}{3}
\raisebox{-25pt}[10pt][10pt]{
\begin{picture}(15,24)
\unitlength 1.2mm
\put(2,6){\circle*{0.33}}
\bezier{512}(2,6)(0,4)(-2,2)
\put(-2,2){\vector(-1,-1){0}}
\bezier{512}(2,6)(4,4)(6,2)
\put(6,2){\vector(1,-1){0}}
\put(8,8){\circle*{0.33}}
\bezier{512}(8,8)(10,5)(12,2)
\put(12,2){\vector(2,-3){0}}
\bezier{512}(8,8)(5,7)(2,6)
\put(2,6){\vector(-3,-1){0}}
\put(1.7431565890855176,11.022589867050925){\circle*{0.33}}
\bezier{512}(1.7431565890855176,11.022589867050925)(1.8715782945427588,8.511294933525463)(2,6)
\put(2,6){\vector(0,-1){0}}
\put(7.9300270062781175,12.995699891993429){\circle*{0.33}}
\bezier{512}(1.7431565890855176,11.022589867050925)(4.836591797681818,12.009144879522177)(7.9300270062781175,12.995699891993429)
\put(7.9300270062781175,12.995699891993429){\vector(3,1){0}}
\bezier{512}(7.9300270062781175,12.995699891993429)(7.965013503139058,10.497849945996714)(8,8)
\put(8,8){\vector(0,-1){0}}
\put(4.284789841553829,15.136022630914107){\circle*{0.33}}
\bezier{512}(7.9300270062781175,12.995699891993429)(6.107408423915974,14.065861261453769)(4.284789841553829,15.136022630914107)
\put(4.284789841553829,15.136022630914107){\vector(-2,1){0}}
\bezier{512}(4.284789841553829,15.136022630914107)(3.013973215319673,13.079306248982515)(1.7431565890855176,11.022589867050925)
\put(1.7431565890855176,11.022589867050925){\vector(-1,-2){0}}
\bezier{512}(4.284789841553829,15.136022630914107)(-2.54509791207584,13.159636322102537)(2,6)
\put(2,6){\vector(1,-1){0}}
\end{picture}}
\quad {}+{3}\!\!
\raisebox{-25pt}[10pt][10pt]{
\begin{picture}(15,24)
\unitlength 1.2mm
\put(2,6){\circle*{0.33}}
\bezier{512}(2,6)(0,4)(-2,2)
\put(-2,2){\vector(-1,-1){0}}
\bezier{512}(2,6)(4,4)(6,2)
\put(6,2){\vector(1,-1){0}}
\put(8,8){\circle*{0.33}}
\bezier{512}(8,8)(10,5)(12,2)
\put(12,2){\vector(2,-3){0}}
\bezier{512}(8,8)(5,7)(2,6)
\put(2,6){\vector(-3,-1){0}}
\put(1.7431565890855176,11.022589867050925){\circle*{0.33}}
\bezier{512}(1.7431565890855176,11.022589867050925)(1.8715782945427588,8.511294933525463)(2,6)
\put(2,6){\vector(0,-1){0}}
\put(7.9300270062781175,12.995699891993429){\circle*{0.33}}
\bezier{512}(1.7431565890855176,11.022589867050925)(4.836591797681818,12.009144879522177)(7.9300270062781175,12.995699891993429)
\put(7.9300270062781175,12.995699891993429){\vector(3,1){0}}
\bezier{512}(7.9300270062781175,12.995699891993429)(7.965013503139058,10.497849945996714)(8,8)
\put(8,8){\vector(0,-1){0}}
\put(4.284789841553829,15.136022630914107){\circle*{0.33}}
\bezier{512}(7.9300270062781175,12.995699891993429)(6.107408423915974,14.065861261453769)(4.284789841553829,15.136022630914107)
\put(4.284789841553829,15.136022630914107){\vector(-2,1){0}}
\bezier{512}(4.284789841553829,15.136022630914107)(3.013973215319673,13.079306248982515)(1.7431565890855176,11.022589867050925)
\put(1.7431565890855176,11.022589867050925){\vector(-1,-2){0}}
\bezier{512}(4.284789841553829,15.136022630914107)(12.54509791207584,15.159636322102537)(8,8)
\put(8,8){\vector(-1,-1){0}}
\end{picture}} 
\!\qquad \Biggr\}.
\end{equation}
The right\/-\/hand side of~\eqref{EqStayVanish} expands into the sum of $12$~different graphs. They are marked in the first twenty\/-\/four lines 
of Table~\ref{Table201} by $\diamondsuit_{i}, \heartsuit_{i}, \clubsuit_{i}$ and $\spadesuit_{i}$ for $1 \leqslant i \leqslant 3$, respectively; by definition, a suit 
with different values of its subscript~$i$ denotes the $i$th cyclic permutation of the ground vertices for the same graph.\footnote{By taking a graph, placing it consecutively over three cyclic permutations of its sinks' content, \emph{and} bringing the three graph encodings to their normal form, see above, one can obtain an extra sign factor in front of some of these graphs. This is due to a convention about ``minimal'' graph encoding, not signalling any mismatch in the arithmetic. For example, after the normalization such is the case with the columns in Table~\ref{TableContributions}: each column refers to a cyclic permutation of three arguments and the coefficients in every line would coincide if one encoded the graphs for the last column not using the respective minimal $10$-\/digit octonary numbers. To make all the three coefficients in each line coinciding, it is enough to swap $L\rightleftarrows R$ in one internal vertex in every graph from the third column.}
For example, the symbols $\diamondsuit_{1}, \diamondsuit_{2}, \diamondsuit_{3}$ mark the three cyclic permutations of arguments in the first term in the right\/-\/hand side of~\eqref{EqStayVanish}.
The sum of the first two terms in the right-hand side of \eqref{EqStayVanish} -- marked by $\diamondsuit_{i}$ and $\heartsuit_{i}$, respectively -- equals the sum of the first two terms in Fig.~\ref{Fig9Skew}.\footnote{We inspect further that no other graphs in Table~\ref{Table201} make any contribution to the coefficients of these two graphs.}
At the same time, the sum of the last two terms -- whose encodings with coefficients $\pm 1$ are marked by $\clubsuit_{i}$ and $\spadesuit_{i}$, respectively -- cancels against the contributions from the fourth and sixth terms in solution~\eqref{EqSol} -- with coefficients~$\pm 3$, also marked by~$\clubsuit_{i}$ and~$\spadesuit_{i}$ in the rest of Table~\ref{Table201}. In Table~\ref{TableContributions} we calculate the coefficient of each graph marked by the respective indexed symbol.
\begin{table}[htb]
\centering
\caption{The coefficients of graphs marked by the four suits.}
\label{TableContributions}
{\renewcommand{\arraystretch}{1.3}
\begin{tabular}{rr|rr|rr}
$\diamondsuit_{1}:$&$-1$&$\diamondsuit_{2}:$&$-1$&$\diamondsuit_{3}:$&$+1$\\\hline
$\heartsuit_{1}:$&$-1$&$\heartsuit_{2}:$&$-1$&$\heartsuit_{3}:$&$+1$\\\hline
$\clubsuit_{1}:$&$-1-1-1+3=0$&$\clubsuit_{2}:$&$-1-1-1+3=0$&$\clubsuit_{3}:$&$+1+1+1-3=0$\\\hline
$\spadesuit_{1}:$&$-1-1-1+3=0$&$\spadesuit_{2}:$&$-1-1-1+3=0$&$\spadesuit_{3}:$&$+1+1+1-3=0$
\end{tabular}}
\end{table}

Now, in the same way all other similar terms are collected. There remain only $39$ terms with nonzero coefficients.
One verifies that those $39$ terms are none other than the entries of Table~\ref{TableLHS}, that is, realizations of the $39$ graphs in the left-hand side of~\eqref{EqFactor}. This shows that equation~\eqref{EqFactor} holds for the operator $\Diamond$ contained in~\eqref{EqSol}. 
\end{proof}

\begin{rem}
Operator~\eqref{EqSol} is not a unique solution of factorization problem~\eqref{EqFactor}.
We claim that apart from this sum of $27$~Leibniz graphs, there is another solution which consists of $102$~Leibniz graphs; it is also linear with respect to the Jacobiator (that is, its realization in the form~$\Diamond(\cP,\Jac(\cP),\Jac(\cP))$ is not possible).
\end{rem}

\section*{Discussion}
\subsection*{Non\/-\/triviality}
A flow specified on the space of Poisson bi\/-\/vectors by using the Kon\-tse\-vich graphs can be Poisson cohomology trivial modulo a sum of Leibniz graphs that would vanish identically at any Poisson structure. However, this is not the case of the Kon\-tse\-vich tetrahedral flow $\dot{\cP}=\cQ_{1:6}(\cP)$.

\begin{proposition}
There is no $1$-vector field~$\cX$ encoded over~$N^{n}$ by the Kontsevich graphs and there is no operator~$\nabla$ encoded using the Leibniz graphs such that 
\[
\cQ_{1:6}(\cP)=\schouten{\cP,\cX}+\nabla(\cP,\Jac(\cP)).
\]
\end{proposition}
\noindent%
The claim is established 
by a 
run\/-\/through over all Kontsevich graphs with three internal vertices and one sink (making an ansatz for~$\cX$) and all Leibniz graphs (in the operator~$\nabla$) with two copies of~$\cP$ and one Jacobiator in the internal vertices; all such graphs of both types are taken with undetermined coefficients. The resulting inhomogeneous linear algebraic system has no solution.

\subsection*{Integrability}
By using the technique of Kontsevich graphs one can proceed with a higher order expansion of the tetrahedral deformation, 
\[
\cP \mapsto \cP + \varepsilon \cQ_{1:6}(\cP)  + \varepsilon \cR(\cP) + \dots +\bar{o}(\varepsilon^{d}) ,\quad d \geqslant 2,
\]
for Poisson structures~$\cP$. Assuming that the master\/-\/equation holds up to~$\bar{o}(\varepsilon^{d})$, 
\begin{multline}
\schouten{\cP + \varepsilon \cQ_{1:6}(\cP)  + \varepsilon \cR(\cP) + \dots +\bar{o}(\varepsilon^{d}),\cP + \varepsilon \cQ_{1:6}(\cP)  + \varepsilon \cR(\cP) + \dots +\bar{o}(\varepsilon^{d})}\doteq\bar{o}(\varepsilon^{d})\\
\text{ via }\  \schouten{\cP,\cP}=0,\label{EqMasterHighOrder}
\end{multline}
we obtain a chain of linear equations for the higher order expansion terms, namely, 
\begin{equation}\label{EqHighOrd}
2\schouten{\cP,\cR(\cP)}+\schouten{\cQ_{1:6}(\cP),\cQ_{1:6}(\cP)}\doteq 0\qquad \text{via }\ \schouten{\cP,\cP}=0,\ \text{ etc.}
\end{equation}
A solution consisting of~$\cR(\cP)$ and consecutive terms at higher powers of the deformation parameter\footnote{In every graph at~$\veps^k$ the number of internal vertices is~$3k+1$.}
can be sought using the same factorization techniques and computer\/-\/assisted proof schemes~\cite{cpcOrder4} which have been implemented in this paper~--- whenever such solution actually exists. It is clear that there can be Poisson cohomological obstructions to resolvability of cocycle conditions~\eqref{EqHighOrd}. Hence the integrability issue for the Kontsevich tetrahedral flow may be Poisson model\/-\/dependent, unlike the universal nature of such deformation's infinitesimal part.

\appendix
\section{Encoding of the solution}\label{AppCode}
\noindent%
Let $\Gamma$ be a labelled Kontsevich graph with $n$~internal and $m$~external vertices.
We assume the ground vertices of $\Gamma$ are labelled~$[0$,\ $\ldots$,\ $m-1]$ and the internal vertices are labelled~$[m$,\ $\ldots$,\ $m + n - 1]$.
We define the \emph{encoding} of~$\Gamma$ to be the \emph{prefix}~$(n,m)$, followed by 
a list of \emph{targets}. The list of targets consists of ordered pairs where the $k$th~pair ($k\geqslant0$) contains the two targets of the internal vertex number~$m+k$.

The expansion of the Schouten bracket $[\![\mathcal{P}, \mathcal{Q}_{a:b}]\!]$ for the ratio $a:b = 1:6$ depicted in Figure \ref{FigLHS} simplifies to a sum of $39$ graphs with coefficients $\pm \tfrac{1}{4}$,\ $\pm \tfrac{3}{4}$.
The encodings of these graphs, followed by their respective coefficients, are listed in Table~\ref{TableLHS}.
\begin{table}[htb]
\caption{Machine\/-\/readable encoding of Fig.~\ref{FigLHS} on p.~\pageref{FigLHS}.}\label{TableLHS}
\vskip 1em
\begin{tabular}{r l l l r p{1cm} r l l l r }
1.1 & 3 & 5 & 4 2 0 1 4 6 4 7 4 5 & $1/4$ &&
7.1 & 3 & 5 & 6 2 7 0 1 4 4 5 5 6 & $3/4$ \\
1.2 & 3 & 5 & 4 0 1 2 4 6 4 7 4 5 & $1/4$ &&
7.2 & 3 & 5 & 6 0 7 1 2 4 4 5 5 6 & $3/4$ \\
1.3 & 3 & 5 & 4 1 2 0 4 6 4 7 4 5 & $1/4$ &&
7.3 & 3 & 5 & 6 1 7 2 0 4 4 5 5 6 & $3/4$ \\
& & & & &&
& & & & \\
2.1 & 3 & 5 & 7 0 3 5 3 6 3 4 1 2 & $1/4$ &&
8.1 & 3 & 5 & 7 2 7 0 1 4 4 5 5 6 & $3/4$ \\
2.2 & 3 & 5 & 7 1 3 5 3 6 3 4 2 0 & $1/4$ &&
8.2 & 3 & 5 & 7 0 7 1 2 4 4 5 5 6 & $3/4$ \\
2.3 & 3 & 5 & 7 2 3 5 3 6 3 4 0 1 & $1/4$ &&
8.3 & 3 & 5 & 7 1 7 2 0 4 4 5 5 6 & $3/4$ \\
& & & & &&
& & & & \\
3.1 & 3 & 5 & 5 2 0 1 4 6 4 7 4 5 & $3/4$ &&
9.1 & 3 & 5 & 4 2 7 1 0 4 4 5 5 6 & $-3/4$ \\
3.2 & 3 & 5 & 5 0 1 2 4 6 4 7 4 5 & $3/4$ &&
9.2 & 3 & 5 & 4 0 7 2 1 4 4 5 5 6 & $-3/4$ \\
3.3 & 3 & 5 & 5 1 2 0 4 6 4 7 4 5 & $3/4$ &&
9.3 & 3 & 5 & 4 1 7 0 2 4 4 5 5 6 & $-3/4$ \\
& & & & &&
& & & & \\
4.1 & 3 & 5 & 6 7 0 3 3 4 4 5 1 2 & $3/4$ &&
10.1 & 3 & 5 & 5 2 7 1 0 4 4 5 5 6 & $-3/4$ \\
4.2 & 3 & 5 & 6 7 1 3 3 4 4 5 2 0 & $3/4$ &&
10.2 & 3 & 5 & 5 0 7 2 1 4 4 5 5 6 & $-3/4$ \\
4.3 & 3 & 5 & 6 7 2 3 3 4 4 5 0 1 & $3/4$ &&
10.3 & 3 & 5 & 5 1 7 0 2 4 4 5 5 6 & $-3/4$ \\
& & & & &&
& & & & \\
5.1 & 3 & 5 & 4 2 7 0 1 4 4 5 5 6 & $3/4$ &&
11.1 & 3 & 5 & 6 2 7 1 0 4 4 5 5 6 & $-3/4$ \\
5.2 & 3 & 5 & 4 0 7 1 2 4 4 5 5 6 & $3/4$ &&
11.2 & 3 & 5 & 6 0 7 2 1 4 4 5 5 6 & $-3/4$ \\
5.3 & 3 & 5 & 4 1 7 2 0 4 4 5 5 6 & $3/4$ &&
11.3 & 3 & 5 & 6 1 7 0 2 4 4 5 5 6 & $-3/4$ \\
& & & & &&
& & & & \\
6.1 & 3 & 5 & 5 2 7 0 1 4 4 5 5 6 &  $3/4$ &&
12.1 & 3 & 5 & 7 2 7 1 0 4 4 5 5 6 &  $-3/4$ \\
6.2 & 3 & 5 & 5 0 7 1 2 4 4 5 5 6 &   $3/4$ &&
12.2 & 3 & 5 & 7 0 7 2 1 4 4 5 5 6 &   $-3/4$ \\
6.3 & 3 & 5 & 5 1 7 2 0 4 4 5 5 6 &   $3/4$ &&
12.3 & 3 & 5 & 7 1 7 0 2 4 4 5 5 6 &   $-3/4$ \\
& & & & &&
& & & & \\
& & & & &&
13.1 & 3 & 5 & 6 0 7 3 3 4 4 5 1 2  & $-3/4$ \\
& & & & &&
13.2 & 3 & 5 & 6 1 7 3 3 4 4 5 2 0 &   $-3/4$ \\
& & & & &&
13.3 & 3 & 5 & 6 2 7 3 3 4 4 5 0 1 &   $-3/4$ \\
\end{tabular}
\end{table}
The graphs are collected into groups of three, consisting of the skew\/-\/symmetrization -- by a sum over cyclic permutations -- of a single graph.
Within the encodings in the groups of three, the lists of targets only differ by a cyclic permutation of the target vertices $0,1,2$.

\begin{table}[htb]
\caption{Machine\/-\/readable encoding of Fig.~\ref{Fig9Skew} on p.~\pageref{Fig9Skew}.}\label{Tab9Skew}
\vskip 1em
\begin{tabular}{l l r r}
3&5&0 1 2 3 4 6 4 7 4 5&$-1/2$\\
3&5&0 4 1 2 4 6 4 7 4 5&$-1/2$\\
3&5&0 4 5 6 1 2 5 7 4 5&$3/2$\\
3&5&0 1 2 5 6 7 3 4 4 6&$3/2$\\
3&5&0 4 5 6 1 2 3 7 3 4&$3/2$\\
3&5&0 4 5 6 1 6 2 7 4 5&$-3$\\
3&5&0 4 5 6 1 7 5 7 2 4&$3$\\
3&5&0 4 1 5 2 6 4 7 4 5&$3$\\
3&5&0 4 2 5 6 7 1 4 4 6&$-3$
\end{tabular}
\end{table}

Consisting of $8$~skew\/-\/symmetric terms,
the solution (see~
\eqref{EqSol} on p.~\pageref{pSolution}) 
is encoded in Table~\ref{TableSol}: the sought\/-\/for values of coefficients are written after the encoding of the respective~27 Leibniz graphs.
\label{pTableSol}
\begin{table}[htb]
\caption{Machine\/-\/readable encoding of solution~\eqref{EqSol} on p.~\pageref{pSolution}.}\label{TableSol}
\vskip 1em
\begin{tabular}{r l l l r p{1cm} r l l l r }
1.1 & 3 & 5 &  4 6 5 6 3 6 0 1 6 2  & $-1$ &&
6.1 & 3 & 5 &  1 2 3 5 3 6 0 3 6 4  & $3$ \\
& & & & &&
6.2 & 3 & 5 &  0 2 3 5 3 6 1 3 6 4  & $-3$ \\
2.1 & 3 & 5 &  0 4 1 5 2 3 3 4 6 5  & $-3$ &&
6.3 & 3 & 5 &  4 6 0 1 3 4 2 4 6 5  & $-3$ \\
2.2 & 3 & 5 &  0 4 2 5 1 3 3 4 6 5  & $3$ &&
& & & & \\
& & & & &&
7.1 & 3 & 5 &  1 5 3 5 2 6 0 3 6 4  & $-3$ \\
3.1 & 3 & 5 &  0 4 1 2 3 4 3 4 6 5  & $-3$ &&
7.2 & 3 & 5 &  1 5 3 5 0 6 2 3 6 4  & $3$ \\
3.2 & 3 & 5 &  0 1 2 3 3 4 3 4 6 5  & $-3$ &&
7.3 & 3 & 5 &  0 5 3 5 2 6 1 3 6 4  & $3$ \\
3.3 & 3 & 5 &  0 2 1 3 3 4 3 4 6 5  & $3$ &&
7.4 & 3 & 5 &  2 5 3 5 1 6 0 3 6 4  & $3$ \\
& & & & &&
7.5 & 3 & 5 &  2 5 3 5 0 6 1 3 6 4  & $-3$ \\
4.1 & 3 & 5 &  4 5 1 6 4 6 0 2 6 3  & $-3$ &&
7.6 & 3 & 5 &  0 5 3 5 1 6 2 3 6 4  & $-3$ \\
4.2 & 3 & 5 &  4 5 0 6 4 6 1 2 6 3  & $3$ &&
& & & & \\
4.3 & 3 & 5 &  5 6 3 5 2 6 0 1 6 4  & $-3$ &&
8.1 & 3 & 5 &  1 4 2 5 3 6 0 3 6 4  & $-3$ \\
& & & & &&
8.2 & 3 & 5 &  1 5 2 3 4 6 0 3 6 4  & $-3$ \\
5.1 & 3 & 5 &  1 4 5 6 3 6 0 2 6 3  & $3$ &&
8.3 & 3 & 5 &  0 4 2 5 3 6 1 3 6 4  & $3$ \\
5.2 & 3 & 5 &  0 4 5 6 3 6 1 2 6 3  & $-3$ &&
8.4 & 3 & 5 &  0 5 2 3 4 6 1 3 6 4  & $3$ \\
5.3 & 3 & 5 &  5 6 2 3 4 6 0 1 6 4  & $-3$ &&
8.5 & 3 & 5 &  4 6 0 5 1 3 2 4 6 5  & $-3$ \\
& & & & &&
8.6 & 3 & 5 &  4 6 1 5 0 3 2 4 6 5  & $3$ 
\end{tabular}
\end{table}
Here the sums over permutations of the ground vertices are expanded (thus making the 
27~Leibniz graphs out of the 8~skew\/-\/symmetric groups). 
In every entry of Table~\ref{TableSol}, the sum of three graphs in Jacobiator~\eqref{EqJacFig} is represented by its first term.
For all the in\/-\/coming arrows, 
the vertex~6 is the placeholder for the Jacobiator (again, see~\eqref{EqJacFig} on p.~\pageref{EqJacFig}); in earnest, the Jacobiator contains the internal vertices~6 and~7.
This convention is helpful: 
for every set of derivations acting on the Jacobiator 
with 
internal vertices $6$ and $7$, 
only the first term is listed, 
namely the one where each edge lands on~$6$.

\begin{example}
The first entry of Table~\ref{TableSol} encodes a graph containing a three\/-\/cycle over internal vertices $3,4,5$. Issued from each of these three, the other edge lands on the vertex~$6$: the placeholder for the Jacobiator. This entry is the first term in~\eqref{EqSol} on p.~\pageref{pSolution}.
\end{example}

\begin{example}\label{Gamma0InSol}
The entry~3.1 is one of three terms produced by the third graph in 
solution~ \eqref{EqSol};
the Jacobiator in this entry is expanded using formula~\eqref{EqJacFig}, resulting in three terms (by definition).
It is easy to see that the first term 
contains picture~\eqref{EqWedgeOnTwoWedges} from Remark~\ref{RemGamma0} as a subgraph.
Hence the polydifferential operator encoded by this graph vanishes due to skew\/-\/symmetry.
However, the other two terms 
produced in the entry~3.1 by formula~\eqref{EqJacFig}
do not vanish by skew\/-\/symmetry. 
Likewise, there is one term vanishing by the same mechanism in the entry~3.2 and in~3.3.
\end{example}

The proof of Theorem \ref{ThMain} amounts to expanding the Leibniz rules on Jacobiators in Table~\ref{TableSol} according to the rules above (resulting in Table~\ref{Table201} on p.~\pageref{Table201}, where the prefix ``$3 \quad 5$'' of each graph has been omitted for brevity), simplifying by collecting terms, and seeing that one obtains Table~\ref{Tab9Skew}. 

\begin{table}[htbp]
\caption{Expansion of Leibniz rules on Jacobiators in Table~\ref{TableSol}.}%
\label{Table201}
\tiny\renewcommand{\baselinestretch}{0.9}\selectfont
\begin{tabular}{l l l r}
$\diamondsuit_{1}$& 0 1 2 3 3 6 3 7 3 5  & & $-1$ \\
$\clubsuit_{1}$ & 0 1 2 3 3 6 3 7 4 5  & & $-1$ \\
$\clubsuit_{1}$ & 0 1 2 3 3 6 3 7 4 5  & & $-1$ \\
$\spadesuit_{1}$ & 0 1 2 3 3 6 4 7 4 5  & & $-1$ \\
$\clubsuit_{1}$ & 0 1 2 3 3 6 3 7 4 5  & & $-1$ \\
$\spadesuit_{1}$ & 0 1 2 3 3 6 4 7 4 5  & & $-1$ \\
$\spadesuit_{1}$ & 0 1 2 3 3 6 4 7 4 5  & & $-1$ \\
$\heartsuit_{1}$& 0 1 2 3 4 6 4 7 4 5  & & $-1$ \\
$\diamondsuit_{2}$& 0 4 1 2 4 6 4 7 4 5  & & $-1$ \\
$\clubsuit_{2}$ & 0 4 1 2 3 6 4 7 4 5  & & $-1$ \\
$\clubsuit_{2}$ & 0 4 1 2 3 6 4 7 4 5  & & $-1$ \\
$\spadesuit_{2}$ & 0 4 1 2 3 6 3 7 4 5  & & $-1$ \\
$\clubsuit_{2}$ & 0 4 1 2 3 6 4 7 4 5  & & $-1$ \\
$\spadesuit_{2}$ & 0 4 1 2 3 6 3 7 4 5  & & $-1$ \\
$\spadesuit_{2}$ & 0 4 1 2 3 6 3 7 4 5  & & $-1$ \\
$\heartsuit_{2}$& 0 4 1 2 3 6 3 7 3 5  & & $-1$ \\
$\diamondsuit_{3}$& 0 2 1 3 3 6 3 7 3 5  & & $1$ \\
$\clubsuit_{3}$ & 0 2 1 3 3 6 3 7 4 5  & & $1$ \\
$\clubsuit_{3}$ & 0 2 1 3 3 6 3 7 4 5  & & $1$ \\
$\spadesuit_{3}$ & 0 2 1 3 3 6 4 7 4 5  & & $1$ \\
$\clubsuit_{3}$ & 0 2 1 3 3 6 3 7 4 5  & & $1$ \\
$\spadesuit_{3}$ & 0 2 1 3 3 6 4 7 4 5  & & $1$ \\
$\spadesuit_{3}$ & 0 2 1 3 3 6 4 7 4 5  & & $1$ \\
$\heartsuit_{3}$& 0 2 1 3 4 6 4 7 4 5  & & $1$ \\
& 0 1 2 5 3 6 3 4 3 4  & & $-3$ \\
& 0 1 2 5 3 6 4 7 3 4  & & $3$ \\
& 0 1 2 5 6 7 3 4 3 4  & & $-3$ \\
& 0 1 2 5 6 7 3 4 4 6  & & $3$ \\
& 0 4 1 5 2 6 4 7 4 5  & & $3$ \\
& 0 4 1 5 2 6 4 7 3 5  & & $3$ \\
& 0 4 1 5 2 6 3 7 4 5  & & $3$ \\
& 0 4 1 5 2 6 3 7 3 5  & & $3$ \\
& 0 4 2 5 3 6 3 4 1 3  & & $-3$ \\
& 0 4 2 5 3 6 4 7 1 3  & & $3$ \\
& 0 4 2 5 6 7 1 3 3 4  & & $-3$ \\
& 0 4 2 5 6 7 1 3 4 6  & & $3$ \\
& 0 1 2 3 3 4 3 7 4 5  & & $-3$ \\
& 0 1 2 5 3 6 4 7 3 4  & & $-3$ \\
$\spadesuit_{1}$ & 0 1 2 3 3 6 4 7 4 5  & & $3$ \\
& 0 1 2 5 3 6 4 7 4 5  & & $3$ \\
& 0 4 1 5 6 7 2 4 4 6  & & $3$ \\
& 0 4 1 5 6 7 2 3 4 6  & & $3$ \\
& 0 4 1 5 6 7 2 4 3 6  & & $3$ \\
& 0 4 1 5 6 7 2 3 3 6  & & $3$ \\
& 0 4 5 6 2 3 3 5 1 3  & & $-3$ \\
& 0 4 5 6 2 7 3 5 1 3  & & $-3$ \\
& 0 4 5 6 2 3 5 7 1 3  & & $3$ \\
& 0 4 5 6 2 7 5 7 1 3  & & $3$ \\
& 0 2 1 5 3 6 3 4 3 4  & & $3$ \\
& 0 2 1 5 3 6 4 7 3 4  & & $-3$ \\
& 0 2 1 5 6 7 3 4 3 4  & & $3$ \\
& 0 2 1 5 6 7 3 4 4 6  & & $-3$ \\
& 0 4 2 5 1 6 4 7 4 5  & & $-3$ \\
& 0 4 2 5 1 6 4 7 3 5  & & $-3$ \\
& 0 4 2 5 1 6 3 7 4 5  & & $-3$ \\
& 0 4 2 5 1 6 3 7 3 5  & & $-3$ \\
& 0 4 1 5 3 6 3 4 2 3  & & $3$ \\
& 0 4 1 5 3 6 4 7 2 3  & & $-3$ \\
& 0 4 1 5 6 7 2 3 3 4  & & $3$ \\
& 0 4 1 5 6 7 2 3 4 6  & & $-3$ \\
& 0 2 1 3 3 4 3 7 4 5  & & $3$ \\
$\spadesuit_{3}$ & 0 2 1 3 3 6 4 7 4 5  & & $-3$ \\
& 0 2 1 5 3 6 4 7 3 4  & & $3$ \\
& 0 2 1 5 3 6 4 7 4 5  & & $-3$ \\
& 0 4 2 5 6 7 1 4 4 6  & & $-3$ \\
& 0 4 2 5 6 7 1 4 3 6  & & $-3$ \\
& 0 4 2 5 6 7 1 3 4 6  & & $-3$
\end{tabular}
\hskip 2.5em
\begin{tabular}{l l l r}
& 0 4 2 5 6 7 1 3 3 6  & & $-3$ \\
& 0 4 5 6 1 3 3 5 2 3  & & $3$ \\
& 0 4 5 6 1 3 5 7 2 3  & & $-3$ \\
& 0 4 5 6 1 7 3 5 2 3  & & $3$ \\
& 0 4 5 6 1 7 5 7 2 3  & & $-3$ \\
& 0 4 1 2 3 4 3 7 4 5  & & $-3$ \\
$\clubsuit_{2}$ & 0 4 1 2 3 6 4 7 4 5  & & $3$ \\
& 0 4 5 6 1 2 5 7 4 5  & & $3$ \\
& 0 4 5 6 1 2 5 7 3 5  & & $3$ \\
& 0 4 5 6 1 2 3 5 3 5  & & $3$ \\
& 0 4 5 6 1 2 5 7 3 5  & & $-3$ \\
& 0 4 1 5 2 6 3 4 3 5  & & $3$ \\
& 0 4 1 5 2 6 4 7 3 5  & & $-3$ \\
& 0 4 5 6 1 6 2 7 4 5  & & $-3$ \\
& 0 4 5 6 1 6 2 7 3 5  & & $-3$ \\
& 0 4 2 5 3 6 1 4 3 6  & & $-3$ \\
& 0 4 2 5 6 7 1 4 3 6  & & $3$ \\
& 0 4 1 5 3 6 2 4 3 6  & & $3$ \\
& 0 4 1 5 6 7 2 4 3 6  & & $-3$ \\
& 0 4 5 6 1 7 2 5 4 6  & & $-3$ \\
& 0 4 5 6 1 7 2 5 3 6  & & $-3$ \\
& 0 4 2 5 1 6 3 4 3 5  & & $-3$ \\
& 0 4 2 5 1 6 4 7 3 5  & & $3$ \\
& 0 4 1 5 2 3 3 7 4 5  & & $-3$ \\
& 0 4 1 5 2 6 3 7 4 5  & & $-3$ \\
& 0 4 5 6 1 7 5 7 2 4  & & $3$ \\
& 0 4 5 6 1 7 5 7 2 3  & & $3$ \\
& 0 4 5 6 1 6 2 3 3 5  & & $3$ \\
& 0 4 5 6 1 6 2 7 3 5  & & $3$ \\
& 0 4 2 5 1 3 3 7 4 5  & & $3$ \\
& 0 4 2 5 1 6 3 7 4 5  & & $3$ \\
& 0 4 5 6 2 7 5 7 1 4  & & $-3$ \\
& 0 4 5 6 2 7 5 7 1 3  & & $-3$ \\
& 0 4 5 6 1 3 2 5 3 6  & & $3$ \\
& 0 4 5 6 1 7 2 5 3 6  & & $3$ \\
& 0 4 5 6 1 2 3 5 3 5  & & $-3$ \\
& 0 4 5 6 1 2 3 7 3 5  & & $-3$ \\
& 0 4 5 6 3 7 3 7 1 2  & & $-3$ \\
& 0 4 5 6 3 6 3 7 1 2  & & $3$ \\
& 0 4 5 6 2 3 3 5 1 5  & & $-3$ \\
& 0 4 5 6 2 3 3 7 1 5  & & $-3$ \\
& 0 4 5 6 1 7 3 7 2 3  & & $-3$ \\
& 0 4 5 6 1 7 3 5 2 3  & & $-3$ \\
& 0 4 5 6 1 3 3 5 2 5  & & $3$ \\
& 0 4 5 6 1 3 3 7 2 5  & & $3$ \\
& 0 4 5 6 2 7 3 7 1 3  & & $3$ \\
& 0 4 5 6 2 7 3 5 1 3  & & $3$ \\
& 0 4 1 2 3 4 3 5 4 6  & & $3$ \\
$\spadesuit_{2}$ & 0 4 1 2 3 6 3 7 4 5  & & $3$ \\
& 0 4 5 6 1 2 3 7 3 5  & & $3$ \\
& 0 4 5 6 1 2 3 7 3 4  & & $3$ \\
& 0 4 2 5 3 6 3 4 1 4  & & $-3$ \\
& 0 4 2 5 3 6 3 7 1 4  & & $-3$ \\
& 0 4 1 5 2 6 3 7 3 5  & & $-3$ \\
& 0 4 1 5 2 6 3 7 3 4  & & $-3$ \\
& 0 4 1 5 3 6 3 4 2 4  & & $3$ \\
& 0 4 1 5 3 6 3 7 2 4  & & $3$ \\
& 0 4 2 5 1 6 3 7 3 5  & & $3$ \\
& 0 4 2 5 1 6 3 7 3 4  & & $3$ \\
& 0 2 1 3 3 4 3 5 4 6  & & $-3$ \\
$\clubsuit_{3}$ & 0 2 1 3 3 6 3 7 4 5  & & $-3$ \\
& 0 2 1 5 3 6 3 7 3 5  & & $3$ \\
& 0 2 1 5 3 6 3 7 3 4  & & $3$ \\
& 0 2 1 5 3 6 3 4 3 4  & & $-3$ \\
& 0 2 1 5 3 6 3 7 3 4  & & $-3$ \\
& 0 4 2 5 3 6 1 3 4 6  & & $-3$ \\
& 0 4 2 5 3 6 4 7 1 3  & & $-3$
\end{tabular}
\hskip 2.5em
\begin{tabular}{l l l r}
& 0 4 2 5 3 6 3 4 1 6  & & $-3$ \\
& 0 4 2 5 3 6 1 7 3 4  & & $-3$ \\
& 0 4 2 5 3 6 1 4 3 6  & & $3$ \\
& 0 4 2 5 3 6 3 7 1 4  & & $3$ \\
& 0 4 5 6 1 3 2 3 5 6  & & $-3$ \\
& 0 4 5 6 2 3 5 7 1 3  & & $-3$ \\
& 0 4 5 6 2 3 3 5 1 6  & & $-3$ \\
& 0 4 5 6 1 7 2 3 3 6  & & $3$ \\
& 0 4 5 6 1 6 2 3 3 5  & & $-3$ \\
& 0 4 5 6 2 3 3 7 1 5  & & $3$ \\
& 0 4 2 5 1 3 3 4 5 6  & & $3$ \\
& 0 4 2 5 6 7 1 3 3 4  & & $3$ \\
& 0 4 2 5 3 6 3 4 1 5  & & $-3$ \\
& 0 4 2 5 1 6 3 7 3 4  & & $-3$ \\
& 0 4 2 5 1 6 3 4 3 5  & & $3$ \\
& 0 4 2 5 3 6 1 7 3 4  & & $3$ \\
& 0 4 1 5 2 3 3 5 4 6  & & $3$ \\
& 0 4 5 6 1 7 3 7 2 3  & & $3$ \\
& 0 4 5 6 2 3 3 5 1 4  & & $-3$ \\
& 0 4 1 5 6 7 2 3 3 6  & & $-3$ \\
& 0 4 1 5 3 6 2 3 4 6  & & $-3$ \\
& 0 4 5 6 1 7 2 3 3 6  & & $-3$ \\
& 0 4 1 5 2 3 3 4 5 6  & & $-3$ \\
& 0 4 1 5 6 7 2 3 3 4  & & $-3$ \\
& 0 4 1 5 3 6 3 4 2 5  & & $3$ \\
& 0 4 1 5 2 6 3 7 3 4  & & $3$ \\
& 0 4 1 5 2 6 3 4 3 5  & & $-3$ \\
& 0 4 1 5 3 6 2 7 3 4  & & $-3$ \\
& 0 4 2 5 1 3 3 5 4 6  & & $-3$ \\
& 0 4 5 6 2 7 3 7 1 3  & & $-3$ \\
& 0 4 5 6 1 3 3 5 2 4  & & $3$ \\
& 0 4 2 5 6 7 1 3 3 6  & & $3$ \\
& 0 4 2 5 3 6 1 3 4 6  & & $3$ \\
& 0 4 5 6 1 3 2 7 3 5  & & $-3$ \\
& 0 1 2 3 3 4 3 5 4 6  & & $3$ \\
$\clubsuit_{1}$ & 0 1 2 3 3 6 3 7 4 5  & & $3$ \\
& 0 1 2 5 3 6 3 7 3 5  & & $-3$ \\
& 0 1 2 5 3 6 3 7 3 4  & & $-3$ \\
& 0 1 2 5 3 6 3 4 3 4  & & $3$ \\
& 0 1 2 5 3 6 3 7 3 4  & & $3$ \\
& 0 4 1 5 3 6 2 3 4 6  & & $3$ \\
& 0 4 1 5 3 6 4 7 2 3  & & $3$ \\
& 0 4 1 5 3 6 3 4 2 6  & & $3$ \\
& 0 4 1 5 3 6 2 7 3 4  & & $3$ \\
& 0 4 1 5 3 6 2 4 3 6  & & $-3$ \\
& 0 4 1 5 3 6 3 7 2 4  & & $-3$ \\
& 0 4 5 6 1 3 2 5 3 6  & & $-3$ \\
& 0 4 5 6 1 3 3 7 2 5  & & $-3$ \\
& 0 4 5 6 1 3 3 5 2 6  & & $3$ \\
& 0 4 5 6 1 3 2 7 3 5  & & $3$ \\
& 0 4 5 6 1 3 2 3 5 6  & & $3$ \\
& 0 4 5 6 1 3 5 7 2 3  & & $3$ \\
& 0 1 2 3 3 4 3 4 5 6  & & $-3$ \\
& 0 1 2 3 3 4 3 7 4 5  & & $3$ \\
& 0 1 2 3 3 4 3 5 4 6  & & $-3$ \\
& 0 2 1 3 3 4 3 4 5 6  & & $3$ \\
& 0 2 1 3 3 4 3 7 4 5  & & $-3$ \\
& 0 2 1 3 3 4 3 5 4 6  & & $3$ \\
& 0 4 1 2 3 4 3 4 5 6  & & $-3$ \\
& 0 4 1 2 3 4 3 7 4 5  & & $3$ \\
& 0 4 1 2 3 4 3 5 4 6  & & $-3$ \\
& 0 4 1 5 2 3 3 4 5 6  & & $3$ \\
& 0 4 1 5 2 3 3 7 4 5  & & $3$ \\
& 0 4 1 5 2 3 3 5 4 6  & & $-3$ \\
& 0 4 2 5 1 3 3 4 5 6  & & $-3$ \\
& 0 4 2 5 1 3 3 7 4 5  & & $-3$ \\
& 0 4 2 5 1 3 3 5 4 6  & & $3$
\end{tabular}
\end{table}

\normalsize\renewcommand{\baselinestretch}{1.00}

\section{Perturbation method}\label{AppPerturb}
\noindent%
In section \ref{SecProof} above, the run-through method gave all the terms at once in the operator~$\Diamond$ that establishes 
the factorization $\schouten{\cP,\cQ_{1:6}}=\Diamond(\cP,\Jac(\cP))$. At the same time, there is another method to find $\Diamond$; the operator $\Diamond$ is then constructed gradually, term after term in~ \eqref{EqSol}, by starting with a zero initial approximation for~$\Diamond$. This is the perturbation scheme which we now outline.
(In fact, the perturbation method was tried first, revealing the typical graph patterns and their topological complexity.) 

The difficulty 
is that because the condition $\schouten{\cP,\cQ_{1:6}}=0$ and the Jacobi identity $\schouten{\cP,\cP}=0$ are valid, it is impossible to factorize one through the other; 
both are invisible.
So, we first make both expressions visible by perturbing the Poisson bi-vector 
$\cP \mapsto \cP_{\epsilon}=\cP+\epsilon\Delta$ in such a way that 
the tri\/-\/vector $\schouten{\cP_\epsilon,\cQ_{1:6}(\cP_{\epsilon})}$ and the Jacobiator $\schouten{\cP_{\epsilon}, \cP_{\epsilon}}$ stop vanishing identically:
\[
\schouten{\cP_\epsilon,\cQ_{1:6}(\cP_{\epsilon})} \neq 0 \quad\text{ and }\quad \schouten{\cP_{\epsilon}, \cP_{\epsilon}} \neq 0.
\]
To begin with, put $\Diamond\mathrel{{:}{=}}0$. Now consider 
a class of Poisson brackets on $\mathbb{R}^{3}$ (cf.~\cite{Perelomov}) by using the pre-factor $f(x,y,z)$ and arbitrary function $g(x,y,z)$ in the formula 
\[
\{u,v\}_{\cP}=f \cdot \det \left(\frac{\partial (g,u,v)}{\partial (x,y,z)} \right);
\]
it is helpful to start with some very degenerate dependencies of~$f$ and~$g$ of their arguments (see~\cite{tetra16} and~\cite{Van}).
The next step is to perturb the coefficients of the Poisson bracket $\{\cdot,\cdot\}_{\cP}$ at hand; in a similar way, one starts with degenerate dependency of the perturbation~$\Delta$. The idea is to take perturbations which destroy the validity of Jacobi identity for $\cP_{\epsilon}$ in the linear approximation in the deformation parameter~$\epsilon$. It is readily seen that the expansion of~\eqref{EqFactor} in~$\epsilon$ yields the equality
\[
\schouten{\cP_\epsilon,\cQ_{1:6}}(\epsilon) = (\Diamond
+\bar{o}(1))\,
(\schouten{\cP_{\epsilon},\cP_{\epsilon}}) = 2\epsilon\cdot (\Diamond 
+ \bar{o}(1))\, 
(\schouten{\cP
,\Delta
}
) +
(\Diamond + \bar{o}(1))\,
(\schouten{\cP,\cP}
) + \bar{o}(\epsilon). 
\]
Knowing the left\/-\/hand side at first order in~$\epsilon$ and taking into account that $\lshad\cP,\cP\rshad\equiv0$ for the Poisson bi\/-\/vector~$\cP$ which we perturb by~$\Delta$, we reconstruct the operator~$\Diamond$ that now acts on the known tri\/-\/vector~$2\lshad\cP,\Delta\rshad$.
In this sense, the Jacobiator~$\lshad\cP,\cP\rshad$ shows up through the term~$\lshad\cP,\Delta\rshad$.

For each pair $(\cP,\Delta)$, the above balance at~$\epsilon^1$ 
contains sums over indexes that mark the derivatives falling on the Jacobiator.
By taking those formulae, 
we guess the candidates for graphs that form the next, yet unknown, part of the operator $\Diamond$. Specifically, we inspect which differential operator(s), acting on the Jacobi identity, become visible and we list the graphs that provide such differential operators via the Leibniz rule(s). 
For a while we keep every such 
candidate with an undetermined coefficient.
By repeating the iteration, now for a different 
Poisson bi-vector~$\cP$ or its new, less degenerate perturbation~$\Delta$, we obtain linear constraints for the already introduced undetermined coefficients. Simultaneously, we continue listing the new candidates and introducing new coefficients for~them.

\begin{rem}
By translating formulae into 
graphs, we convert the dimension\/-\/dependent expressions into the dimension\/-\/independent operators which are encoded by the graphs. An obvious drawback of the method which is outlined here is that, presumably, some parts of the operator $\Diamond$ could always stay invisible for all Poisson structures over~$\mathbb{R}^{3}$ 
if they show up only in the higher dimensions. Secondly, the number of variants to consider and in practice, the number of irrelevant terms, each having its own undetermined coefficient, grows exponentially at the initial stage of the reasoning.
\end{rem}

By following the loops of iterations of this algorithm, we managed to find two non-zero coefficients and five zero coefficients in solution~\eqref{EqSol}. Namely, we identified the coefficient $\pm 1$ for the tripod, which is the first term in~\eqref{EqSol}, and we also recognized the coefficient $\pm 3$ of the sum of `elephant' graphs, which is the second to last term in~\eqref{EqSol}.

\begin{rem}
Because of the known skew-symmetry of the tri-vector $\schouten{\cP,\cQ_{1:6}}$ with respect to its arguments $f,g,h$, finding one term in a sum within formula~\eqref{EqSol} for~$\Diamond$ means that the entire such sum is 
reconstructed. Indeed, one then takes the sum over a subgroup of $S_{3}$ acting on $f,g,h$, depending on the actual skew-symmetry of the term which has been found. 

For instance, the first term in \eqref{EqSol}, itself making a sum running over $\{\text{id}\}\prec S_{3}$, is obviously totally antisymmetric with respect to its arguments. 
The other graph which we found by using the perturbation method 
(see the last graph in the second line of formula~\eqref{EqSol} on p.~\pageref{EqSol})
is skew\/-\/symmetric with respect to its second and third arguments 
but it is not yet totally skew\/-\/symmetric with respect to the full set of its arguments. This shows that is suffices to take the sum over the group ${\circlearrowright}=A_{3} 
\prec S_{3}$ of cyclic permutations of~$f,g,h$, thus reconstructing the sixth term in solution~\eqref{EqSol}. 
\end{rem}

\subsubsection*{Acknowledgements}
A.\,K. thanks M.\,Kon\-tse\-vich for posing the problem;
the authors are grateful to P.\,Vanhaecke and A.\,G.\,Sergeev for stimulating discussions.
The authors are profoundly grateful to the referee for constructive criticism and advice.
	
This research 
was supported in part by 
JBI~RUG project~106552 (Groningen
) and by the $\smash{\text{IH\'ES}}$ and MPIM (Bonn), to which A.\,K.\ is grateful for warm hospitality. 
A.\,B.\ and R.\,B.\ thank the organizers of the $8^{\rm th}$ international workshop GADEIS\- VIII\- on Group Analysis of Differential Equations and Integrable Systems (12--16 June 2016, Larnaca, Cyprus) for partial financial support and warm 
hospitality.
A.\,B.\ and R.\,B.\ are also grateful to the Graduate School of Science (Faculty of Mathematics and Natural Sciences, University of Groningen) for financial support.
We thank the Center for Information Technology of the University of Groningen 
for providing access to the \textsf{Peregrine} high performance computing cluster.

\normalsize
\newpage
\section{The condition $a:b=1:6$ is necessary (and maybe sufficient\,?)}
\label{ProofProp1}
\begin{proposition}[\cite{tetra16}]\label{Prop1}
The tetrahedral flow $\dot{P}=\cQ_{a:b}(\cP)$ preserves the property of $\cP+\varepsilon \cQ_{a:b}(\cP)+ \bar{o}(\varepsilon)$ to be (at least infinitesimally)
Poisson for all Poisson bi\/-\/vectors~$\cP$ on all 
affine real manifolds~$N^{n}$ \emph{only~if} the ratio is~$a:b=1:6$.
\end{proposition}

\noindent%
Our proof amounts to producing at least one counterexample when any ratio other than $1:6$ violates equation~\eqref{EqWhichMechanism} for a given Poisson bi-vector $\cP$. 
\begin{proof}
Let $x,y,z$ be the Cartesian coordinates on $\mathbb{R}^{3}$. Consider the Poisson bracket $\{u,v\}_{\cP}=x\cdot \det \bigl( 
{\partial (xyz+y,u,v)}\bigr/{\partial (x,y,z)} \bigr)$ given by the Jacobian, so that the coefficient matrix is
\[\cP^{ij}=
\left( \begin{smallmatrix} 
0&x^{2}y&-x(xz+1) \\ \noalign{\medskip}
-x^{2}y&0&xyz\\ \noalign{\medskip}
-x(xz+1) &-xyz&0 
\end{smallmatrix} \right).
\] 
The coefficient matrices of both bi\/-\/vectors are
\[
\Gamma_{1}(\cP)=6 \cdot
\left( \begin{smallmatrix} 
0&-x^{5}y&-x^{4}(xz+1)\\ \noalign{\medskip}
x^{5}y&0&-x^{3}y\\ \noalign{\medskip}
x^{4}(xz+1)&x^{3}y&0
\end{smallmatrix} \right), \qquad 
\Gamma_{2}(\cP)=
\left( \begin{smallmatrix} 
0&x^{5}y&x^{4}(xz+2)\\ \noalign{\medskip}
-x^{5}y&0&-2x^{3}y\\  \noalign{\medskip}
-x^{4}(xz+2)&2x^{3}y&0
\end{smallmatrix} \right).
\]
It is readily seen that no non\/-\/trivial linear combination $a\cdot\Gamma_{1}(\cP)+b\cdot\Gamma_{2}(\cP)$ of the two flows vanishes everywhere on $\mathbb{R}^{3}\ni (x,y,z)$ for this example. Acting on the bi\/-\/vectors~$\Gamma_{1}$ and~$\Gamma_{2}$ by the Poisson differential $\schouten{\cP,\cdot}$, we obtain two 
tri\/-\/vectors which are completely determined by one component each. Namely, we have that
\[
\schouten{\cP,\Gamma_{1}(\cP)}^{123}=36x^{6}yz+48x^{5}y,\qquad
\schouten{\cP,\Gamma_{2}(\cP)}^{123}=-6x^{6}yz-8x^{5}y.
\]
Clearly, the balance $a:b=1:6$ is the only ratio at which the non-trivial linear combination $\cQ_{a:b}(\cP)=a \cdot \Gamma_{1}(\cP)+b\cdot\Gamma_{2}(\cP)$ solves the equation $\schouten{\cP,\cQ_{a:b}(\cP)}\equiv 0.$
\end{proof}

In fact, more is known~--- this time, about the sufficiency of the condition~$a:b=1:6$. 
%
First, let us recall from~\cite{Perelomov} that on $\mathbb{R}^{3}$ with coordinates~$x$,\ $y$,\ and~$z$ 
there is a class of Poisson brackets that admit first integrals at least locally:\footnote{%
The referee points out that not all the Poisson brackets are given by the Jacobian determinants. Indeed, the function~$g$ in~\eqref{EqPoisson3D} is always a Casimir of such bracket, but there are real Poisson structures on~$\BBR^3$ which do not have (smooth) Casimirs near all of its points: some point(s) can be singular so that in no neighbourhood of it would a Casimir exist. In fact, no exhaustive description is known for Poisson brackets on~$\BBR^3$.} 
\begin{equation}\label{EqPoisson3D}
\{u,v\}_{\cP}=f \cdot \det \left( \frac{\partial(g,u,v)}{\partial(x,y,z)} \right) \quad \text{for } u,v \in C^\infty(\mathbb{R}^3),
\end{equation}
where the free parameter $g$ is a function and the parameter $f$ is a density so that 
\[
f(x,y,z) \cdot \det \left( \frac{\partial(g,u,v)}{\partial(x,y,z)} \right) \mathrm{d}x\mathrm{d}y\mathrm{d}z = f(x,y,z)\raisebox{-10pt}[0pt][15pt]{$\Biggl{|}$}_{\substack{x=x(x',y',z')\\y=y(x',y',z')\\z=z(x',y',z')}} \cdot \det \left( \frac{\partial(g,u,v)}{\partial(x',y',z')} \right) \mathrm{d}x'\mathrm{d}y'\mathrm{d}z'.
\]
In any given coordinate system the parameter $f$ can be chosen freely; then it is recalculated as shown above.

\begin{proposition}[$\mathbb{R}^{3}$,$\{\cdot,\cdot\}_{\cP}$]\label{Prop3D}
The tetrahedral flow $\dot{\cP}=\cQ_{1:6}(\cP)$ does preserve the property of $\cP+\varepsilon \cQ_{a:b}(\cP)+ \bar{o}(\varepsilon)$ to be infinitesimally Poisson for all Poisson structures~\eqref{EqPoisson3D} on $\mathbb{R}^{3}$.
\end{proposition}

We used Proposition~\ref{Prop3D} 
as an heuristic motivation to our 
main Theorem~\ref{ThMain} 
in which the claim from Proposition~\ref{Prop3D} is extended to \emph{all} Poisson structures on all finite-dimensional affine real manifolds. Therefore, in hindsight, Proposition~\ref{Prop3D} above has been proven rigorously as soon as Theorem~\ref{ThMain} was established. 
%

To verify the claim in Proposition~\ref{Prop3D} by direct calculation, it would take years for man still only a few seconds for a computer.\footnote{Running the script below took us approximately $5$~seconds.} A computer-assisted proof of Proposition~\ref{Prop3D} is realized through running the script in Maple (see below).
(All computations are done with the coefficient matrices of bi-vectors at hand. The bi-vectors are computed by using working formulas~\eqref{EqFlow1} and~\eqref{EqFlow2}.) For the balanced flow we have:
\begin{verbatim}
FlowQ := proc (P, y, a, b) 
description "Eval flow Q_a:b of q-dim bi-vector P."; 
local i, j, q, A, F, G, B, T, C; 
q := op(P)[1]; 
F := proc (i, j, k, l, m, n, p, r) options operator, arrow; 
a*(diff(P[i, j], y[k], y[l], y[m]))*(diff(P[k, n], y[p]))
*(diff(P[l, p], y[r]))*(diff(P[m, r], y[n])) end proc; 
G := proc (i, j, k, l, m, n, p, r) options operator, arrow; 
b*(diff(P[i, j], y[k], y[l]))*(diff(P[k, m], y[n], y[p]))
*(diff(P[n, l], y[r]))*(diff(P[r, p], y[j])) end proc; 
B := Array(1 .. q, 1 .. q); 
T := combinat:-cartprod([seq([seq(1 .. q)], i = 1 .. 8)]); 
while not T[finished] do 
C := op(T[nextvalue]()); 
B[C[1], C[2]] := B[C[1], C[2]]+F(C); 
B[C[1], C[5]] := B[C[1], C[5]]+G(C);
end do; 
A := Array(1 .. q, 1 .. q);
for i from 1 to q do 
for j from 1 to q do 
A[i, j] := simplify((1/2)*B[i, j]-(1/2)*B[j, i]);
end do;
end do; 
Matrix(A);
end proc:
\end{verbatim}
To implement the Schouten bracket of two bi-vectors $A$ and $B$, we use a component expansion (cf. \cite{Dubrovin}):
$$
\schouten{A,B}^{ijk}=\sum\nolimits_{s=1}^{n}A^{sk}B^{ij}_{s}+B^{sk}A^{ij}_{s}+A^{sj}B^{ki}_{s}+B^{sj}A^{ki}_{s}+A^{si}B^{jk}_{s}+B^{si}A^{jk}_{s},
$$
where superscripts and subscripts denote the bi-vector components and partial derivatives with respect to the coordinates $y^{s}$, respectively.
\begin{verbatim}
SchoutenBracket := proc (A, B, y) 
description "Evaluate the Schouten-bracket of A and B."; 
local T, t, F, n, res, cnt; 
n := op(A)[1]; 
F := proc (i, j, k) options operator, arrow; 
A[s, k]*(diff(B[i, j], y[s]))+B[s, k]*(diff(A[i, j], y[s]))+
A[s, j]*(diff(B[k, i], y[s]))+B[s, j]*(diff(A[k, i], y[s]))+
A[s, i]*(diff(B[j, k], y[s]))+B[s, i]*(diff(A[j, k], y[s])) end proc; 
T := combinat:-choose(n, 3); 
for t in T do 
print([[t[1], t[2], t[3]],simplify(add(F(t[1], t[2], t[3]), s = 1 .. n))]);
end do; 
end proc:
\end{verbatim}
Finally, the following script provides a computer-assisted proof of Proposision~\ref{Prop3D}.
\begin{verbatim}
# All 3-dimensional Poisson bi-vectors are of the following form.
> P:=<<0,-f(x,y,z)*(diff(g(x,y,z),z)),f(x,y,z)*(diff(g(x,y,z),y))>|
      <f(x,y,z)*(diff(g(x,y,z),z)),0,-f(x,y,z)*(diff(g(x,y,z),x))>|
      <-f(x,y,z)*(diff(g(x,y,z),y)),f(x,y,z)*(diff(g(x,y,z),x)),0>>:
# We evaluate the balanced flow Q_{1:6} on the above bi-vector.
> Q:=FlowQ(P,{x,y,z},1,6)
                [Length of output exceeds limit of 1000000]
# If so, let us inspect whether the flow Q_{1:6} vanishes.
> LinearAlgebra:-Equal(Q,Matrix(1..3,1..3,0))
                                  false
# Still, let us act on this Q_{1:6} by the Poisson differential.
> SchoutenBracket(P,Q,{x,y,z})
                               [[1,2,3], 0]
\end{verbatim}
This 
reasoning hints us that the condition $a:b=1:6$ could be sufficient for equation~\eqref{EqWhichMechanism} to hold for all Poisson structures on all finite dimensional affine real manifolds. A rigorous proof of the respective claim in Theorem~\ref{ThMain} is provided in section~\ref{SecProof}.

\section{The count of Leibniz graphs in Fig.~\protect\ref{FigSomeLeibniz}}\label{pCountLinear}
\label{AppCount}
\noindent%
We count all possible differential consequences of the Jacobi identity, that is, we consider the differential operators acting on the Jacobiator. We do this by constructing all possible graphs that encode 
trivector\/-\/valued differential consequences (see Lemma~\ref{Lemma} on p.~\pageref{Lemma}). The graphs that encode such differential consequences 
%
have $3$ ground vertices.
%
%
The Schouten bracket $\schouten{\cP,\cQ_{1:6}(\cP)}$ consists of graphs with $5$ internal vertices. Since two of these internal vertices are accounted for by 
the Jacobi identity, 
there remain $3$~spare 
internal vertices. 
%

\smallskip
First, let the Jacobiator stand, with all its three edges, on the $3$ ground vertices.
The only freedom that remains is how the $3$ free internal vertices act on each other and on the Jacobiator.
With its first edge, every free internal vertex can act on itself, on its $2$ neighbouring free vertices, or on the Jacobiator; there are $4$ possible targets. 
No second edge can meet the first edge at the same target (as this would yield no contribution due to the anti-symmetry, which is explained in Remark~\ref{RemGamma0}).
Hence there are only $3$ possible targets for this second edge. Finally, again due to anti-symmetry, every possibility is constructed exactly twice this way. Swapping the targets of the first and second edge only contributes to the sign of the graph. The total number of this type of differential consequence is therefore $\left( \frac{4 \cdot 3}{2} \right)^{3}=216$ graphs. This type of graph  is drawn first from the top-left in Figure~\ref{FigSomeLeibniz}.

Now let the Jacobiator stand on only $2$ of the ground vertices. The remaining edge of the Jacobiator has only $3$ possible targets, as the third edge cannot fall back onto the Jacobiator itself. One of the free internal vertices acts with an edge on the remaining ground vertex. The other edge has $4$ candidates as its target, namely the vertex itself, the neighbouring $2$ free internal vertices, and the Jacobiator. The $2$ internal vertices not falling on a ground vertex have each $\frac{4\cdot 3}{2}$ possible targets. The total number of graphs is therefore equal to $3 \cdot 4 \cdot \left( \frac{4\cdot 3}{2} \right)^{2}=432$. This type of graph is the second from the top-left in Figure~\ref{FigSomeLeibniz}.

Next, let the Jacobiator stand on only $1$ ground vertex. 
We distinguish between two cases: namely, the case where $1$ free internal vertex stands on both the remaining ground vertices and the case where two different internal vertices act 
by one edge each on the remaining two ground vertices. 
These are the third and fourth graphs from the top-left in Figure~\ref{FigSomeLeibniz}, respectively.\\[1pt]
$\bullet$\quad In the first case, the remaining $2$ internal vertices each have $\frac{4\cdot 3}{2}$ possible targets. The Jacobiator must act with its two remaining free edges on two different targets out of the $3$ available, yielding $3$ possibilities. 
The number of graphs in the first case is $3 \cdot \left( \frac{4\cdot 3}{2} \right) ^{2}=108$.\\ 
$\bullet$\quad For the second case, two internal vertices can each act on themselves, on the neighbouring $2$ internal vertices, or on the Jacobiator. With two of its edges,
the Jacobiator can act in $3$ different ways on the $3$ 
internal vertices. The third 
internal vertex has $\frac{4\cdot 3}{2}$ possible targets. This brings the total number of graphs for the second case to $4 \cdot 4 \cdot \frac{4\cdot 3}{2}\cdot 3=288$.

The last case to consider is where the Jacobiator does not act on any of the ground vertices. Again, since the outgoing edges of the Jacobiator must have different targets, 
it is clear that the Jacobiator acts in a unique way on all $3$ internal vertices.
We now distinguish two cases: namely, the case where $1$ free internal vertex stands on $2$ ground vertices, $1$ free internal vertex acts on $1$ ground vertex, and $1$ free internal vertex falling on no ground vertex, and the second case where each internal vertex acts with one edge on one ground vertex. These two cases are represented by the last $2$ graphs in Figure~\ref{FigSomeLeibniz}, respectively.\\[1pt]
$\bullet$\quad In the first case, there is a free internal vertex with one free edge, which has $4$ possible targets. The remaining free internal vertex with two free edges has $\frac{4\cdot 3}{2}$ possible targets. The total number of graphs for this case is $4 \cdot \frac{4\cdot 3}{2}=24$.\\ 
$\bullet$\quad In the second case, each 
internal vertex can act on itself, on its $2$ neighbouring internal vertices, and on the Jacobiator. This results in a total of $4^{3}=64$ graphs.

Summarizing, the total number of all trivector\/-\/valued Leibniz graphs, linear in the Jacobiator and containing five internal vertices
, is~$1132$.

\section{Properties of the found solution}\label{SecProperties}
\begin{rem}\label{RemOverDetermined}
Let us recall that equation~\eqref{EqWhichMechanism} yields the linear system of 7,025~inhomogeneous equations for the coefficients of 1132~patterns from Fig.~\ref{FigSomeLeibniz}. This shows that the algebraic system at hand is extremely overdetermined.
Moreover, out of those 1132~admissible totally antisymmetric graphs, solution~\eqref{EqSol} involves only 8~of them. In this sense, the factorising operator~$\Diamond$ in~\eqref{EqWhichMechanism} is special; for it expands via~\eqref{EqSol} over a very low dimensional affine subspace in the affine space of unknowns in that inhomogeneous linear algebraic system.
\end{rem}

\begin{property}
The relevant 
Leibniz graphs, with respect to which the solution $\Diamond(\cP,\,\cdot\,)$ expands,
do not contain tadpoles nor two-cycles (or ``eyes'', see Fig.~\ref{FigTadpoleEye} 
on p.~\pageref{FigTadpoleEye}).

\noindent%
$\bullet$\quad None of the arrows that act back on the Jacobiator is issued 
from any of its 
arguments.

\noindent%
$\bullet$\quad In all the graphs the source vertices (if any), on which no arrows fall after all the Leibniz rules are expanded,
belong to the Jacobiator (cf.~\eqref{EqJacFig} on p.~\pageref{EqJacFig}).
\end{property}


\begin{property}
The found solution $\Diamond$ does contain the graphs in which two or three arrows fall on the Jacobiator.\footnote{For instance, the first term in $\Diamond$ is the tripod standing on $\Jac(\cP)$.}%
\end{property}

It has been explained in~\cite{gvbv,dq15} that the existence of two or more such arrows falling on the equation
$\lshad\cP,\cP\rshad=0$ is an obstruction to an extension of the main claim,
\begin{equation}
\lshad\cP,\cQ_{1:6}(\cP)\rshad\doteq 0\quad \text{via} \ \lshad\cP,\cP\rshad=0,
\tag{\ref{EqWhichMechanism}}
\end{equation}	
to the infinite\/-\/dimensional geometry of jet spaces $J^{\infty}(\pi)$ for affine bundles over a manifold~$M^m$ or jet spaces 
$J^{\infty}(M^m\to N^n)$ of maps from~$M^m$, and of variational Poisson brackets $\{\,,\,\}_{\bcP}$ for functionals on such
jet spaces (see~\cite{Olver,TwelveLect} and~\cite{cycle14,dq15}). Namely, it can then be that
\begin{equation}\label{EqWhichMechanismVar}
\lshad\bcP,\bcQ_{1:6}(\bcP)\rshad\ncong0\quad\text{although }\ \lshad\bcP,\bcP\rshad\cong0.
\end{equation}
We denote here by $\lshad\,,\,\rshad$ the variational Schouten bracket; the variational bi\/-\/vector~$\bcQ_{1:6}$ is 
constructed from the variational Poisson bi\/-\/vector~$\bcP$ by using techniques from the geometry of iterated variations
of functionals (see~\cite{gvbv,cycle14,dq15}). An explicit counterexample of~\eqref{EqWhichMechanismVar} is known from
\cite{tetra16} for the variational Poisson structure of the Harry Dym partial differential equation. 

The reason why the obstruction arises is that in the variational setting, the second and higher order variations of a 
trivial integral functional $\Jac(\bcP)\cong0$ in the horizontal cohomology can still be nonzero (although its first variation would of course vanish).%
\footnote{The same effect has been foreseen for a variational lift of deformation quantisation~\cite{KontsevichFormality}:
it has been argued in~\cite{dq15} why the associativity of noncommutative 
star\/-\/product
$\star=\times+\hbar\{\,\cdot\,,\,\cdot\,\}_{\bcP}+\bar{o}(\hbar)$
can leak and it has been shown in~\cite{sqs15} that if it actually does at~$O(\hbar^k)$, the order~$k$ at which this leak
of associativity can occur is high:~$k\geqslant4$.} 

\begin{rem}
The eight graphs in \eqref{EqSol}
represent a \emph{linear} differential operator with respect to the Jacobiator $\Jac(\cP)$. However, a quadratic 
nonlinearity with respect to the two\/-\/vertex argument $\Jac(\cP)$ could 
be hidden in the five\/-\/vertex graphs in formula~\eqref{EqSol}, so that it would in fact encode a bi\/-\/differential operator~$\Diamond(\cP,\,{\cdot}\,,\,{\cdot}\,)$.
If this be the case, expansion of one or the other copy of the Jacobiator 
using~\eqref{EqJacFig} in such a polydifferential operator $\Diamond(\cP,\,{\cdot}\,,\,{\cdot}\,)$ would produce two seemingly distinct linear differential operators~$\Diamond(\cP,\,{\cdot}\,)$.
\end{rem}

The scenarios to build the bi\/-\/linear, bi\/-\/differential terms in the 
operator~$\Diamond$ are drawn in Fig.~\ref{FigJacJac} 
below.\label{pCountQuadratic} 
\begin{figure}[htb]
\begin{center}
\unitlength=1mm
\special{em:linewidth 0.4pt}
\linethickness{0.4pt}
\begin{picture}(50.00,30.67)
\put(0.00,30.00){\makebox(0,0)[ct]{``3":}}
\put(10.00,30.00){\circle*{1.33}}
\put(17.00,25.00){\framebox(13.00,5.00)[cc]{$\bullet\ \bullet$}}
\put(37.00,25.00){\framebox(13.00,5.00)[cc]{$\bullet\ \bullet$}}
\bezier{16}(10.00,30.00)(8.00,28.00)(7.00,27.00)
\bezier{16}(10.00,30.00)(12.00,28.00)(13.00,27.00)
\bezier{16}(20.00,25.00)(18.00,23.00)(17.00,22.00)
\bezier{12}(23.67,25.00)(23.67,23.00)(23.67,22.00)
\bezier{16}(27.33,25.00)(29.33,23.00)(30.33,22.00)
\put(39.00,25.00){\vector(-1,-1){10.00}}
\put(43.67,25.00){\vector(-1,-2){5.00}}
\put(48.00,25.00){\vector(0,-1){10.00}}
\put(29.00,10.00){\makebox(0,0)[cb]{\tiny(\ )}}
\put(38.67,10.00){\makebox(0,0)[cb]{\tiny(\ )}}
\put(48.00,10.00){\makebox(0,0)[cb]{\tiny(\ )}}
\end{picture}
\qquad
\unitlength=1mm
\special{em:linewidth 0.4pt}
\linethickness{0.4pt}
\begin{picture}(50.00,30.67)
\put(0.00,30.00){\makebox(0,0)[ct]{``2"(1):}}
\put(10.00,30.00){\circle*{1.33}}
\put(17.00,25.00){\framebox(13.00,5.00)[cc]{$\bullet\ \bullet$}}
\put(37.00,25.00){\framebox(13.00,5.00)[cc]{$\bullet\ \bullet$}}
\bezier{16}(10.00,30.00)(8.00,28.00)(7.00,27.00)
\bezier{16}(10.00,30.00)(12.00,28.00)(13.00,27.00)
\put(20.00,25.00){\vector(0,-1){10.00}}
\bezier{12}(23.67,25.00)(23.67,23.00)(23.67,22.00)
\bezier{16}(27.33,25.00)(29.33,23.00)(30.33,22.00)
\bezier{16}(39.00,25.00)(37.00,23.00)(36.00,22.00)
\put(43.67,25.00){\vector(-1,-2){5.00}}
\put(48.00,25.00){\vector(0,-1){10.00}}
\put(20.00,10.00){\makebox(0,0)[cb]{\tiny(\ )}}
\put(38.67,10.00){\makebox(0,0)[cb]{\tiny(\ )}}
\put(48.00,10.00){\makebox(0,0)[cb]{\tiny(\ )}}
\end{picture}
\qquad
\unitlength=1mm
\special{em:linewidth 0.4pt}
\linethickness{0.4pt}
\begin{picture}(50.00,30.67)
\put(0.00,30.00){\makebox(0,0)[ct]{``2"(2):}}
\put(10.00,30.00){\circle*{1.33}}
\put(17.00,25.00){\framebox(13.00,5.00)[cc]{$\bullet\ \bullet$}}
\put(37.00,25.00){\framebox(13.00,5.00)[cc]{$\bullet\ \bullet$}}
\bezier{100}(10.00,30.00)(5.00,22.00)(10.00,15.00)
\put(10,15.00){\vector(1,-2){0.33}}
\bezier{16}(10.00,30.00)(12.00,28.00)(13.00,27.00)
\bezier{16}(20.00,25.00)(18.00,23.00)(17.00,22.00)
\bezier{12}(23.67,25.00)(23.67,23.00)(23.67,22.00)
\bezier{16}(27.33,25.00)(29.33,23.00)(30.33,22.00)
\put(39.00,25.00){\vector(-1,-1){10.00}}
\put(43.67,25.00){\vector(-1,-2){5.00}}
\bezier{12}(48,25.00)(48,23.00)(48,22.00)
\put(29.00,10.00){\makebox(0,0)[cb]{\tiny(\ )}}
\put(38.67,10.00){\makebox(0,0)[cb]{\tiny(\ )}}
\put(10.00,10.00){\makebox(0,0)[cb]{\tiny(\ )}}
\end{picture}
\qquad
\unitlength=1mm
\special{em:linewidth 0.4pt}
\linethickness{0.4pt}
\begin{picture}(50.00,30.67)
\put(0.00,30.00){\makebox(0,0)[ct]{``1"(1):}}
\put(10.00,30.00){\circle*{1.33}}
\put(17.00,25.00){\framebox(13.00,5.00)[cc]{$\bullet\ \bullet$}}
\put(37.00,25.00){\framebox(13.00,5.00)[cc]{$\bullet\ \bullet$}}
\bezier{100}(10.00,30.00)(5.00,22.00)(10.00,15.00)
\put(10,15.00){\vector(1,-2){0.33}}
\bezier{16}(10.00,30.00)(12.00,28.00)(13.00,27.00)
\bezier{16}(20.00,25.00)(18.00,23.00)(17.00,22.00)
\bezier{12}(23.67,25.00)(23.67,23.00)(23.67,22.00)
\bezier{16}(39.00,25.00)(37.00,23.00)(36.00,22.00)
\bezier{16}(48,25.00)(50,23.00)(51,22.00)
\put(27.33,25.00){\vector(0,-1){10.00}}
\put(43.67,25.00){\vector(0,-1){10.00}}
\put(28.00,10.00){\makebox(0,0)[cb]{\tiny(\ )}}
\put(44,10.00){\makebox(0,0)[cb]{\tiny(\ )}}
\put(10.00,10.00){\makebox(0,0)[cb]{\tiny(\ )}}
\end{picture}
\qquad
\unitlength=1mm
\special{em:linewidth 0.4pt}
\linethickness{0.4pt}
\begin{picture}(50.00,30.67)
\put(0.00,30.00){\makebox(0,0)[ct]{``1"(2):}}
\put(10.00,30.00){\circle*{1.33}}
\put(17.00,25.00){\framebox(13.00,5.00)[cc]{$\bullet\ \bullet$}}
\put(37.00,25.00){\framebox(13.00,5.00)[cc]{$\bullet\ \bullet$}}
\bezier{100}(10.00,30.00)(5.00,22.00)(10.00,15.00)
\bezier{100}(10.00,30.00)(10.00,22.00)(20.00,15.00)
\put(10,15.00){\vector(1,-2){0.33}}
\put(20,15.00){\vector(1,-1){0.33}}
\bezier{16}(20.00,25.00)(18.00,23.00)(17.00,22.00)
\bezier{12}(23.67,25.00)(23.67,23.00)(23.67,22.00)
\bezier{16}(27.33,25.00)(29.33,23.00)(30.33,22.00)
\put(39.00,25.00){\vector(-1,-1){10.00}}
\bezier{12}(43.67,25.00)(43.67,23.00)(43.67,22.00)
\bezier{16}(48,25.00)(50,23.00)(51,22.00)
\put(29.00,10.00){\makebox(0,0)[cb]{\tiny(\ )}}
\put(20,10.00){\makebox(0,0)[cb]{\tiny(\ )}}
\put(10.00,10.00){\makebox(0,0)[cb]{\tiny(\ )}}
\end{picture}
\end{center}
\caption{
The Leibniz graphs by using which a quadratic --\,with respect to the Jacobiator\,--
part~$\Diamond(\cP,\cdot,\cdot)$ of the factorizing operator could be sought for in~\eqref{EqFactor}; such quadratic part (if any) itself is necessarily totally skew\/-\/symmetric with respect to the three sinks.
}\label{FigJacJac}
\end{figure}
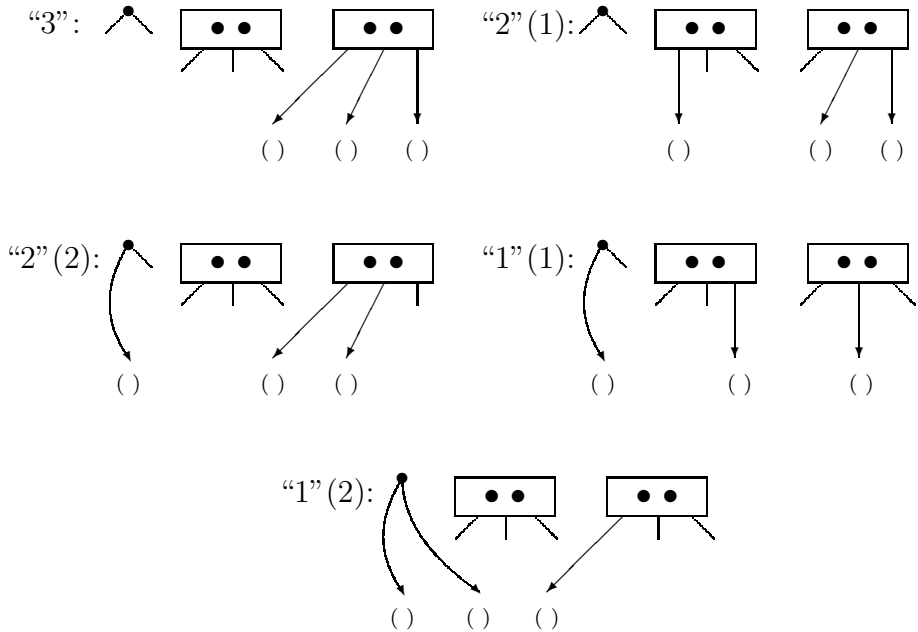
We consider --\,in fact, without any loss of generality\,-- only those eight 
Leibniz graphs in which
\begin{itemize}
\item the three arguments of each copy of Jacobiator~\eqref{EqJacFig} are different; in particular,
\item neither of the Jacobiators acts on the other copy by two or three arrows (so that only none or one such arrow is possible).
\end{itemize}
We recall that known solution~\eqref{EqSol} is the sum of 39~graphs from which a linear dependence on the Jacobiator~$\Jac(\cP)$ is retrieved by using the 27~Leibniz graphs (see Table~\ref{TableSol} on p.~\pageref{TableSol}).
Let us inspect whether any solution of equation~\eqref{EqFactor} 
can be nonlinear in~$\Jac(\cP)$; in particular, let us check whether there is (or is not)
a bi\/-\/linear dependence in~$\Jac(\cP)$ hidden in~\eqref{EqSol}.


%

\begin{proposition}
There is no quadratic part in all the solutions of equation~\eqref{EqFactor}.
\end{proposition}
\noindent
This claim is supported by a computer\/-\/assisted run\/-\/through over all Leibniz graphs with linear and with quadratic dependence on the Jacobiator, combined with a requirement that at least one coefficient of those quadratic (in~$\Jac(\cP)$) Leibniz graphs be nonzero. There is no solution.

\section{Open problems}
\subsection{}
   For the factorisation $\lshad\cP,\cQ_{1:6}(\cP)\rshad=\Diamond\bigl(\cP,\Jac(\cP)\bigr)$ to guarantee that the equality
$\boldsymbol{\dd}_{\cP}\bigl(\cQ_{1:6}(\cP)\bigr)=0$ holds if $\Jac(\cP)=0$, its mechanism is nontrivial. Relying on
Lemma~\ref{Lemma} (see~\cite{sqs15}), it tells us how the differential consequences of Jacobi identity are split into
separately vanishing expressions. This mechanism works not only in the construction of flows that
satisfy~\eqref{EqWhichMechanism} but also in the associativity,
\[
\text{Assoc}_{\cP}(f,g,h)\mathrel{{:}{=}}(f\star g)\star h-f\star(g\star h)\doteq 0 \quad \text{via}\ \lshad\cP,\cP\rshad=0,
\]
of the non-commutative unital star\/-\/product $\star=\times+\hbar\{\,\cdot\,,\,\cdot\,\}_{\cP}+o(\hbar)$. The formula
for $\star$-products was given in~\cite{KontsevichFormality}, establishing the deformation quantisation
 $\times\mapsto\star$ of the usual product $\times$ in the algebra $C^{\infty}(N^n)\ni f,g,h$ on a finite-dimentional
affine Poisson manifold $(N^n,\cP)$, see also~\cite{sqs15,dq15}. In fact, the construction of graph complex and the pictorial 
language of graphs~\cite{Ascona96,KontsevichFormality} that encode polydifferential operators is common to all these deformation procedures 
(cf.~\cite{cpcOrder4}, also~\cite{Merkulov}).

\begin{open}
Consider the Kontsevich star-product $\star=\times+\hbar\{\,\cdot\,,\,\cdot\,\}_{\cP}+o(\hbar)$
in the algebra $C^{\infty}(N^n)[[\hbar]]$ on a finite-dimensional affine Poisson manifold $(N^n,\cP)$. Given by the 
tetrahedra $\Gamma_1$ and $\Gamma_2'$ (see Fig.~\ref{FigTetra} on p.~\pageref{FigTetra}),  the infinitesimal deformation
$\cP\mapsto\cP+\veps\cQ_{1:6}(\cP)+o(\veps)$ induces the infinitesimal deformation
$\star\mapsto\star+\hbar\veps\,\lshad\lshad\cQ_{1:6}(\cP),\,\cdot\,\rshad,\,\cdot\,\rshad+o(\veps)$
of the star\/-\/product. What are the properties of this infinitesimally deformed $\star(\veps)$-\/product\,? In particular, is the
condition that $\cQ_{1:6}(\cP)$ be $\boldsymbol{\dd}_{\cP}$-\/trivial necessary for the $\star(\veps)$-product to be gauge-equivalent to the unperturbed $\star$-\/product at~$\veps=0$\,?
\end{open}

We recall that the theory of (infinitesimal) deformations of associative algebra structures is very well studied in the broadest context (e.g., of the Yang\/--\/Baxter equation, Witten\/--\/Dijkgraaf\/--\/Verlinde\/--\/Verlinde (WDVV) equation, Frobenius manifolds and F-\/structures, etc.), see~\cite{Dubrovin,ManinFrobenius}.
We expect that in that theory's part which is specific to the deformation of associative structures on finite\/-\/dimensional affine Poisson manifolds~$N^n$, there must be a dictionary between the construction of Kontsevich flows for spaces of Poisson bi\/-\/vectors and other instruments to deform the associative product in the algebra~$C^\infty(N^n)$.

\subsection{}
The Kontsevich tetrahedral flow $\dot{\cP}=\cQ_{1:6}(\cP)$ is a universal procedure to deform a given Poisson bi\/-\/vector~$\cP$ on any finite\/-\/dimensional affine real manifold~$N^n$ (i.\,e.\ not necessarily topologically trivial).
For consistency, let us recall that generally speaking, not every infinitesimal deformation $\cP \mapsto \cP+\varepsilon \cQ+\bar{o}(\varepsilon)$ of a Poisson bi\/-\/vector $\cP$ can be completed to a Poisson deformation $\cP \mapsto \cP+\cQ(\varepsilon)$ at all orders in $\varepsilon$. 
The obstructions are contained in the third $\boldsymbol{\partial}_{\cP}$-\/cohomology group $\mathrm{H}^{3}_{\cP}=\bigl\{\mathrm{T}\in\Gamma\bigl(\bigwedge^3 TN\bigr)$ $|$ $\boldsymbol{\dd}_{\cP}(\mathrm{T})=0\bigr\}$
$\bigl/$ $\bigl\{\mathrm{T}=\boldsymbol{\partial}_{\cP}(\mathrm{R})$, 
$\mathrm{R}\in\Gamma\bigl(\bigwedge^2 TN\bigr)\bigr\}$.
Indeed, cast the master\/-\/equation $\schouten{\cP+\cQ(\varepsilon),\cP+\cQ(\varepsilon)}=0$ for the Poisson deformation to the coboundary statement $\schouten{\cQ(\varepsilon),\cQ(\varepsilon)}=\boldsymbol{\partial}_{\cP}(-\cP-2\cQ(\varepsilon))$, 
whence $\boldsymbol{\partial}_{\cP}(\schouten{\cQ(\varepsilon),\cQ(\varepsilon)}\equiv 0$ by $\boldsymbol{\partial}_{\cP}^{2}=0$. Therefore, the vanishing of the third $\boldsymbol{\partial}_{\cP}$-\/cohomology group guarantees the existence of a power series solution $\cQ(\varepsilon)$ to the cocycle\/-\/coboundary equation $\schouten{\cQ(\varepsilon),\cQ(\varepsilon)}=-2 \boldsymbol{\partial}_{\cP}(\cQ(\varepsilon))$: known to be a cocycle, the left\/-\/hand side has been proven to be a coboundary as well. 
(In other words, an infinitesimal deformation $\cP\mapsto\cP+\veps\cQ_{1:6}(\cP)+o(\veps)$ 
can be completed to the construction of Poisson bi-vector~$\cP(\veps)$ such that $\cP(\veps=0)=\cP$ and
$\left.\frac{\Id}{\Id\veps}\right|_{\veps=0}\cP(\veps)=\cQ_{1:6}(\cP)$ if 
the third Poisson cohomology group $\mathrm{H}^3_{\cP}(N^n)$ with respect to the Poisson differential
$\boldsymbol{\dd}_{\cP}=\lshad\cP,\,\cdot\,\rshad$ vanishes for the manifold~$N^n$.)

In the symplectic case, i.\,e.\ for $n$~even and bracket $\{\,\cdot\,,\,\cdot\,\}_{\cP}$ nondegenerate, the Poisson complex is known to be
isomorphic to the de Rham complex for $N^n$ (see \cite{VanhaeckePoisson}).
We are not yet aware of any way to constrain the Poisson cohomology groups $\mathrm{H}^k_{\cP}(N^n)$ for \emph{degenerate} Poisson
brackets $\{\,\cdot\,,\,\cdot\,\}_{\cP}$ on real manifolds $N^n$ of not necessarily even dimension $n<\infty$.
(E.g., the algorithm for construction of cubic Poisson brackets on the basis of a class of $R$-\/matrices, which is explained in~\cite{VanhaeckePoisson}, yields a rank\/-\/six bracket on~$N^9\subset\mathbb{R}^9$.) 

\subsection{} 
The second 
Poisson cohomology group $\mathrm{H}^2_{\cP
}(N^n)$ of the manifold~$N^n$, if nonzero, provides room for the
$\boldsymbol{\dd}_{\cP}$-\/nontrivial deformations of $\cP$ using $\cQ_{1:6}(\cP)$ such that $\cQ_{1:6}(\cP)\ne\lshad\cP,\cX\rshad$
for all globally defined 1-vectors $\cX$ on $N^n$. In particular, this implies that there are no $\boldsymbol{\dd}_{\cP}$-nontrivial
tetrahedral graph flows on even-dimensional star-shaped domains equipped with nondegenerate Poisson brackets. 

A possibility for the right\/-\/hand side~$\cQ_{1:6}(\cP)$ of the tetrahedral flow to be $\boldsymbol{\dd}_{\cP}$-\/trivial is thus a global, 
topological effect; it cannot always be seen within a single chart in~$N^n$. Moreover, it is not universal with respect to
the calculus of graphs.

\begin{rem}
Kontsevich notes~\cite{Ascona96} that if $n=2$ so that every bi\/-\/vector~$\cP$ on~$N^2$ is Poisson and every flow $\dot{\cP}=\cQ_{a:b}(\cP)$ preserves this property, the tetrahedron~$\Gamma_1$ (or, equivalently, the velocity~$\cQ_{1:0}(\cP)$) is always $\dd_{\cP}$-\/exact. 
The required $1$-\/vector field~$\cX(\cP)$ in the coboundary statement $\cQ_{1:0}(\cP)=\lshad\cP,\cX\rshad$ can be expressed in terms of the bi\/-\/vector~$\cP$, e.g., by the Leibniz\/-\/rule graph
$\cX=
\raisebox{0pt}[6mm][4mm]{\unitlength=0.4mm
\special{em:linewidth 0.4pt}
\linethickness{0.4pt}
\begin{picture}(17,24)(5,5)
\put(-5,-7){
\begin{picture}(17.00,24.00)
\put(10.00,10.00){\circle*{1}}
\put(17.00,17.0){\circle*{1}}
\put(3.00,17.0){\circle*{1}}
\put(10.00,10.00){\vector(0,-1){7.30}}
\put(17.00,17.00){\vector(-1,0){14.00}}
\put(3.00,17.00){\vector(1,-1){6.67}}
\bezier{30}(3,17)(6.67,13.67)(9.67,10.33)
%
%
\put(17,17){\vector(-1,-1){6.67}}
\bezier{30}(17,17)(13.67,13.67)(10.33,10.33)
\bezier{52}(17.00,17.00)(16.33,23.33)(10.00,24.00)
\bezier{52}(10.00,24.00)(3.67,23.33)(3.00,17.00)
\put(16.8,18.2){\vector(0,-1){1}}
\put(10,17){\oval(18,18)}
\put(10,10){\line(1,0){10}}
\bezier{52}(20,10)(27,10)(21,16)
\put(21,16){\vector(-1,1){0}}
\end{picture}
}\end{picture}}
$\:. (This is a particular, not general solution.) We recall 
that after the dimension~$n$ is fixed (here $n=2$), a given differential polynomial in~$\cP$ can be encoded by the Kontsevich graphs in non\/-\/unique way (cf.~\cite{LE16} for details).%
\end{rem}

\begin{open}
The formalism developed in~\cite{Ascona96} suggests that there are, most likely, infinitely many 
Kontsevich graph flows on the spaces of Poisson bi-vectors on finite-dimensional affine Poisson manifolds. Forming an 
example $\cQ_{1:6}(\cP)$ of such a cocycle in the graph complex, the tetrahedra $\Gamma_1$ and $\Gamma_2'$ in
Fig.~\ref{FigTetra} are built over four internal vertices. What is or are the next --~with respect to the ordering of 
natural numbers~-- Poisson cohomology\/-\/nontrivial Kontsevich graph cocycle(s) built over five or more internal vertices\,?
\end{open}

\subsection{} 
   The tetrahedral flow $\dot{\cP}=\cQ_{1:6}(\cP)$ preserves the space
$\{\cP\in\Gamma(\bigwedge^2 TN^n)$ $|$ $\lshad\cP,\cP\rshad=0\}$ 
of Poisson bi-vectors; this is guaranteed by Theorem~\ref{ThMain} that asserts $\boldsymbol{\dd}_{\cP}(\cQ_{1:6})\doteq 0$
within the (graded-)commutative geometry of finite\/-\/dimensional affine real manifolds~$N^n$.

\begin{open}
Does the proven property,
\begin{equation}
\lshad\cP,\cQ_{1:6}(\cP)\rshad\doteq 0 \quad \text{via} \ \lshad\cP,\cP\rshad=0,
\tag{\ref{EqWhichMechanism}}
\end{equation}
generalize to the formal noncommutative symplectic supergeometry~\cite{KontsevichCyclic}, to the calculus of multivectors performed by using their necklace brackets (see~\cite{cycle14} and references therein), and to Poisson structures on the
commutative non-associative unital algebras of cyclic words (e.\,g., see~\cite{OlverSokolov1998})\,?
\end{open}

\end{document}